\documentclass[a4paper,10pt]{article}

\usepackage{latexsym}
\usepackage{mathrsfs}
\usepackage{amsfonts,amsmath,amssymb}
\usepackage{indentfirst}
\usepackage{subeqnarray}
\usepackage{verbatim}
\usepackage[pdftex]{graphicx}
\usepackage{epstopdf}
\usepackage{stmaryrd}
\usepackage{multirow}
\usepackage{diagbox}
\usepackage{subcaption}

\newcommand{\bx}{\mbox{\boldmath{$x$}}}

\newcommand{\bb}{\mbox{\boldmath{$b$}}}

\newcommand{\bn}{\mbox{\boldmath{$n$}}}

\newcommand{\bW}{\mbox{\boldmath{$W$}}}

\newcommand{\dx}{\mathrm{d}x}
\newcommand{\ds}{\mathrm{d}s}
\newcommand{\dt}{\mathrm{d}t}

\newtheorem{theorem}{Theorem}[section]
\newtheorem{lemma}[theorem]{Lemma}
\newtheorem{assumption}[theorem]{Assumption}
\newtheorem{corollary}[theorem]{Corollary}

\newtheorem{example}[theorem]{Example}

\newtheorem{remark}[theorem]{Remark}
\numberwithin{equation}{section}

\newtheorem{rem}{Remark}[section]

\newenvironment{proof}[1][Proof]{\textbf{#1.} }
{\ \rule{0.75em}{0.75em}\smallskip}

\usepackage[pdftex,unicode,colorlinks,linkcolor=blue]{hyperref}

\begin{document}

\begin{center}
\Large\bf Local Randomized Neural Networks with Discontinuous Galerkin Methods for Partial Differential Equations
\end{center}

\begin{center}
Jingbo Sun\footnote{School of Mathematics and Statistics, Xi'an Jiaotong University, Xi'an, Shaanxi 710049, P.R. China. The work of this author was partially supported by the National Natural Science Foundation of China (Grant No. 12171383).  E-mail: {\tt jingbosun.xjtu@xjtu.edu.cn}.},
\quad 
Suchuan Dong\footnote{Center for Computational and Applied Mathematics, Department of Mathematics, Purdue University, West Lafayette, Indiana, USA. The work of this author was partially supported by the US National Science Foundation (DMS-2012415). Email: {\tt sdong@purdue.edu}.}
\quad and \quad
Fei Wang\footnote{School of Mathematics and Statistics, Xi’an Jiaotong University, Xi’an, Shaanxi 710049, China. The work of this author was partially supported by the National Natural Science Foundation of China (Grant No. 12171383). Email: {\tt feiwang.xjtu@xjtu.edu.cn}.}
\end{center}

\medskip
\begin{quote}
  {\bf Abstract.} Randomized neural networks (RNN) are a variation of neural networks in which the hidden-layer parameters are fixed to randomly assigned values and the output-layer parameters are obtained by solving a linear system by least squares. This improves the efficiency without degrading the accuracy of the neural network. In this paper, we combine the idea of the local RNN (LRNN) and the discontinuous Galerkin (DG) approach for solving partial differential equations. RNNs are used to approximate the solution on the subdomains, and the DG formulation is used to glue them together. Taking the Poisson problem as a model, we propose three numerical schemes and provide the convergence analyses. Then we extend the ideas to time-dependent problems. Taking the heat equation as a model, three space-time LRNN with DG formulations are proposed. Finally, we present numerical tests to demonstrate the performance of the methods developed herein. We compare the proposed methods with the finite element method and the usual DG method. The LRNN-DG methods can achieve better accuracy under the same degrees of freedom, signifying that this new approach has a great potential for solving partial differential equations.

\end{quote}

\medskip

\section{Introduction}

Artificial neural networks (ANNs) have been successfully applied to solve the problems of segmentation, classification, pattern recognition, automatic control, etc (\cite{Badrinarayanan2017dcedaforiseg, Goodfellow2016deeplearning, Kim2014sentenceclassification, Karpathy2014largeclassi, LeCun2015deeplearning, Lai2015rnnclassi, Shelhamer2016seg, Simonyan2013deepinside, Wang2017resattenforic}). In recent years, many research works based on neural networks have been proposed for solving partial differential equations (PDEs) in light of the excellent approximation capability of neural networks (NNs). Some of these contributions are based on the strong form of partial differential equations. Physical informed neural networks (PINNs) (\cite{Raissi2019PINN}) and the deep Galerkin method (DGM) (\cite{Sirignano2018DGM}) are two representative methods among them. Specifically, PINNs train the neural network by minimizing the mean squared error loss consisting of information about the PDE, boundary conditions, and/or initial conditions on certain collocation points. The loss function of DGM measures the residual of PDE in the sense of integral. Based on PINNs, a number of new models (\cite{Belbute2021HyperPINN, Fang2021hpinn, Liu2021PhysicsAugmented, Leung2021nhpinn, Lyu2020MIM, Pang2020nPINNs, Pang2019fPINNs, Ramabathiran2021SPINN, Wong2021sspacePinn, Zhang2020pinnpde}) have been proposed aiming to improve the performance, and some other studies (\cite{almajid2022mediafluidflow, Irrgang2021earth, Jin2021NSFnets, Jiang2021pinn.nonlineardyn, Ncube2021pinn, Raissi2020Hidden fluid mechanics}) focus on the application of this technique for different kinds of problems.

The solutions to some problems are known to have a low regularity, and these problems cannot be described by PDEs. Therefore, in some investigations on NNs the loss functions are constructed based on a weak formulation, such as the deep Ritz method (\cite{Ee2018deepRitz}), deep Nitsche method (\cite{Liao2019deepRitzboundary}), weak adversarial networks (\cite{Zang2020adversarialnns}), and other methods (\cite{Sheng2021PFNN, Zhang2018}). 
One issue with these neural network-based methods in strong or weak forms is that they generally suffer from a limited accuracy and they are time-consuming. Although the NNs have a strong capability for function approximation, they are hard to train to arrive at the global optimal state due to the lack of an efficient optimization method for the training process. In addition, the computational cost of network training is very high, which hampers practical applications. In terms of the accuracy and efficiency, these neural network-based methods generally cannot compete with traditional methods such as the finite element method (FEM), finite difference method, and finite volume method. 

As an alternative approach to the fully parameterized NN models, randomized neural networks have been proposed (\cite{Igelnik1995RNN1, Igelnik1999RNN2, Pao1992RNN3, Pao1994RNN4}). The parameters of the links between the hidden layers are chosen randomly and then fixed during training, while the parameters for the links between the last hidden layer and output layer are solved by a least-squares method. Extreme learning machine (ELM) (\cite{Huang2006ELMtheorandapp}) is an example of a randomized neural network, and it has been applied to solve various problems (\cite{Frenay2010Svmelm, Huang2011ELMclassi, Huang2010optibasedonELM, Liu2008Extremesupport, Miche2008OP-ELM, Parviainen2010InterpretingELM, Rong2008ELMclassi, Zhu2005ELMrecog}) including successfully solving differential equations (\cite{Balasundaram2011minimizedELM, Dwivedi2020PIELM, Sun2019BELM, Yang2018elmpde}). The feasibility analysis of ELM was proved in \cite{Liu2014ELMfeasible}, which proves that the generalization capability of the ELM is the same as NNs if suitable activation function and initialization strategy of those fixed parameters are given properly. For solving PDEs, in \cite{Dong2021locELM} Dong and Li combine the ideas of local extreme learning machine and domain decomposition (locELM-DD) to solve PDEs, and this approach improves the accuracy and efficiency greatly. However, locELM-DD is based on the strong form of PDE problems, and so it may be unfavorable for problems that can only be described by a weak formulation. In \cite{Shang2022DeepPetrov}, based on ELM and the Petrov-Galerkin formulation, the deep Petrov-Galerkin method is proposed to solve partial differential equations, and the numerical examples show that this approach is accurate and efficient with respect to degrees of freedom (DoF).

In the current work, we focus on the weak formulations of PDEs and their approximations by local neural networks. We combine the local randomized neural networks and the discontinuous Galerkin (DG) approach, and seek to exploit the DG framework to glue the local NNs together. Specifically, we take the Poisson equation and the heat equation as model problems to develop three schemes, and show how to implement these methods. The first scheme is termed the LRNN-DG (local randomized neural networks with discontinuous Galerkin) method, which uses the output fields of the last hidden layer as the local basis functions for the DG formulation on each subdomain and solves the final system of linear equations by a linear solver or a least-squares method. The other two schemes are termed LRNN-$C^0$DG (local randomized neural networks with $C^0$ discontinuous Galerkin) and LRNN-$C^1$DG (local randomized neural networks with $C^1$ discontinuous Galerkin) methods, respectively, in which continuity conditions for the function and its gradients are enforced across the sub-domain boundaries.
We provide a convergence analysis of these methods under certain appropriate assumptions. For the time-dependent problem, we use the heat equation as a model and propose three space-time LRNN/DG type formulations. The space-time approach is very natural for neural networks and we do not need to compute the numerical solution with time iteration. Finally, we present numerical examples to show that the proposed methods are able to compete with traditional methods in some aspects. First, when the number of degrees of freedom is fixed and small, the accuracies of the LRNN/DG type methods developed herein are better than the FEM and the usual DG methods. Second, when the time-space approach is adopted, a notable advantage is that the error accumulation can be avoided, and one can get the numerical solution at any time instant without interpolation.

The rest of this paper is organized as follows. In Section \ref{sec:weak}, we introduce the idea of randomized neural networks and propose three LRNN-DG formulations for solving a Poisson equation. In Section \ref{sec:csemi}, we make certain assumptions and present a convergence analysis of the methods. In Section \ref{sec:semi}, three space-time LRNN/DG methods are given for solving the heat equation. In Section \ref{sec:numericalexample}, we present numerical examples to show the performance of the proposed methods. Finally, we make a summary in the last section.

\section{Local randomized neural networks with DG methods}
\label{sec:weak}

In this section, we first describe one type of neural network, the so-called randomized neural network. Then, taking the Poisson equation as an example, we introduce local randomized neural networks with the discontinuous Galerkin formulations for solving the problem.

\subsection{Randomized neural networks}\label{rnn}

The general deep neural networks can be represented as compositions of many hidden layers and an output layer. 
A hidden layer is defined as a composition of a linear transformation and an activation function as follows:
\begin{equation*}
N(\bx) =\sigma(\bW\bx+\bb), 
\end{equation*}
where $\bx \in \Omega \subset \mathbb{R}^d$, ${N}(\bx) \in \mathbb{R}^{\bar{d}}$, $\bW \in \mathbb{R}^{\bar{d} \times d}$ is the matrix of weights, $\bb \in \mathbb{R}^{\bar{d}}$ is the bias, and $\sigma$ is a nonlinear activation function. The first layer is usually called the input layer, and the number of layers is the depth of the neural network. The output layer is a linear transformation
\begin{equation*}
N^o(\bx) =\bW\bx+\bb, 
\end{equation*}
where $\bx \in \mathbb{R}^{\bar{d}}$, ${N^o}(\bx) \in \mathbb{R}^{n_{o}}$, $\bW \in \mathbb{R}^{{n_{o}} \times \bar{d}}$ is the weight, and $\bb \in \mathbb{R}^{n_{o}}$ is the bias. Here, ${n_{o}}$ is the dimension of output data. 

Then a fully connected neural network can be represented by
\begin{equation*}
\mathcal{U}(\bx) =\bW^{(L+1)} ( N^{(L)} \circ \cdots N^{(2)} \circ N^{(1)}(\bx))+\bb^{(L+1)},
\end{equation*}
where L is the depth of the neural network, $\bW^{(l)} \in \mathbb{R}^{n_l \times n_{l-1}}$, and $\bb \in \mathbb{R}^{n_l}$ are the parameters, and $n_l$ is the width of the l-th layer of the neural network. Given the depth of the network, and the width of each layer, we denote the set of NN functions by
\begin{equation*}
\mathcal{M}(\theta,L)=\{\mathcal{U}(\bx)= \bW^{(L+1)} (N^{(L)} \circ \cdots \circ N^{(1)}(\bx))+\bb^{(L+1)}:\bW ^{(l)} \in \mathbb{R}^{n_l \times n_{l-1}},\bb^{(l)} \in \mathbb{R}^{n_l}, l=1,...,L+1\},
\end{equation*}
where $\theta = \{(\bW^{(l)} , \bb^{(l)})\}^{L+1}_{l=1}$.

Next, let us consider randomized neural networks. The structure of randomized neural networks is the same as that of fully connected neural networks. The difference lies in that all the parameters of the fully connected NNs are to be trained, while for randomized neural networks the output-layer parameters  are adjustable and the hidden-layer parameters are randomly assigned and fixed. We consider single-hidden-layer NNs with the dimension of the output layer being one, i.e.~$\bW^{(2)}\in\mathbb{R}^{1 \times n_{1}}$. Moreover, let $\bb^{(2)}$ be zero. So in the randomized neural networks framework, we define the following function space
\begin{equation}\label{rnnspace}
\mathcal{M}_{RNN} (D)=\left\{\mathcal{U}(\alpha, \theta, \bx) = \sum^M_{j=1} \alpha^D_j \phi^D (\theta_j, \bx) : \bx \in D\right\},
\end{equation}
where $D \subset \Omega$ is the domain, $M =n_1$ is the width of the last hidden layer, $\theta$ is the parameter of hidden layers, $\alpha$ is the parameter of the output layer, and $\phi$ is a nonlinear function that represents the output of the last hidden layer. For the sake of simplicity, we rewrite $\phi^D (\theta_j,\bx)$ as $\phi^D_j(\bx)$ in the rest of the paper.

\subsection{LRNN-DG method}\label{lrnndgmethod}

In \cite{Dong2021locELM}, we see that the locELM which combines the idea of randomized neural networks and the domain decomposition is very successful in solving partial differential equations, and can compete with traditional methods such as FEM. This method has shown its strong potential for numerically solving PDEs. However, locELM is based on the strong form of PDEs, which may be unfavorable for problems that can only be described by weak formulations.
The departure point of this paper lies in that
here we combine local randomized NNs and the DG methods for solving PDEs in weak forms. That is, we use the output fields of the local neural networks' last hidden layers to form local basis functions, and use DG formulation to glue them together. 

Let us introduce the local randomized neural networks with the discontinuous Galerkin formulation. Here, we take the Poisson equation as a model problem,
\begin{subequations}\label{possioneq}
\begin{align}
- \Delta u &= f \quad {\rm in}\ \Omega, \\
u &= g \quad {\rm on}\ \partial \Omega,\label{strongboundc}
\end{align}
\end{subequations}
where $f$ is a given source term, $\partial \Omega$ is the boundary of $\Omega$, and $g$ is a function defined on $\partial \Omega$. The weak formulation of the above problem is: Find $u\in H^1_g(\Omega)$ such that
\begin{align}\label{weakform}
a(u,v) = \int_\Omega f\, v\, \dx\quad \forall v\in H^1_0(\Omega).
\end{align}
Here, $H^1_g(\Omega) = \{v\in H^1(\Omega): v=g\; {\rm on}\; \partial\Omega \}$ and
$$
a(u,v) = \int_\Omega \nabla u\cdot \nabla v\, \dx.
$$

Like the setting in locELM and DG method, we partition the domain into some subdomains, and approximate the solution on each subdomain by using a local neural network. First, we give some notation. Let $\{ \mathcal{T}_h \}$ be the decomposition of $\bar{\Omega}$, where $h = \max_{K\in {\cal T}_h}\{{\rm diam}(K)\}$. For each $\mathcal{T}_h$, $N_e$ denotes the number of elements in $\mathcal{T}_h$, that is, $|\mathcal{T}_h|=N_e$. Let ${\cal E}_h$ be the union of the boundaries of all the elements $K \in \mathcal{T}_h$, $\mathcal{E}_h^i$ is the set of all interior edges, and ${\cal E}_h^\partial = {\cal E}_h \backslash {\cal E}_h^i$. Let $K^+$ and $K^-$ be two neighboring elements sharing a common edge $e$. Denote by
$\boldsymbol{n}^{\pm}=\boldsymbol{n}|_{\partial K^{\pm}}$ the unit outward normal vectors
on $\partial K^{\pm}$. For a scalar function $v$ and a
vector-valued function $\boldsymbol{q}$, let $v^{\pm}=v|_{\partial K^{\pm}}$ and $\boldsymbol{q}^\pm=\boldsymbol{q}|_{\partial K^{\pm}}$. We define the averages $\{ \cdot \}$ and the jumps $\llbracket \cdot \rrbracket$, $[ \cdot ]$ on $e \in \mathcal{E}_h^i$ by
\begin{subequations}
\begin{align}
\{ v \} = \frac{1}{2} (v^+ + v^-),& \quad \llbracket v\rrbracket = v^+ \boldsymbol{n}^+ + v^- \boldsymbol{n}^-, \nonumber\\
\{ \boldsymbol{q} \} = \frac{1}{2}(\boldsymbol{q}^+ +\boldsymbol{q}^-),& \quad
[\boldsymbol{q}] = \boldsymbol{q}^+ \cdot \boldsymbol{n}^+ +\boldsymbol{q}^- \cdot \boldsymbol{n}^-.\nonumber
\end{align}
\end{subequations}
If $e \in \mathcal{E}_h^{\partial}$, we set
\begin{subequations}
\begin{align}
\llbracket v\rrbracket = v \boldsymbol{n},\quad
\{ \boldsymbol{q} \} = \boldsymbol{q},\nonumber
\end{align}
\end{subequations}
where $\boldsymbol{n}$ is the unit outward normal vector on $\partial \Omega$.
In the analysis, we need the following identities:
\begin{align}
\int_K \nabla v \cdot \boldsymbol{q}\,\dx &= -\int_K v~(\nabla\cdot \boldsymbol{q}) \,\dx + \int_{\partial K} v\,\boldsymbol{q}\cdot\bn_K\,\ds,\label{ibps}\\
\sum_{K\in \mathcal{T}_h} \int_{\partial K} v \boldsymbol{q} \cdot \boldsymbol{n}_K \, \ds
&= \int_{{\cal E}_h} \llbracket v\rrbracket \cdot \{\boldsymbol{q}\}\, \ds
+ \int_{{\cal E}_h^i} \{v\} [\boldsymbol{q}]\, \ds.\label{iden_dg}
\end{align}

We introduce the following DG space based on local randomized neural networks associated with the partition $ \mathcal{T}_h$:
\begin{align*}
V_h&=\{v_h \in L^2(\Omega): \;v_h |_K \in \mathcal{M}_{RNN}(K)\quad \forall\,K\in\mathcal{T}_h \},
\end{align*}
where $\mathcal{M}_{RNN}(K)$ denotes the function space of the randomized neural networks given in \eqref{rnnspace}. So for each $v_h\in V_h$, $v_h|_K=\sum^M_{j=1}v^K_j\phi^K_j(x)$. 

We make the following assumption. 
\begin{assumption}\label{assump:indep}
For any $K \in \mathcal{T}_h$, assume that the functions $\{\phi^{K}_j(\bx): j = 1,2, \cdots, M\}$ of last hidden layers in subdomain $K$ are linearly independent.
\end{assumption}

For example, let $\phi^K_j(\bx)=\sin(\bW_j \bx + \bb_j)$, then $\{\sin(\bW_j\bx + \bb_j),\, j=1,2,\cdots,M\}$ is a set of linearly independent functions if proper values of weights $\bW_j$ and bias $\bb_j$ are chosen. Of course, this assumption can be satisfied for other activation functions, like $\phi^K_j(\bx)=\tanh(\bW_j\bx + \bb_j)$.

The local randomized neural networks with DG (LRNN-DG) method for solving the Poisson problem is:
 Find $u_h\in V_h$ such that
  \begin{equation} \label{DisForm_DG}
    a_h(u_h,v_h) = l(v_h)\quad \forall\, v_h \in V_h,
  \end{equation}
  where
  \begin{align}\label{biform}
    a_h(u,v) &=\int_{\Omega} \nabla_{h} u \cdot \nabla_{h} v \dx - \int_{{\cal E}_h} \{ \nabla_{h} u \} \cdot \llbracket v\rrbracket \ds - \int_{{\cal E}_h} \llbracket u\rrbracket \cdot \{ \nabla_{h} v \} \ds + \int_{{\cal E}_h} \eta\llbracket u\rrbracket \cdot \llbracket v\rrbracket \ds,\\
    l(v_h) &= \int_{\Omega} f v_h \ds - \int_{{\cal E}_h^\partial} g \boldsymbol{n}\cdot \nabla_{h} v_{h} \ds +\int_{{\cal E}_h^\partial} \eta g v_h \ds.
    \label{liform}
  \end{align}
Here, $\nabla_hv_h=\nabla v_h|_K$, $\int_{{\cal E}_h} \eta \llbracket u\rrbracket \cdot \llbracket v\rrbracket$ ds is the penalty term, the function $\eta$ equals a constant $\eta_e(h_e)^{-1}$ on each $e \in \mathcal{E}_h$, with $\eta_e$ being a positive number.
In this paper, we focus on the interior penalty DG (IPDG) scheme and note that other DG schemes studied in \cite{Arnold2002dgs} can be considered as well, that is, the bilinear form \eqref{biform} and the linear form \eqref{liform} can be changed to other DG formulations.

The outputs of the last hidden layer $\{\phi^{K}_j(\bx): K \in \mathcal{T}_h,\; j = 1,2, \cdots, M\}$ can be regarded as the local basis functions of the LRNN-DG, then we can get global stiffness matrices $\mathbb{A}$ and the right-hand side $L$. This is a system of linear algebraic equations, all the values of terms in $\mathbb{A}$ and $L$ are computable. We also need to note that the penalty number in IPDG scheme \eqref{DisForm_DG} needs to be adjusted with the real problem. From \eqref{DisForm_DG}, we obtain the system of equations,
\begin{align}\label{linear_system}
\mathbb{A}U=L,
\end{align}
where $\mathbb{A}$ is a $N_e M \times N_e M $ matrix, $L$ is a $N_e M \times 1$ vector, and there are $N_e M$ unknown variables $U = \{u^{K}_j: K \in \mathcal{T}_h, j = 1,2, \cdots, M\}$.  

\subsection{Some properties of the LRNN-DG method}

The following lemma shows the consistency of the DG scheme, a similar argument can be found in \cite{Arnold2002dgs} and other references on DG methods. For completeness, we give a brief proof as well.
\begin{lemma}[Consistency]\label{lem:consis}
The LRNN-DG scheme is consistent, i.e., for the solution $u\in H^2(\Omega)$ of problem \eqref{weakform}, we have
\begin{equation}\label{consis}
a_h(u,v_h)= l(v_h)\quad\forall v_h\in V_h.
\end{equation}
\end{lemma}
\begin{proof}
We know that $u\in H^2(\Omega)$ implies $\llbracket u\rrbracket=0$, $[ \nabla u]=0$ on $\mathcal{E}^i_h$ and $u=g$ on $\mathcal{E}^\partial_h$. Then, by the identities \eqref{ibps}, \eqref{iden_dg} and \eqref{possioneq}, we have
\begin{equation*}
\begin{aligned}
a_h(u,v_h)=&\int_{\Omega} \nabla u \cdot \nabla_h v_h\dx - \int_{\mathcal{E}_h} \{\nabla_h u\}\cdot \llbracket v_h\rrbracket \ds
- \int_{{\cal E}_h^\partial} g \boldsymbol{n}\cdot \nabla_{h} v_{h} \ds +\int_{{\cal E}_h^\partial} \eta g v_h \ds\\
=& -\int_{\Omega} \Delta u v_h\dx +\sum_{K\in\mathcal{T}_h}\int_{K} \nabla u \cdot \boldsymbol{n}_K v_h\ds - \int_{\mathcal{E}_h} \{\nabla_h u\}\cdot \llbracket v_h\rrbracket \ds\\&- \int_{{\cal E}_h^\partial} g \boldsymbol{n}\cdot \nabla_{h} v_{h} \ds +\int_{{\cal E}_h^\partial} \eta g v_h \ds\\
=& l(v_h).
\end{aligned}
\end{equation*}
\end{proof}

From the LRNN-DG scheme \eqref{DisForm_DG} and Lemma \ref{lem:consis}, we have 
\begin{equation}
a_h(u-u_h,v_h)=0 \quad \forall\,v_h\in V_h.\label{G_orth}
\end{equation}

Next, let $V(h)=V_h + H^2(\Omega)$, then we define some seminorms and norms by the following relations:
\begin{equation}
\begin{aligned}
|v|^2_{1,h}=\sum_{K\in \mathcal{T}_h}|v|^2_{1,K},&\quad |v|^2_{1,*}=\sum_{e\in \mathcal{E}_h}h_e^{-1}\Vert\llbracket v\rrbracket\Vert^2_{0,e},\\
\Vert v\Vert^2_w = | v|^2_{1,h} +| v|^2_{1,*},&\quad\Vert v\Vert^2_* = \Vert v\Vert^2_w +\sum_{K\in \mathcal{T}_h}h^2_K |v|^2_{2,K}.
\end{aligned}
\end{equation}
The norms $\Vert v\Vert^2_w$ and $\Vert v\Vert^2_*$ are well-defined because
\begin{equation}
\Vert v\Vert_0\leq C\Vert v\Vert_w\leq C\Vert v\Vert_* \quad \forall v \in V(h).
\end{equation}
Here, and in the rest of the paper, $C$ denotes a constant that is independent of $h$ and $M$.

Then we have the boundedness and stability of the bilinear form $a_h$ by standard argument (see \cite{Arnold2002dgs} and the references therein).
\begin{lemma}[Boundedness]\label{lem:bound}
$a_h(u,v)$ satisfies
\begin{equation}
a_h(u,,v)\leq C_b\Vert u\Vert_*\Vert v\Vert_*\quad\forall\,v\in V(h). \label{bounded}
\end{equation}
\end{lemma}

\begin{lemma}[Stability]\label{lem:sta}
Set $\eta_0=\inf_e\eta_e>0$, if $\eta_0$ is large enough, then $a_h$ satisfies
\begin{equation}
a_h(v,v)\geq C_s\Vert v\Vert^2_*\quad \forall\,v\in V_h. \label{stability}
\end{equation}
\end{lemma}

By Assumption \ref{assump:indep}, Lemma \ref{lem:bound} and Lemma \ref{lem:sta}, we know that problem \eqref{DisForm_DG} is well-posed and $\mathbb{A}$ is symmetric positive definite (SPD). Therefore, many solvers for SPD system can be used to solve \eqref{linear_system}.
The randomized neural network has a certain possibility that the functions $\{\phi^{K}_j(\bx): j = 1,2, \cdots, M\}$ are not linearly independent, which means that $\mathbb{A}$ is singular, and we need to solve the linear system \eqref{linear_system} by the least-squares approach. Then, the parameters ${U}$ in the neural networks' output layers can be obtained by a least-squares method.

\subsection{LRNN-$C^0$DG method}\label{lrnnc0dgmethod}

The LRNN-DG method introduced in the previous subsection is based on the IPDG scheme. It is well-known that the performance of the IPDG method depends on the choice of the penalty parameter $\eta$. It might be cumbersome to determine an appropriate value of the penalty parameter. Of course, we can use other DG formulation, such as local DG, to avoid the difficulty of choosing proper penalty parameter. However, taking advantage of least square method, we can enforce the $C^0$-continuous condition on each $e \in \mathcal{E}^i_h$ and the Dirichlet boundary condition on $\mathcal{E}^{\partial}_h$ to avoid this issue. 

We add additional equations to enforce the solution to satisfy the boundary condition \eqref{strongboundc}, that is, we choose some collocation points on the boundary edge, $P^\partial_h=\{\bx^e_j \in e: e\in \mathcal{E}^\partial_h,j=1,2,\cdots,N^e_\partial\}$ and $|P^\partial_h|=N_\partial$, such that
\begin{equation}\label{boundary_C0}
u_h(\bx^e_j) = g(\bx^e_j)\quad \forall \bx^e_j\in P^\partial_h.
\end{equation}
In addition, we also need to make sure that the numerical solution $u_h$ satisfies certain $C^0$-continuity conditions across the interior edges $e\in\mathcal{E}^i_h$. We choose some collocation points on the interior edges, $P^i_h=\{\bx^e_j \in e:  e\in \mathcal{E}^i_h,j=1,2,\cdots,N^e_{in}\}$ and $|P^i_h|=N_{in}$, on these points, we set
\begin{equation}\label{inboundary_C0}
\llbracket u_h(\bx^e_j )\rrbracket = 0\quad\forall \bx^e_j\in P^i_h.
\end{equation}
Then we obtain a system of equations with respect to \eqref{boundary_C0} and \eqref{inboundary_C0}, 
\begin{equation}
\mathbb{A}_2 U = L_2,
\end{equation}
where $\mathbb{A}_2$ is a $(N_\partial + N_{in}) \times N_e M$ matrix, $U$ is the $N_e M \times 1$ unknown vector, and $L_2$ is a $(N_\partial + N_{in}) \times 1$ vector. 

Condition \eqref{inboundary_C0} makes $\llbracket u_h \rrbracket \approx 0$, so the LRNN-$C^0$DG scheme is to find $u_h\in V_h$ such that
\begin{equation} \label{DisForm_C0}
\begin{aligned}
    a^0_h(u_h,v_h) &= l^0(v_h)\quad &\forall v_h \in V_h,\\
    \llbracket u_h(\bx^e_j )\rrbracket &= 0\quad &\forall \bx^e_j\in P^i_h,\\
    u_h(\bx^e_j) &= g(\bx^e_j)\quad &\forall \bx^e_j\in P^\partial_h,\\
    \end{aligned}
\end{equation}
  where
\begin{align}\label{biform_C0}
 {a^0_h}(u_h,v_h) &=\int_{\Omega} \nabla_{h} u_h \cdot \nabla_{h} v_h \dx - \int_{{\cal E}_h} \{ \nabla_{h} u_h \} \cdot \llbracket v_h\rrbracket \ds - \int_{{\cal E}_h} \{ \nabla_{h} v_h \} \cdot \llbracket u_h\rrbracket \ds.\\
{l^0}(v_h) &= \int_{\Omega} f v_h \ds - \int_{{\cal E}_h^\partial} g \boldsymbol{n}\cdot \nabla_{h} v_{h} \ds.
\end{align}
Here, we keep the term $\int_{{\cal E}_h} \{ \nabla_{h} v_h \} \cdot \llbracket u_h\rrbracket \ds$ for the symmetry of the bilinear form $a^0_h$. Note that this scheme is free of penalty parameters. 
Finally, from \eqref{DisForm_C0}, we get a linear system,
\begin{equation}
\begin{bmatrix} \mathbb{A}_1 \\ \mathbb{A}_2 \end{bmatrix} U = \begin{bmatrix} L_1 \\ L_2 \end{bmatrix},
\end{equation}
where $\mathbb{A}_1$ is a $N_e M \times N_e M$ matrix, $L_1$ is a $N_e M \times 1$ vector.
We look for the least-squares solution for this linear system.

\begin{rem}
From numerical examples, we see that the scheme \eqref{DisForm_C0} has a good performance. But note that the following scheme which destroys the symmetry still works well.
\begin{equation} \label{DisForm_C02}
    \widetilde{a^0_h}(u_h,v_h) =\int_\Omega fv_h\ds\quad \forall\, v_h \in V_h,
\end{equation}
  where
\begin{align}\label{biform_C02}
    \widetilde{a^0_h}(u_h,v_h) &=\int_{\Omega} \nabla_{h} u_h \cdot \nabla_{h} v_h \dx - \int_{{\cal E}_h} \{ \nabla_{h} u_h \} \cdot \llbracket v_h\rrbracket \ds.
\end{align}
\end{rem}

\subsection{LRNN-$C^1$DG method}
\label{lrnnc1dgmethod}

In the previous subsection, the DG scheme has been simplified by enforcing the continuity of $u_h$, that is, by setting $\llbracket u_h\rrbracket=0$ on $e\in\mathcal{E}^i_h$. Can we move further in this direction? Let us introduce the LRNN-$C^1$DG method in this subsection.

In each subdomain $K$, $-\Delta u=f$, so we have 
\begin{equation}\label{femform}
\int_{K} \nabla u \cdot \nabla v \dx - \int_{\partial K} \nabla u \cdot \boldsymbol{n}_K v \ds = \int_{K} f v \dx\quad \forall K\in\mathcal{T}_h,
\end{equation}
where, $\boldsymbol{n}_K$ is the unit outer normal vector on $\partial K$. 
We know that \eqref{femform} with Dirichlet boundary condition \eqref{strongboundc} is not equivalent to the Poisson problem because the local problems lack the connections with each others. From domain decomposition method (\cite{Mathew2008dp}), we know that the continuities of $u$ and flux are needed, i.e., we require $\llbracket u_h\rrbracket=0$ and $[\nabla u_h]=0$ on each $e\in\mathcal{E}^i_h$.

We need to make sure that local representations of the solution have the $C^1$-continuity conditions across the subdomain boundaries because of the consistency.
We choose some points on the interior edges $P^i_h$ stated in Section \ref{lrnnc0dgmethod}, on these points, with the same set-up of the LRNN-$C^0$DG method, we have LRNN-$C^1$DG method: find $u_h\in V_h$ such that 

\begin{equation}
\label{DisForm_C1}
\begin{aligned}
 a^K_h(u_h,v_h) &= \int_{K} f v_h \dx&\quad \forall\, v_h \in V_h\quad\forall\ K \in \mathcal{T}_h,\\
\llbracket{u_h(\bx^e_j)}\rrbracket&=0&\quad \forall\, \bx^e_j\in P^i_h,\\
[ \nabla{u_h}(\bx^e_j )] &= 0&\quad\forall\ \bx^e_j \in P^i_h,\\
u_h(\bx^e_j) &=g(\bx^e_j)&\quad \forall\,  \bx^e_j\in P^\partial_h,
\end{aligned}
\end{equation}
where
\begin{align}\label{biform_C1}
        a^K_h(u_h,v_h) =\int_{K} \nabla_h u_h \cdot \nabla_h v_h \dx - \int_{\partial K} \nabla_h u_h \cdot \boldsymbol{n}_K v_h \ds.
\end{align}

Finally, we obtain the following linear system,
\begin{equation*}
\begin{bmatrix} \mathbb{A}_1 \\ \mathbb{A}_2 \end{bmatrix} U = \begin{bmatrix} L_1 \\ L_2 \end{bmatrix},
\end{equation*}
where $ \mathbb{A}_1$ is a $N_e M \times N_e M$ matrix, $L_1$ is a $N_e M \times 1$ vector. $ \mathbb{A}_2$ is a $(2N_{in} + N_\partial) \times N_e M$ matrix, U is a $N_e M \times 1$ vector of unknown variables, $L_2$ is a $(2N_{in} + N_\partial) \times 1$ vector.
We look for the least-squares solution to this system. After the weights of the output layer in each local neural network are obtained by the linear least-squares computation, we can get all the values of the problem \eqref{possioneq} in the domain $\Omega$.

\section{Convergence of the LRNN with DG methods}\label{sec:csemi}

\subsection{Convergence of the LRNN-DG method}

We now consider the error analysis for the LRNN-DG method. From \cite{Liu2014ELMfeasible}, we know that if the exact solution $u$ is a smooth function, ELM does not degrade the generalization capability of neural networks with the proper activation functions and random initialization strategies. And in \cite{Guhring2021approximation, Jiao2021errordeepRithz}, we know that the neural networks can approximate the solution well with the proper depth and width. We make the following assumption. Let $u_\sigma\in V_h$ be a suitable approximation of the exact solution $u$. 
\begin{assumption}\label{assump_appro}
Given a decomposition $\mathcal{T}_h$ with $|\mathcal{T}_h|=N_e$ and $V_h$ is the associated DG space of LRNN. For any small positive number $\epsilon$, there exists a positive integer $M_\epsilon$ such that if $M > M_\epsilon$, we have a function $u_\sigma\in V_h$ satisfying
$$
\Vert u-u_\sigma \Vert_{0,K}\leq C h_KN_e^{-1/2}\epsilon,\quad |u-u_\sigma|_{1,K}\leq C N_e^{-1/2} \epsilon, 
\quad | u-u_\sigma |_{2,K}\leq C h_K^{-1}N_e^{-1/2}\epsilon.
$$
Here, $M$ is the number of the basis of $\mathcal{M}_{RNN}(K)$ and $C$ denotes a constant number that is independent of $h$ and $M$.
\end{assumption}

\begin{remark}
For any function $u\in H^{p+1}(K)$, we know that there exists a polynomial function $u_I\in P_p(K)$ such that
$$
\Vert u-u_I \Vert_{0,K}\leq C h_K^{p+1} |u|_{p+1,K} ,\quad |u-u_I|_{1,K}\leq C h_K^{p} |u|_{p+1,K}, \quad | u-u_I |_{2,K}\leq C h_K^{p-1} |u|_{p+1,K}.
$$
Similarly, we make Assumption \ref{assump_appro} in light of good approximation properties of neural networks.
\end{remark}

From the above assumption and the trace inequality, we have
\begin{align}
\Vert u-u_\sigma\Vert^2_* = & \sum_{K\in \mathcal{T}_h}|u-u_\sigma|^2_{1,K} +\sum_{K\in \mathcal{T}_h}h^2_K |u-u_\sigma|^2_{2,K}
+ \sum_{e\in \mathcal{E}_h}h_e^{-1}\Vert\llbracket u-u_\sigma\rrbracket\Vert^2_{0,e}\nonumber\\
\leq & C \left(\sum_{K\in \mathcal{T}_h}|u-u_\sigma|^2_{1,K} +\sum_{K\in \mathcal{T}_h}h^2_K |u-u_\sigma|^2_{2,K}
+\sum_{K\in \mathcal{T}_h}h^{-2}_K \|u-u_\sigma\|^2_{0,K}\right)\nonumber\\
\leq & C \epsilon.\label{inter}
\end{align}

For the LRNN-DG method, we have the following Ce\'a-type inequality.
\begin{theorem}
Let $u$ and $u_h$ be solutions of the problem \eqref{weakform} and the LRNN-DG scheme \eqref{DisForm_DG}, we obtain
\begin{equation}\label{ceaineq}
\Vert u-u_h\Vert_*\leq (1+C_b/C_s)\inf_{v_h\in V_h}\Vert u-v_h\Vert_*.
\end{equation}
\end{theorem}
\begin{proof}
For any $v_h\in V_h$, by the boundedness \eqref{bounded} and stability \eqref{stability} of the bilinear form $a_h$, as well as \eqref{G_orth}, we have
\begin{equation*}
\begin{aligned}
C_s\Vert v_h-u_h\Vert^2_* &\leq a_h(v_h-u_h,v_h-u_h)\\
 & = a_h (v_h-u,v_h-u_h) + a_h(u-u_h,v_h-u_h)\\
 & \leq C_b\Vert v_h-u\Vert_*\Vert v_h-u_h\Vert_*,
\end{aligned}
\end{equation*}
then we get
\begin{equation}
\Vert v_h-u_h\Vert_*\leq C_b/C_s\Vert u-v_h\Vert_*.
\end{equation}

Finally, by triangle inequality, we obtain
\begin{equation}
\Vert u-u_h\Vert_*\leq \Vert u-v_h\Vert_* +\Vert v_h-u_h\Vert_* \leq (1+C_b/C_s)\Vert u-v_h\Vert_*,
\end{equation}
which completes the proof of the theorem.
\end{proof}

From the Ce\'a-type inequality and \eqref{inter}, let $v_h = u_\rho$ in \eqref{ceaineq}, we can obtain the convergence of the LRNN-DG scheme \eqref{DisForm_DG}.
\begin{corollary}\label{thm:DG_con}
Let $u$ and $u_h$ be solutions of the problems \eqref{weakform} and \eqref{DisForm_DG}, respectively. If Assumption \ref{assump_appro} holds, then
for any small positive number $\epsilon$, there exists a positive integer $M_\epsilon$ such that if $M > M_\epsilon$, then
\begin{align}
\Vert u-u_h\Vert_*\leq C  \epsilon.
\end{align}
\end{corollary}

\subsection{Convergence of the LRNN-$C^0$DG method}

In this subsection, let $\widetilde{u_h}$ denote the solution of the LRNN-$C^0$DG scheme \eqref{DisForm_C0}.
By enforcing the conditions \eqref{boundary_C0} and \eqref{inboundary_C0}, we have $\widetilde{u_h} -g \approx 0$ on boundary edges and $\llbracket\widetilde{u_h}\rrbracket \approx 0$ on interior edges. Especially, when the number of the points $\bx_j^e$ on each edge $e$ are increased, $\widetilde{u_h} -g$ and $\llbracket\widetilde{u_h}\rrbracket$ will become as smaller as we need. Therefore, let us make an assumption as follows.

\begin{assumption}\label{assump_approC0}
Given a decomposition $\mathcal{T}_h$ with $|\mathcal{T}_h|=N_e$ and $V_h$ is the associated DG space of LRNN. For any small positive number $\epsilon$, on every edge $e\in\mathcal{E}_h$, there exist $N_\epsilon^e$ such that if $N^e > N_\epsilon^e$, then
$$
\Vert \llbracket\widetilde{u_h}\rrbracket \Vert_{0,e}\leq C h_e^{1/2}\epsilon \quad {\rm and} \quad \Vert \widetilde{u_h} -g \Vert_{0,e}\leq C h_e^{1/2}\epsilon.
$$
Here, $N^e$ is the number of points $\bx_j^e$ on the edge $e$ and $C$ denotes a constant number which is independent of $h$ and $M$.
\end{assumption}

Next, we prove the convergence of the LRNN-$C^0$DG scheme \eqref{DisForm_C0}. 

\begin{theorem} 
Let $u$ and $\widetilde{u_h}$ be solutions of the problem \eqref{weakform} and the LRNN-$C^0$DG scheme \eqref{DisForm_C0}, respectively. If Assumption \ref{assump_appro} and Assumption \ref{assump_approC0} hold, for any small positive number $\epsilon$, there exist positive integers $M_\epsilon$, $N_\epsilon^e$ such that if $M > M_\epsilon$, $N^e > N_\epsilon^e$, then
\begin{equation}
\Vert u-\widetilde{u_h}\Vert_*\leq C\epsilon.
\end{equation}
\end{theorem}
\begin{proof}
From the LRNN-$C^0$DG scheme \eqref{DisForm_C0} and the LRNN-DG scheme \eqref{DisForm_DG}, we know that 
\begin{equation}
\begin{aligned}
{a^0_h}(\widetilde{u_h},v_h) = {l^0}(v_h)\quad \forall\, v_h \in V_h,\\
{a_h}({u_h},v_h) = {l}(v_h)\quad \forall\, v_h \in V_h,
\end{aligned}
\end{equation}
so 
\begin{equation}
\begin{aligned}
{a_h}(\widetilde{u_h},v_h) -\int_{{\cal E}_h} \eta\llbracket \widetilde{u_h}\rrbracket \cdot \llbracket v_h\rrbracket \ds = l(v_h)- \int_{{\cal E}_h^\partial} \eta g  v_h \ds.
\end{aligned}
\end{equation}
And by the consistency \eqref{consis}, we have
\begin{equation}
\begin{aligned}
a_h(\widetilde{u_h}-u,v_h)&= \int_{{\cal E}_h} \eta\llbracket \widetilde{u_h}\rrbracket \cdot \llbracket v_h\rrbracket \ds -\int_{{\cal E}_h^\partial} \eta g  v_h \ds.
\end{aligned}
\end{equation}

From the stability and boundedness of $a_h$ and \eqref{G_orth}, we get
\begin{equation*}
\begin{aligned}
C_s\Vert \widetilde{u_h}-u_h\Vert^2_*&\leq a_h(\widetilde{u_h}-u_h,\widetilde{u_h}-u_h)\\
&=a_h(\widetilde{u_h}-u,\widetilde{u_h}-u_h)+a_h(u-u_h,\widetilde{u_h}-u_h)\\
&=\int_{{\cal E}_h} \eta\llbracket \widetilde{u_h}\rrbracket \cdot \llbracket \widetilde{u_h}-u_h\rrbracket \ds -\int_{{\cal E}_h^\partial} \eta g  (\widetilde{u_h}-u_h) \ds\\
&\leq\eta_{\rm max}\left(\sum_{e\in\mathcal{E}^i_h}h^{-1}_e\Vert\llbracket \widetilde{u_h}\rrbracket\Vert^2_{0,e}+\sum_{e\in\mathcal{E}^\partial_h}h^{-1}_e\Vert \widetilde{u_h}-g\Vert^2_{0,e}\right)^{\frac{1}{2}}|\widetilde{u_h}-u_h|_{1,*}.
\end{aligned}
\end{equation*}
Therefore, by Assumption \ref{assump_approC0}, we have
$$
\Vert \widetilde{u_h}-u_h\Vert_* \leq \frac{\eta_{\rm max}}{C_s}
\left(\sum_{e\in\mathcal{E}^i_h}h^{-1}_e\Vert\llbracket \widetilde{u_h}\rrbracket\Vert^2_{0,e}+\sum_{e\in\mathcal{E}^\partial_h}h^{-1}_e\Vert \widetilde{u_h}-g\Vert^2_{0,e}\right)^{\frac{1}{2}}\leq C\epsilon.
$$

Finally, by triangle inequality and Corollary \ref{thm:DG_con}, we obtain
\begin{align*}
&\Vert u-\widetilde{u_h}\Vert_*\leq \Vert u-{u_h}\Vert_*+\Vert u_h-\widetilde{u_h}\Vert_*\leq C\epsilon.
\end{align*}
\end{proof}

\subsection{Convergence of the LRNN-$C^1$DG method}

In this subsection, let $\overline{u_h}$ denote the solution of the LRNN-$C^1$DG scheme \eqref{DisForm_C1}.
Similar to the case of the LRNN-$C^0$DG method, by enforcing $[ \nabla{\overline{u_h}}(\bx^e_j )]= 0$ for each point $\bx^e_j$ on $e$, we have $[ \nabla{\overline{u_h}}(\bx^e_j )] \approx 0$ as well. So we give the following assumption.

\begin{assumption}\label{assump_approC1}
Given a decomposition $\mathcal{T}_h$ with $|\mathcal{T}_h|=N_e$ and $V_h$ is the associated DG space of LRNN. For any small positive number $\epsilon$, on every edge $e\in\mathcal{E}_h$, there exists $N_\epsilon^e$ such that if $N^e > N_\epsilon^e$, then
$$
\Vert \llbracket\overline{u_h}\rrbracket \Vert_{0,e}\leq C h_e^{1/2}\epsilon,\quad \Vert \overline{u_h} -g \Vert_{0,e}\leq C h_e^{1/2}\epsilon
\quad {\rm and} \quad \Vert [ \nabla\overline{u_h}(\bx^e_j )] \Vert_{0,e}\leq C h_e^{-1/2}\epsilon.
$$
Here, $N^e$ is the number of the points $\bx_j^e$ on edge $e$ and $C$ denotes a constant number which is independent of $h$ and $M$.
\end{assumption}

\begin{rem}
To show that Assumption \ref{assump_approC0} and Assumption \ref{assump_approC1} are reasonable,   
in Section \ref{sec:numericalexample}, we give some numerical evidences which shows that $\|\llbracket u_h\rrbracket\|_{0,e}$, $\|[\nabla u_h]\|_{0,e}$ on interior edges and $\|u_h - g\|_{0,e}$ on boundary edges are decreasing with the increase of the number of collection points. 
\end{rem}

Finally, we show the convergence of the LRNN-$C^1$DG scheme.

\begin{theorem}
Let $u$ and $\overline{u_h}$ be solutions of the problem \eqref{weakform} and the LRNN-$C^1$DG scheme \eqref{DisForm_C1}, respectively.
If Assumption \ref{assump_appro} and Assumption \ref{assump_approC1} hold, for any small positive number $\epsilon$, there exist positive integers $M_\epsilon$, $N_\epsilon^e$ such that if $M > M_\epsilon$, $N^e > N_\epsilon^e$, then
\begin{equation}
\Vert u-\overline{u_h}\Vert_*\leq C\epsilon.
\end{equation}
\end{theorem}
\begin{proof}
We know that in each subdomain,
\begin{equation}
\int_{K} \nabla_h \overline{u_h} \cdot \nabla_h v_h \dx - \int_{\partial K} \nabla_h \overline{u_h} \cdot \boldsymbol{n}_K v_h \ds = \int_{K} f v_h \dx\quad \forall K\in\mathcal{T}_h.
\end{equation}
Then we add over all the elements to obtain
\begin{equation}
a^1_h(\overline{u_h},v_h)=\int_{\Omega} f v_h \dx\quad \forall v_h\in V_h,
\end{equation}
where
\begin{equation}
\begin{aligned}
a^1_h(\overline{u_h},v_h)&=\int_{\Omega} \nabla_h \overline{u_h} \cdot \nabla_h v_h \dx - \sum_{K\in\mathcal{T}_h}\int_{\partial K} \nabla_h \overline{u_h} \cdot \boldsymbol{n}_K v_h \ds\\
&=\int_{\Omega} \nabla_h \overline{u_h} \cdot \nabla_h v_h \dx -\int_{\mathcal{E}_h}\{\nabla_h\overline{u_h}\}\cdot\llbracket v_h\rrbracket\ds-\int_{\mathcal{E}^i_h}\{v_h\}\cdot[ \nabla_h\overline{u_h}]\ds.
\end{aligned}
\end{equation}
So
\begin{equation}
\begin{aligned}
&{a_h}(\overline{u_h},v_h) -\int_{\mathcal{E}^i_h}\{v_h\}\cdot[ \nabla_h\overline{u_h}]\ds +\int_{{\cal E}_h} \llbracket \overline{u_h}\rrbracket \cdot \{\nabla_h v_h\} \ds-\int_{{\cal E}_h} \eta\llbracket \overline{u_h}\rrbracket \cdot \llbracket v_h\rrbracket \ds \\
=& l(v_h)+\int_{{\cal E}_h^\partial} g \boldsymbol{n} \cdot\nabla_h v_h \ds- \int_{{\cal E}_h^\partial} \eta g  v_h \ds.
\end{aligned}
\end{equation}
And we know $a_h({u},v_h)= {l}(v_h)$, so we have
\begin{equation}
\begin{aligned}
&a_h(\overline{u_h}-u,v_h)\\= &\int_{\mathcal{E}^i_h}\{v_h\}\cdot[ \nabla_h\overline{u_h}]\ds -\int_{{\cal E}_h} \llbracket \overline{u_h}\rrbracket \cdot \{\nabla_h v_h\} \ds+\int_{{\cal E}_h} \eta\llbracket \overline{u_h}\rrbracket \cdot \llbracket v_h\rrbracket \ds+\int_{{\cal E}_h^\partial} g \boldsymbol{n} \cdot\nabla_h v_h \ds- \int_{{\cal E}_h^\partial} \eta g  v_h \ds.
\end{aligned}
\end{equation}

By the stability and boundedness of $a_h$ and \eqref{G_orth}, we get
\begin{equation*}
\begin{aligned}
&C_s\Vert \overline{u_h}-u_h\Vert^2_*\leq a_h(\overline{u_h}-u_h,\overline{u_h}-u_h)
=a_h(\overline{u_h}-u,\overline{u_h}-u_h)+a_h(u-u_h,\overline{u_h}-u_h)\\
=&\int_{\mathcal{E}^i_h}\{ \overline{u_h}-u_h\}\cdot[ \nabla_h\overline{u_h}]\ds -\int_{{\cal E}_h} \llbracket \overline{u_h}\rrbracket \cdot \{\nabla_h  (\overline{u_h}-u_h)\} \ds+\int_{{\cal E}_h} \eta\llbracket \overline{u_h}\rrbracket \cdot \llbracket  \overline{u_h}-u_h\rrbracket \ds\\
&+\int_{{\cal E}_h^\partial} g \boldsymbol{n} \cdot\nabla_h ( \overline{u_h}-u_h) \ds- \int_{{\cal E}_h^\partial} \eta g   (\overline{u_h}-u_h) \ds\\
\leq &C\left(\sum_{e\in\mathcal{E}^i_h}h^{-1}_e\Vert\llbracket \overline{u_h}\rrbracket\Vert^2_{0,e}+\sum_{e\in\mathcal{E}^i_h}h_e\Vert[\nabla_h \overline{u_h}]\Vert^2_{0,e}+\sum_{e\in\mathcal{E}^\partial_h}h^{-1}_e\Vert \overline{u_h}-g\Vert^2_{0,e}\right)^{\frac{1}{2}}\|\overline{u_h}-u_h\|_{*},
\end{aligned}
\end{equation*}
so
\begin{equation}
\begin{aligned}
\Vert \overline{u_h}-u_h\Vert_*&\leq\frac{C}{C_s}\left(\sum_{e\in\mathcal{E}^i_h}h^{-1}_e\Vert\llbracket \overline{u_h}\rrbracket\Vert^2_{0,e}+\sum_{e\in\mathcal{E}^i_h}h_e\Vert[\nabla_h \overline{u_h}]\Vert^2_{0,e}+\sum_{e\in\mathcal{E}^\partial_h}h^{-1}_e\Vert \overline{u_h}-g\Vert^2_{0,e}\right)^{\frac{1}{2}}\\
&\leq C\epsilon.
\end{aligned}
\end{equation}

Finally,
\begin{equation}
\Vert u-\overline{u_h}\Vert_*\leq \Vert u-{u_h}\Vert_*+\Vert u_h-\overline{u_h}\Vert_*\leq C\epsilon.
\end{equation}
\end{proof}

\section{Space-time LRNN with DG methods for heat equation}\label{sec:semi}

In this section, we study local randomized neural networks with DG methods for solving a typical time-dependent PDE, heat equation. Unlike the traditional method, which solves the problem by an iteration over time steps, we use the space-time approach, in which temporal and spatial variables are treated equally and jointly. Therefore, the accumulation of errors can be avoided. 

Consider
\begin{subequations}\label{heateq}
\begin{align}
u_t (t, \bx) - \Delta u (t, \bx) &= f (t, \bx) \quad{\rm in}\ \Sigma,\\
u(0, \bx) &= u_0( \bx) \quad\ {\rm in}\ \Omega,\label{timeinicondi}\\
u (t, \bx) & = g (t, \bx) \quad  {\rm on}\ I \times\partial \Omega,\label{timeboundarycondi}
\end{align}
\end{subequations}
where $\Omega\subset \mathbb{R}^d$ is a bounded space domain, $I=(0,\ T)$ is the time interval, $\Sigma = I\times \Omega$ is the space-time domain, $u$ is the unknown solution to be solved, $f$ is the given source term, $u_0$ is the initial condition and $g$ is a function defined on $I\times\partial \Omega$. 

We partition the domain into some subdomains, and approximate the solution on each subdomain by using a local neural network. First, we give the decomposition of the time interval $\mathcal{D}_\tau =\{I_i =(t_{i-1},t_i), 0=t_0< t_1<\cdots<t_{N_t}=T\}$, where $\tau = \max\limits_{I_i\in {\cal D}_\tau}\{{\rm length}(I_i)\}$ and $N_t$ denotes the number of subintervals along the temporal direction. Let ${\cal P}_\tau=\{t_i,i=0,\cdots,{N_t}\}$ be the union of the boundary points of all the intervals $I_i = (t_{i-1}, t_i) \in \mathcal{D}_\tau$, and $\mathcal{P}_\tau^{i} = \mathcal{P}_\tau \backslash \{t_0,t_{N_t}\}$ is the set of all interior points. Let $\{ \mathcal{T}_h \}$ be the decomposition of $\bar{\Omega}$, where $h = \max_{K\in {\cal T}_h}\{{\rm diam}(K)\}$. $\mathcal{E}_h$, $\mathcal{E}^i_h$ and $\mathcal{E}^\partial_h$ have the same definitions stated in Section \ref{lrnndgmethod}. Let $\{\mathcal{D}_\tau\times\mathcal{T}_h\}$ denote the decomposition of the space-time domain $\bar\Sigma$. For $\mathcal{T}_h$, $N_s$ denotes the number of elements in $\mathcal{T}_h$, that is, $N_e=|\mathcal{D}_\tau\times\mathcal{T}_h|=N_t N_s$. Let $\sigma^+$ and $\sigma^-$ be two neighboring elements sharing a common face $f$. Denote by
$\boldsymbol{n}^{\pm}=\boldsymbol{n}|_{\partial \sigma^{\pm}}$ the unit outward normal vectors
on $\partial \sigma^{\pm}$. For a scalar-valued function $v$ and a
vector-valued function $\boldsymbol{q}$, let $v^{\pm}=v|_{\partial \sigma^{\pm}}$ and $\boldsymbol{q}^\pm=\boldsymbol{q}|_{\partial \sigma^{\pm}}$. We define the averages $\{ \cdot \}$ and the jumps $\llbracket \cdot \rrbracket$, $[ \cdot ]$ on $f \in (\mathcal{P}_\tau^{i}\times \mathcal{T}_h)\cup (\mathcal{D}_\tau\times\mathcal{E}_h^{i})$ by
\begin{subequations}
\begin{align*}
\{ v \} = \frac{1}{2} (v^+ + v^-),& \quad \llbracket v\rrbracket = v^+ \boldsymbol{n}^+ + v^- \boldsymbol{n}^-, \\
\{ \boldsymbol{q} \} = \frac{1}{2}(\boldsymbol{q}^+ +\boldsymbol{q}^-),& \quad
[\boldsymbol{q}] = \boldsymbol{q}^+ \cdot \boldsymbol{n}^+ +\boldsymbol{q}^- \cdot \boldsymbol{n}^-.
\end{align*}
\end{subequations}
If $f \in (\mathcal{P}_\tau^{\partial}\times \mathcal{T}_h)\cup(\mathcal{D}_\tau\times \mathcal{E}_h^{\partial})$, we set
\begin{subequations}
\begin{align*}
\llbracket v\rrbracket = v \boldsymbol{n},\quad
\{ \boldsymbol{q} \} = \boldsymbol{q},
\end{align*}
\end{subequations}
where $\boldsymbol{n}$ is the unit outward normal vector on $\partial\Sigma$.

\subsection{Space-time LRNN-DG method}
We introduce the following DG space based on the local randomized neural network associated with the partition $\mathcal{D}_\tau \times \mathcal{T}_h$:
\begin{align*}
V^\tau_h&=\{v^\tau_h \in L^2(\Sigma): \;v^\tau_h |_{I_i\times K} \in \mathcal{M}_{RNN}(I_i\times K)\quad \forall\,I_i\in\mathcal{D}_\tau\ \forall\,K\in\mathcal{T}_h \},\\
\boldsymbol{Q}^\tau_h&=\{\boldsymbol{q}^\tau_h \in [L^2(\Sigma)]^d: \;\boldsymbol{q}^\tau_h |_{I_i\times K} \in [\mathcal{M}_{RNN}(I_i\times K)]^d\quad \forall\,I_i\in\mathcal{D}_\tau\ \forall\,K\in\mathcal{T}_h \}.
\end{align*}

We rewrite the heat equation as the first-order system, 
\begin{subequations}
\begin{align*}
{\boldsymbol{p}} = \nabla u \quad {\rm in} \ \Sigma, \\
\frac{\partial u}{\partial t} - \nabla \cdot {\boldsymbol{p}} = f \quad {\rm in} \ \Sigma.
\end{align*}
\end{subequations}
In the above equations, multiply the test functions ${\boldsymbol{q}}$ and $v$ respectively on subdomain $\sigma = I_i \times K$, then we get by integration by parts,
\begin{equation*}
\begin{aligned}
\int_{\sigma} {\boldsymbol{p}} \cdot {\boldsymbol{q}} \dx \dt = - \int_{\sigma} u \nabla \cdot {\boldsymbol{q}} \dx \dt &+ \int_{I_i\times\partial K} u \boldsymbol{q}\cdot\boldsymbol{n} \ds \dt,\\
- \int_{\sigma} u\frac{\partial v}{\partial t} \dx \dt + \int_{\sigma} {\boldsymbol{p}} \cdot \nabla v \dx\dt + \int_{K}(uv)|^{t_i}_{t_{i-1}}\dx&= \int_{\sigma} f v \dx \dt +\int_{I_i \times \partial K} {\boldsymbol{p}} \cdot\boldsymbol{n} v\ds \dt .
\end{aligned}
\end{equation*}
We append subscript $h$ on $\nabla$, append subcript $\tau$ on $\partial$ and append subscript $h$ and $\tau$ on $u$, $v$, $\boldsymbol{p}$ and $\boldsymbol{q}$. Besides, we use numerical traces $\widehat{u_h^\tau}$ and $\widehat{\boldsymbol{p}^\tau_h}$ to approximate $u$ and $\boldsymbol{p}$ in spatial cross-section $f\in \mathcal{D}_\tau\times\mathcal{E}^{}_h$ and use numerical traces $\widetilde{u_h^\tau}$ to approximate $u$ in temporal cross-section $f\in \mathcal{P}^{}_\tau\times\mathcal{T}_h$, 
\begin{equation*}
\begin{aligned}
\int_{\sigma} {\boldsymbol{p}^\tau_h} \cdot {\boldsymbol{q}^\tau_h} dx dt = - \int_{\sigma} u^\tau_h \nabla_h \cdot {\boldsymbol{q}^\tau_h} \dx\dt &+ \int_{I_i \times \partial K} \widehat{u^\tau_h} \boldsymbol{q}^\tau_h\cdot\boldsymbol{n} \ds\dt,\\
 -\int_{\sigma} u^\tau_h\frac{\partial_\tau v^\tau_h}{\partial_\tau t} \dx \dt + \int_{\sigma} {\boldsymbol{p}^\tau_h} \cdot \nabla_h v^\tau_h \dx\dt + \int_{K}(\widetilde{u^\tau_h} v^\tau_h)|^{t_i}_{t_{i-1}}\dx&= 
 \int_{\sigma} f v^\tau_h \dx \dt + \int_{I_i \times\partial K} \widehat{\boldsymbol{p}^\tau_h} \cdot  \boldsymbol{n} v^\tau_h\ds\dt.
\end{aligned}
\end{equation*}

Then we add over all the elements, use integration by parts and \eqref{iden_dg},
\begin{align*}
&\int_{\Sigma} {\boldsymbol{p}^\tau_h} \cdot {\boldsymbol{q}^\tau_h} dx dt = \int_{\Sigma} \nabla_h u^\tau_h \cdot {\boldsymbol{q}^\tau_h} \dx\dt + \int_{\mathcal{D}_\tau\times\mathcal{E}^{}_h} \llbracket\widehat{u^\tau_h}-u^\tau_h\rrbracket \cdot\{ \boldsymbol{q}^\tau_h\} \ds\dt+ \int_{\mathcal{D}_\tau\times\mathcal{E}^{i}_h} [\boldsymbol{q}^\tau_h] \cdot\{\widehat{u^\tau_h}-u^\tau_h \} \ds\dt,\\
& \int_{\Sigma} \frac{\partial_\tau u^\tau_h}{\partial_\tau t}v^\tau_h \dx \dt + \int_{\Sigma} {\boldsymbol{p}^\tau_h} \cdot \nabla_h v^\tau_h \dx\dt + \sum^{N_t}_{i=0}\int_{\mathcal{T}_h} \llbracket\widetilde{u^\tau_h}(t_i,\bx)-u^\tau_h(t_i,\bx)\rrbracket \cdot\{ v^\tau_h(t_i,\bx)\} \dx -  \int_{\Sigma} f v^\tau_h \dx \dt \\
&+ \sum^{N_t -1}_{i=1}\int_{\mathcal{T}_h} \llbracket v^\tau_h(t_i,\bx)\rrbracket \cdot\{\widetilde{u^\tau_h}(t_i,\bx)-u^\tau_h(t_i,\bx) \} \dx=  \int_{\mathcal{D}_\tau\times\mathcal{E}^{}_h} \llbracket v^\tau_h\rrbracket \cdot\{\widehat{\boldsymbol{p}^\tau_h}\} \ds\dt+ \int_{\mathcal{D}_\tau\times\mathcal{E}^{i}_h} [\widehat{\boldsymbol{p}^\tau_h}] \cdot\{v^\tau_h \} \ds\dt .
\end{align*}
Here, we take 
\begin{subequations}
\begin{align*}
\widehat{u^\tau_h} = \{ u^\tau_h \} \quad &on \ f\in\mathcal{D}_\tau\times\mathcal{E}^i_h,\\
\widehat{u^\tau_h} =g \quad &on \ f\in\mathcal{D}_\tau\times\mathcal{E}^\partial_h,\\
\widetilde{u^\tau_h} = \{ u^\tau_h \} - \eta \llbracket u^\tau_h\rrbracket \quad &on \ f\in\mathcal{P}^i_\tau\times\mathcal{T}_h,\\
\widetilde{u^\tau_h} = u_0 \quad &on \ f\in\{t_0\}\times\mathcal{T}_h,\\
\widetilde{u^\tau_h} = u^\tau_h \quad &on \ f\in\{t_{N_t}\}\times\mathcal{T}_h,\\
\widehat{\boldsymbol{p}^\tau_h} = \{ \nabla_h u^\tau_h \} - \eta \llbracket u^\tau_h\rrbracket \quad &on \ f\in \mathcal{D}_\tau\times\mathcal{E}^i_h,\\
\widehat{\boldsymbol{p}^\tau_h} = \nabla_h u^\tau_h - \eta g\cdot\boldsymbol{n} \quad &on \ f\in \mathcal{D}_\tau\times\mathcal{E}^\partial_h,
\end{align*}
\end{subequations}
where $\eta = {\eta_f}{(h_f)}^{-1}$, and $\eta_f$ can be different by the choice of the face $f$. And we choose $\boldsymbol{q}^\tau_h = \nabla_h v^\tau_h$, then the sapce-time LRNN-DG scheme for solving the heat problem is: Find $u^\tau_h \in V_h^\tau$ such that
\begin{equation}\label{timec0for}
\begin{aligned}
B_{h\tau}(u^\tau_h,v^\tau_h) = l(v^\tau_h) \quad \forall v^\tau_h\in V^\tau_h,
\end{aligned}
\end{equation}
where
\begin{equation}
\begin{aligned}
B_{h\tau}(u^\tau_h,v^\tau_h)=
&\int_{\Sigma} \frac{\partial_\tau u^\tau_h}{\partial_\tau t}v^\tau_h \dx \dt + \int_{\Sigma} \nabla_h u^\tau_h \cdot\nabla_h v^\tau_h \dx \dt\\ 
&-\sum^{N_t -1}_{i=0}\int_{\mathcal{T}_h} \llbracket u^\tau_h(t_i,\bx)\rrbracket\cdot \{ v^\tau_h(t_i,\bx) \} \dx-\sum^{N_t -1}_{i=1} \int_{\mathcal{T}_h} \eta \llbracket u^\tau_h(t_i,\bx)\rrbracket\cdot\llbracket v^\tau_h(t_i,\bx)\rrbracket \dx\\
&- \int_{\mathcal{D}_\tau \times\mathcal{E}^{}_h} \left(\llbracket u^\tau_h\rrbracket\cdot \{ \nabla_{h}v^\tau_h \} +\llbracket v^\tau_h\rrbracket\cdot \{ \nabla_{h}u^\tau_h \} -\eta \llbracket u^\tau_h\rrbracket\cdot \llbracket v^\tau_h\rrbracket\right) \ds\dt,
\end{aligned}
\end{equation}
\begin{equation}
\begin{aligned}
l(v^\tau_h)
= \int_{\Sigma} f v^\tau_h dx dt
- \int_{ \mathcal{D}_\tau\times{\cal E}_h^{\partial}} g \boldsymbol{n}\cdot \nabla_{h} v^\tau_{h} \dt\ds + \int_{ \mathcal{D}_\tau\times{\cal E}_h^{\partial}} \eta g v^\tau_h \dt\ds
+ \int_{\mathcal{T}_h} u_0(\bx) v^\tau_{h}(t_0,\bx) \dx.
\end{aligned}
\end{equation}

From the above scheme, we can get a linear system of equations
\begin{align}
\mathbb{A}U=L,
\end{align}
where $\mathbb{A}$ is a $N_e M \times N_e M $ matrix, $L$ is a $N_e M \times 1$ vector, and there are $N_e M$ unknown variables $U = \{u^{I_i\times K}_j: I_i \in \mathcal{D}_\tau, K \in \mathcal{T}_h, j=1,2,\cdots,M\}$. Here, the width of the last hidden layer is $M$.
We look for the least-squares solution for this system. 
Therefore, the parameters ${U}$ in the neural networks' output layers are obtained by the linear least-squares computation.

\begin{rem}
Because the variables $x$ and $t$ are inputs of randomized neural networks, so the LRNN-DG scheme \eqref{timec0for} is based on a space-time approach, and we can solve this time-dependent problem by one least-squares computation, which is more efficient than traditional approaches.
\end{rem}

We give the following lemma to show the consistency of the space-time LRNN-DG method.
\begin{lemma}\label{lem:spacetimeconsis}
The space-time LRNN-DG scheme is consistent, i.e., for the solution $u\in C^0(I; H^2(\Omega))$ of the heat equation \eqref{heateq}, we have
\begin{equation}\label{consis_st}
B_{h\tau}(u,v^\tau_h)= l(v^\tau_h)\quad\forall v^\tau_h\in V^\tau_h.
\end{equation}
\end{lemma}
\begin{proof}
We know that $u\in C^0(I; H^2(\Omega))$ implies $\llbracket u\rrbracket=0$, $[ \nabla u]=0$ on $\mathcal{E}^i_h$, $u=g$ on $\mathcal{E}^\partial_h$, $\llbracket u(t_i,\bx)\rrbracket = 0$ for $i = 1, 2,\cdots,N_t-1$ and $u(t_0,\bx)=u_0(\bx)$. Then, by the identities \eqref{ibps} and \eqref{iden_dg}, we have
\begin{equation*}
\begin{aligned}
B_{h\tau}(u,v^\tau_h)=
&\int_{\Sigma} \frac{\partial u}{\partial t}v^\tau_h \dx \dt + \int_{\Sigma} \nabla u \cdot\nabla_h v^\tau_h \dx \dt - \int_{\mathcal{D}_\tau \times\mathcal{E}^{}_h} \llbracket v^\tau_h\rrbracket\cdot \{ \nabla u \} \ds\dt\\
&- \int_{ \mathcal{D}_\tau\times{\cal E}_h^{\partial}} g \boldsymbol{n}\cdot \nabla_{h} v^\tau_{h} \dt\ds + \int_{ \mathcal{D}_\tau\times{\cal E}_h^{\partial}} \eta g v^\tau_h \dt\ds + \int_{\mathcal{T}_h} u_0(\bx) v^\tau_{h}(t_0,\bx) \dx\\
=& \int_{\Sigma} (\frac{\partial u}{\partial t} - \Delta u) v^\tau_h\dx\dt +\int_{\mathcal{D}_\tau\times{\cal E}_h} \nabla u \cdot \boldsymbol{n}_K v^\tau_h\dt\ds - \int_{\mathcal{D}_\tau \times\mathcal{E}^{}_h} \llbracket v^\tau_h\rrbracket\cdot \{ \nabla u \} \ds\dt\\
&- \int_{ \mathcal{D}_\tau\times{\cal E}_h^{\partial}} g \boldsymbol{n}\cdot \nabla_{h} v^\tau_{h} \dt\ds + \int_{ \mathcal{D}_\tau\times{\cal E}_h^{\partial}} \eta g v^\tau_h \dt\ds + \int_{\mathcal{T}_h} u_0(\bx) v^\tau_{h}(t_0,\bx) \dx\\
=& l(v_h).
\end{aligned}
\end{equation*}
\end{proof} 

\subsection{Space-time LRNN-$C^0$DG method}

In the space-time LRNN-DG method, a proper value of the penalty parameter is hard to choose. To avoid this difficulty, in this subsection, we introduce the space-time LRNN-$C^0$DG method. 
 
The setting of local randomized neural networks and the partition are the same as ones stated above. Now we enforce the $C^0$-continuous condition on $f \in \left(\mathcal{D}_\tau\times\mathcal{E}^{i}_h\right)\cup\left(\mathcal{P}^{i}_\tau\times\mathcal{T}_h\right)$, the initial condition on $f\in\{t_0\}\times\mathcal{T}_h$ and Dirichlet boundary condition on $f\in\mathcal{D}_\tau\times\mathcal{E}^{\partial}_h$ to solve this problem. We add the additional equations to enforce the solution to satisfy the boundary condition \eqref{timeboundarycondi}, that is, we choose some points on boundary faces, $P^{S\partial}_h = \{(t^f_j,\bx^f_j) \in f: f\in \mathcal{D}_\tau\times\mathcal{E}^{\partial}_h,j=1,2,\cdots,N_\partial^f\}$ and $|P^{S\partial}_h| = N^S_\partial$, such that
\begin{equation}\label{timeboundarypoint}
u^\tau_h(t^f_j,\bx^f_j) = g(t^f_j,\bx^f_j)\quad \forall(t^f_j,\bx^f_j)\in P^{S\partial}_h.
\end{equation}

We add the additional equations to enforce the solution to satisfy the initial condition \eqref{timeinicondi}, that is, we choose some points at the initial time, $P^{t_0}_h=\{(t_0,\bx^f_j) \in f:f\in\{t_0\}\times\mathcal{T}_h,j=1,2,\cdots,N^f_{t_0}\}$ and $|P^{t_0}_h| = N^T_{t_0}$, such that
\begin{equation}\label{timeinipoint}
u^\tau_h(t_0,\bx^f_j) = u_0(\bx^f_j)\quad \forall(t_0,\bx^f_j)\in P^{t_0}_h.
\end{equation}

Moreover, we impose that the numerical solution $u_h$ satisfies certain $C^0$-concontinuity conditions across the interior faces $f\in \left(\mathcal{D}_\tau\times\mathcal{E}^{i}_h\right)\cup\left(\mathcal{P}^{i}_\tau\times\mathcal{T}_h\right)$ along the spatial and temporal directions. We choose some points on the boundary faces, $P^i_h=\{(t^f_j,\bx^f_j) \in f: f\in \left(\mathcal{D}_\tau\times\mathcal{E}^{i}_h\right)\cup\left(\mathcal{P}^{i}_\tau\times\mathcal{T}_h\right),j=1,2,\cdots,N^f_{in}\}$ and $|P^{i}_h| = N_{in}$, such that
\begin{equation}\label{timeinterpoint}
\llbracket u^\tau_h(t^f_j,\bx^f_j)\rrbracket = 0\quad\forall(t^f_j,\bx^f_j)\in P^{i}_h.
\end{equation}
The system of equations with respect to \eqref{timeboundarypoint}-\eqref{timeinterpoint} confirms the boundary conditions,
initial conditions and continuity conditions of interior edges, so we have
\begin{equation*}
\mathbb{A}_2 U =L_2,
\end{equation*}
where $\mathbb{A}_2$ is a $(N^S_\partial + N^T_{t_0} + N_{in}) \times N_e M$ matrix, $U$ is a $N_e M \times 1$ unknown vector, $L_2$ is a $(N^S_\partial + N^T_{t_0} + N_{in}) \times 1$ vector. This system of equations makes the jump of solution $\llbracket u^\tau_h \rrbracket \approx 0$, so the LRNN-$C^0$DG scheme is to find $u^\tau_h\in V^\tau_h$ such that
\begin{equation} \label{DisForm_ST_C0}
\begin{aligned}
    B^0_{h\tau}(u^\tau_h,v^\tau_h) &= l^0(v^\tau_h)\quad &\forall\, v^\tau_h \in V^\tau_h,\\
    \llbracket u^\tau_h(t^{f}_j,\bx^{f}_j )\rrbracket &= 0\quad &\forall (t^{f}_j,\bx^f_j)\in P^i_h,\\
    u^\tau_h(t^{f}_j,\bx^f_j) &= g(t^{f}_j,\bx^f_j)\quad &\forall (t^{f}_j,\bx^f_j)\in P^{S\partial}_h,\\
    u^\tau_h(t_0,\bx^f_j) &= u_0(\bx^f_j)\quad &\forall (t_0,\bx^f_j)\in P^{t_0}_h,\\
    \end{aligned}
\end{equation}
  where
\begin{align}
{B^0_{h\tau}}(u^\tau_h,v^\tau_h) =&\int_{\Sigma} \frac{\partial_\tau u^\tau_h}{\partial_\tau t} v^\tau_h \dx \dt + \int_{\Sigma} \nabla_h u^\tau_h \cdot\nabla_h v^\tau_h \dx \dt\nonumber \\
&- \int_{\mathcal{D}_\tau \times\mathcal{E}_h} ( \llbracket u^\tau_h\rrbracket\cdot \{ \nabla_{h}v^\tau_h \} +\llbracket v^\tau_h\rrbracket\cdot \{ \nabla_{h}u^\tau_h \}) \ds\dt,\label{timea0h}\\
{l^0}(v^\tau_h)= &\int_{\Sigma} f v^\tau_h dx dt
- \int_{ \mathcal{D}_\tau\times{\cal E}_h^{\partial}} g \boldsymbol{n}\cdot \nabla_{h} v^\tau_{h} \dt\ds.\nonumber
\end{align}
   
Then we can get global stiffness matrix $\mathbb{A}_1$ and the right-hand side $L_1$ of \eqref{timec0for}, where $\mathbb{A}_1$ is a $N_e M \times N_e M$ matrix, $L_1$ is a $N_e M \times 1$ vector. Combine $\mathbb{A}_1$, $\mathbb{A}_2$, $L_1$ and $L_2$, we get a linear system
\begin{equation}
\begin{bmatrix} \mathbb{A}_1 \\ \mathbb{A}_2 \end{bmatrix} U = \begin{bmatrix} L_1 \\ L_2 \end{bmatrix}.
\end{equation}
We look for the least-squares solution to this system. 
After the weights of the output layer in each local neural network are obtained by the linear least-squares computation, we can get all the values of the problem \eqref{heateq} in the domain $\Sigma$.
Note that this scheme is free of the penalty parameter.

\begin{rem}
From the mumerical examples, we see that the scheme \eqref{DisForm_ST_C0} has a good performance. But if we get rid of the term $- \int_{\mathcal{D}_\tau \times\mathcal{E}_h} \llbracket v^\tau_h\rrbracket\cdot \{ \nabla_{h}u^\tau_h \} \ds\dt$ in \eqref{timea0h}, the new scheme still work well in numerical experiments, that is
\begin{equation} \label{DisForm_ST_C02}
    \widetilde{B^0_{h\tau}}(u^\tau_h,v^\tau_h) = \int_{\Sigma} f v^\tau_h\dx\ds\quad \forall\, v^\tau_h \in V^\tau_h,
\end{equation}
  where
 \begin{equation}
\begin{aligned}\label{biform_ST_C02}
\widetilde{B^0_{h\tau}}(u^\tau_h,v^\tau_h)=
\int_{\Sigma} \frac{\partial_\tau u^\tau_h}{\partial_\tau t} v^\tau_h \dx \dt + \int_{\Sigma} \nabla_h u^\tau_h \cdot\nabla_h v^\tau_h \dx \dt - \int_{\mathcal{D}_\tau \times\mathcal{E}_h} \llbracket v^\tau_h\rrbracket\cdot \{ \nabla_{h}u^\tau_h \} \ds\dt.
\end{aligned}
\end{equation}
\end{rem}
\subsection{Space-time LRNN-$C^1$DG method}
We state the basic idea of LRNN-$C^1$DG in Section \ref{lrnnc1dgmethod}. Similarly, we introduce the space-time LRNN-$C^1$DG method for solving the heat equation.

We add the system of equations to enforce the solution $u^\tau_h$ to satisfy the boundary condition, the initial condition, and $C^0$-continuity conditions along with spatial and temporal directions just like \eqref{timeboundarypoint}-\eqref{timeinterpoint}. Moreover, we impose that the numerical solution $u^\tau_h$ satisfies certain $C^1$-continuity conditions across the interior faces $f\in\mathcal{D}_\tau\times\mathcal{E}^{i}_h$ along the spatial direction, choose points $P^{Si}_h=\{(t^f_j, \bx^f_j)\in f: f \in \mathcal{D}_\tau\times\mathcal{E}^{i}_h,j=1,2,\cdots,N^f_{in}\}$ and $|P^{Si}_h|=N^S_{in}$,
\begin{equation}\label{timeinterpoint_c1}
[ \nabla_h u^\tau_h(\bx^{f}_j )] = 0\quad \forall \bx^f_j\in P^{Si}_h.
\end{equation}

So we have the new problem that finding $u_h\in V^\tau_h$ such that 
\begin{equation} \label{DisForm_ST_C1}
    B_{h\tau}^\sigma(u^\tau_h,v^\tau_h) = \int_{\sigma} f v^\tau_h \dt\dx\quad \forall\, \sigma = \left(I_i\times K\right) \in \left(\mathcal{D}_\tau\times\mathcal{T}_h\right),
\end{equation}
  where
\begin{align}\label{biform_ST_C1}
    B^\sigma_{h\tau}(u^\tau_h,v^\tau_h) =\int_{\sigma} \frac{\partial u^\tau_h}{\partial t} v^\tau_h \dt\dx + \int_{\sigma} \nabla u^\tau_h \cdot \nabla v^\tau_h \dt\dx - \int_{I_i\times \partial K} \nabla u^\tau_h \cdot \boldsymbol{n} v^\tau_h \ds\dt.
\end{align}
Then we can get the global stiffness matrix $ \mathbb{A}_1$ and the right-hand side $L_1$, where $ \mathbb{A}_1$ is a $N_e M \times N_e M$ matrix, $L_1$ is a $N_e M \times 1$ vector, such that $ \mathbb{A}_1 U = L_1$.

Let $u^\tau_h$ satisfy the conditions \eqref{timeboundarypoint}-\eqref{timeinterpoint}, $C^1$-continuity condition \eqref{timeinterpoint_c1}. Then we can obtain a system of equations
\begin{equation}
 \mathbb{A}_2 U = L_2,
\end{equation}
where $ \mathbb{A}_2$ is a $(N^T_\partial +N^S_\partial + N_{in} + N^S_{in}) \times N_e M$ matrix, $U$ is a $N_e M \times 1$ matrix of unknown variables, $L_2$ is a $(N^T_\partial +N^S_\partial + N_{in} + N^S_{in}) \times 1$ vector.
Combine $ \mathbb{A}_1$, $ \mathbb{A}_2$, $L_1$ and $L_2$, we obtain
\begin{equation}
\begin{bmatrix} \mathbb{A}_1 \\ \mathbb{A}_2 \end{bmatrix} U = \begin{bmatrix} L_1 \\ L_2 \end{bmatrix}.
\end{equation}

We look for the least-squares solution to this system. After the weights of the output layer in each local neural network are obtained by the linear least-squares computation, we can get all the values of the problem \eqref{heateq} in the domain $\Sigma$.

\section{Numerical Examples}\label{sec:numericalexample}

In this section, we present several test problems to demonstrate the performance of
the methods developed herein.
In these examples, the neural network is implemented using the Tensorflow and Keras libraries in Python,  as stated in Section \ref{rnn}. Each local neural network (for each sub-domain) consists of a single hidden layer, whose parameters are pre-assigned and fixed to uniform random values generated from $[-w_0, w_0]$, where $w_0$ is a constant. The overall neural network is composed of all the local neural networks, which are coupled with one another through the DG formulation or the $C^0/C^1$ conditions. The integrals in the weak formulations are computed by the Gaussian quadrature.
We employ 70 quadrature points for the integrals in 1-d, and $70\times70$ quadrature points for the integrals in 2-d. For solving the rectangular system of equations about the output-layer coefficients, we employ the linear least-squares routine from LAPACK, available through wrapper functions in the scipy package in Python. The ${\rm DoF}_K$ or ${\rm DoF}_\sigma$ in all tables below denote the degrees of freedom on each subdomain, and the DoF in the figures refers to the total degrees of freedom.

\begin{example}[One-Dimensional Helmholtz Equation]
\label{1dpoisson}
The first test problem is a one-dimensional Helmholtz equation on the domain $\Omega = [0,1]$,
\begin{equation}
\begin{aligned}
u_{xx}-\lambda u &= f(x),\\
u(0)&= \sin(0.4\pi)\cos(0.4\pi),\\
u(1)&= \sin(4.4\pi)\cos(4.4\pi), \nonumber
\end{aligned}
\end{equation}
where the $\lambda$ is a constant and $f(x)$ is a prescribed source term,
with the manufactured exact solution
\begin{equation}
u(x)= \sin(4\pi(x+0.1))\cos(4\pi(x+0.1)).\nonumber
\end{equation}
\end{example}

\begin{table}[h]
	\centering
\begin{tabular}{|l|l|ll|ll|ll|}
\hline
\multirow{2}{*}{parameter}    & $h$    & \multicolumn{2}{c|}{$2^{-2}$}                                & \multicolumn{2}{c|}{$2^{-3}$}                                & \multicolumn{2}{c|}{$2^{-4}$}                                \\ \cline{2-8} 
                              & \diagbox[width=5em,trim=l] {${\rm DoF}_K$}{Norm}  & \multicolumn{1}{c|}{$L^{2}$}  & \multicolumn{1}{c|}{$H^{1}$} & \multicolumn{1}{c|}{$L^{2}$}  & \multicolumn{1}{c|}{$H^{1}$} & \multicolumn{1}{c|}{$L^{2}$}  & \multicolumn{1}{c|}{$H^{1}$} \\ \hline
\multirow{8}{*}{$\lambda=10$} & 10   & \multicolumn{1}{l|}{1.63E-02} & 1.23E-00                     & \multicolumn{1}{l|}{4.31E-02} & 4.72E-00                     & \multicolumn{1}{l|}{3.33E-03} & 5.15E-01                     \\ \cline{2-8} 
                              & 20   & \multicolumn{1}{l|}{3.99E-05} & 6.81E-03                     & \multicolumn{1}{l|}{3.22E-05} & 9.27E-03                     & \multicolumn{1}{l|}{1.81E-06} & 1.23E-03                     \\ \cline{2-8} 
                              & 40   & \multicolumn{1}{l|}{1.65E-07} & 5.83E-05                     & \multicolumn{1}{l|}{3.28E-08} & 2.03E-05                     & \multicolumn{1}{l|}{1.55E-08} & 1.64E-05                     \\ \cline{2-8} 
                              & 80   & \multicolumn{1}{l|}{1.33E-08} & 4.48E-06                     & \multicolumn{1}{l|}{2.67E-09} & 1.80E-06                     & \multicolumn{1}{l|}{2.84E-10} & 4.29E-07                     \\ \cline{2-8} 
                              & 160  & \multicolumn{1}{l|}{2.55E-09} & 1.16E-06                     & \multicolumn{1}{l|}{5.34E-10} & 4.99E-07                     & \multicolumn{1}{l|}{7.33E-11} & 1.37E-07                     \\ \cline{2-8} 
                              & 320  & \multicolumn{1}{l|}{1.96E-09} & 9.54E-07                     & \multicolumn{1}{l|}{2.69E-10} & 2.69E-07                     & \multicolumn{1}{l|}{5.20E-11} & 1.01E-07                     \\ \cline{2-8} 
                              & 640  & \multicolumn{1}{l|}{1.37E-09} & 6.68E-07                     & \multicolumn{1}{l|}{1.94E-10} & 1.96E-07                     & \multicolumn{1}{l|}{4.07E-11} & 8.34E-08                     \\ \cline{2-8} 
                              & 1280 & \multicolumn{1}{l|}{9.61E-10} & 4.79E-07                     & \multicolumn{1}{l|}{1.94E-10} & 2.02E-07                     & \multicolumn{1}{l|}{3.30E-11} & 7.10E-08                     \\ \hline
\multirow{8}{*}{$\lambda=1$}  & 10   & \multicolumn{1}{l|}{1.45E-02} & 1.24E-00                     & \multicolumn{1}{l|}{2.26E-02} & 1.98E-00                     & \multicolumn{1}{l|}{1.57E-03} & 3.08E-01                     \\ \cline{2-8} 
                              & 20   & \multicolumn{1}{l|}{4.00E-05} & 8.60E-03                     & \multicolumn{1}{l|}{3.93E-05} & 1.22E-02                     & \multicolumn{1}{l|}{1.26E-06} & 8.23E-04                     \\ \cline{2-8} 
                              & 40   & \multicolumn{1}{l|}{1.89E-07} & 6.41E-05                     & \multicolumn{1}{l|}{2.20E-07} & 1.11E-04                     & \multicolumn{1}{l|}{5.92E-09} & 6.52E-06                     \\ \cline{2-8} 
                              & 80   & \multicolumn{1}{l|}{1.03E-08} & 3.68E-06                     & \multicolumn{1}{l|}{1.14E-09} & 8.20E-07                     & \multicolumn{1}{l|}{2.22E-10} & 3.13E-07                     \\ \cline{2-8} 
                              & 160  & \multicolumn{1}{l|}{1.98E-09} & 8.13E-07                     & \multicolumn{1}{l|}{4.55E-10} & 3.79E-07                     & \multicolumn{1}{l|}{1.16E-10} & 1.95E-07                     \\ \cline{2-8} 
                              & 320  & \multicolumn{1}{l|}{1.72E-09} & 7.89E-07                     & \multicolumn{1}{l|}{3.27E-10} & 2.89E-07                     & \multicolumn{1}{l|}{5.09E-11} & 8.83E-08                     \\ \cline{2-8} 
                              & 640  & \multicolumn{1}{l|}{1.10E-09} & 4.94E-07                     & \multicolumn{1}{l|}{2.66E-10} & 2.35E-07                     & \multicolumn{1}{l|}{6.58E-11} & 1.16E-07                     \\ \cline{2-8} 
                              & 1280 & \multicolumn{1}{l|}{1.26E-09} & 5.49E-07                     & \multicolumn{1}{l|}{2.32E-10} & 2.11E-07                     & \multicolumn{1}{l|}{4.92E-11} & 9.27E-08                     \\ \hline
\end{tabular}
\caption{Errors of the LRNN-DG method for 1-d Helmholtz equation in Example \ref{1dpoisson}.}
\label{tablelrnndg}
\end{table}

\begin{table}[h]
	\centering
\begin{tabular}{|l|l|ll|ll|ll|}
\hline
\multirow{2}{*}{parameter}    & $h$    & \multicolumn{2}{c|}{$2^{-2}$}                                & \multicolumn{2}{c|}{$2^{-3}$}                                & \multicolumn{2}{c|}{$2^{-4}$}                                \\ \cline{2-8} 
                              & \diagbox[width=5em,trim=l] {${\rm DoF}_K$}{Norm}  & \multicolumn{1}{c|}{$L^{2}$}  & \multicolumn{1}{c|}{$H^{1}$} & \multicolumn{1}{c|}{$L^{2}$}  & \multicolumn{1}{c|}{$H^{1}$} & \multicolumn{1}{c|}{$L^{2}$}  & \multicolumn{1}{c|}{$H^{1}$} \\ \hline
\multirow{8}{*}{$\lambda=10$} & 10   & \multicolumn{1}{l|}{1.91E-02} & 1.47E-00                     & \multicolumn{1}{l|}{2.24E-02} & 2.45E-00                     & \multicolumn{1}{l|}{3.83E-03} & 5.42E-01                     \\ \cline{2-8} 
                              & 20   & \multicolumn{1}{l|}{6.60E-05} & 1.37E-02                     & \multicolumn{1}{l|}{2.15E-05} & 5.39E-03                     & \multicolumn{1}{l|}{9.20E-06} & 6.43E-03                     \\ \cline{2-8} 
                              & 40   & \multicolumn{1}{l|}{1.30E-07} & 3.48E-05                     & \multicolumn{1}{l|}{4.56E-08} & 2.82E-05                     & \multicolumn{1}{l|}{1.46E-08} & 1.65E-05                     \\ \cline{2-8} 
                              & 80   & \multicolumn{1}{l|}{5.22E-09} & 2.04E-06                     & \multicolumn{1}{l|}{7.38E-10} & 5.49E-07                     & \multicolumn{1}{l|}{1.48E-10} & 2.16E-07                     \\ \cline{2-8} 
                              & 160  & \multicolumn{1}{l|}{5.13E-10} & 2.35E-07                     & \multicolumn{1}{l|}{1.21E-10} & 1.15E-07                     & \multicolumn{1}{l|}{5.32E-11} & 9.81E-08                     \\ \cline{2-8} 
                              & 320  & \multicolumn{1}{l|}{4.25E-10} & 2.11E-07                     & \multicolumn{1}{l|}{8.85E-11} & 8.93E-08                     & \multicolumn{1}{l|}{3.12E-11} & 6.35E-08                     \\ \cline{2-8} 
                              & 640  & \multicolumn{1}{l|}{2.65E-10} & 1.40E-07                     & \multicolumn{1}{l|}{7.32E-11} & 7.86E-08                     & \multicolumn{1}{l|}{2.60E-11} & 5.44E-08                     \\ \cline{2-8} 
                              & 1280 & \multicolumn{1}{l|}{3.12E-10} & 1.63E-07                     & \multicolumn{1}{l|}{6.48E-11} & 6.90E-08                     & \multicolumn{1}{l|}{2.24E-11} & 5.03E-08                     \\ \hline
\multirow{8}{*}{$\lambda=1$}  & 10   & \multicolumn{1}{l|}{1.97E-02} & 1.47E-00                     & \multicolumn{1}{l|}{3.81E-02} & 3.58E+00                     & \multicolumn{1}{l|}{4.21E-03} & 5.42E-01                     \\ \cline{2-8} 
                              & 20   & \multicolumn{1}{l|}{6.54E-05} & 1.36E-02                     & \multicolumn{1}{l|}{2.18E-05} & 5.45E-03                     & \multicolumn{1}{l|}{9.24E-06} & 6.45E-03                     \\ \cline{2-8} 
                              & 40   & \multicolumn{1}{l|}{1.31E-07} & 3.51E-05                     & \multicolumn{1}{l|}{4.60E-08} & 2.83E-05                     & \multicolumn{1}{l|}{1.46E-08} & 1.65E-05                     \\ \cline{2-8} 
                              & 80   & \multicolumn{1}{l|}{5.32E-09} & 2.07E-06                     & \multicolumn{1}{l|}{7.82E-10} & 5.94E-07                     & \multicolumn{1}{l|}{1.46E-10} & 2.14E-07                     \\ \cline{2-8} 
                              & 160  & \multicolumn{1}{l|}{4.94E-10} & 2.23E-07                     & \multicolumn{1}{l|}{9.95E-11} & 9.28E-08                     & \multicolumn{1}{l|}{5.28E-11} & 9.77E-08                     \\ \cline{2-8} 
                              & 320  & \multicolumn{1}{l|}{4.16E-10} & 2.08E-07                     & \multicolumn{1}{l|}{8.70E-11} & 8.82E-08                     & \multicolumn{1}{l|}{2.95E-11} & 6.05E-08                     \\ \cline{2-8} 
                              & 640  & \multicolumn{1}{l|}{2.67E-10} & 1.38E-07                     & \multicolumn{1}{l|}{6.47E-11} & 7.01E-08                     & \multicolumn{1}{l|}{2.68E-11} & 5.63E-08                     \\ \cline{2-8} 
                              & 1280 & \multicolumn{1}{l|}{3.16E-10} & 1.65E-07                     & \multicolumn{1}{l|}{7.07E-11} & 7.62E-08                     & \multicolumn{1}{l|}{1.90E-11} & 4.24E-08                     \\ \hline
\end{tabular}
\caption{Errors of the LRNN-$C^0$DG method for 1-d Helmholtz equation in Example \ref{1dpoisson}.}
\label{tablelc0rnndg}
\end{table}

\begin{table}[h]
	\centering
\begin{tabular}{|l|l|ll|ll|ll|}
\hline
\multirow{2}{*}{parameter}    & $h$    & \multicolumn{2}{c|}{$2^{-2}$}                                & \multicolumn{2}{c|}{$2^{-3}$}                                & \multicolumn{2}{c|}{$2^{-4}$}                                \\ \cline{2-8} 
                              & \diagbox[width=5em,trim=l] {${\rm DoF}_K$}{Norm}  & \multicolumn{1}{c|}{$L^{2}$}  & \multicolumn{1}{c|}{$H^{1}$} & \multicolumn{1}{c|}{$L^{2}$}  & \multicolumn{1}{c|}{$H^{1}$} & \multicolumn{1}{c|}{$L^{2}$}  & \multicolumn{1}{c|}{$H^{1}$} \\ \hline
\multirow{7}{*}{$\lambda=10$} & 20   & \multicolumn{1}{l|}{7.48E-05} & 1.23E-02                     & \multicolumn{1}{l|}{3.79E-06} & 1.39E-03                     & \multicolumn{1}{l|}{1.10E-06} & 7.26E-04                     \\ \cline{2-8} 
                              & 40   & \multicolumn{1}{l|}{1.34E-07} & 3.30E-05                     & \multicolumn{1}{l|}{2.94E-08} & 1.64E-05                     & \multicolumn{1}{l|}{2.37E-08} & 2.39E-05                     \\ \cline{2-8} 
                              & 80   & \multicolumn{1}{l|}{7.79E-09} & 2.91E-06                     & \multicolumn{1}{l|}{4.50E-10} & 3.39E-07                     & \multicolumn{1}{l|}{1.80E-10} & 2.42E-07                     \\ \cline{2-8} 
                              & 160  & \multicolumn{1}{l|}{6.59E-10} & 3.07E-07                     & \multicolumn{1}{l|}{9.30E-11} & 8.81E-08                     & \multicolumn{1}{l|}{4.51E-11} & 8.16E-08                     \\ \cline{2-8} 
                              & 320  & \multicolumn{1}{l|}{2.45E-10} & 1.20E-07                     & \multicolumn{1}{l|}{8.31E-11} & 8.07E-08                     & \multicolumn{1}{l|}{2.04E-11} & 4.14E-08                     \\ \cline{2-8} 
                              & 640  & \multicolumn{1}{l|}{1.99E-10} & 1.01E-07                     & \multicolumn{1}{l|}{6.52E-11} & 6.88E-08                     & \multicolumn{1}{l|}{1.30E-11} & 2.84E-08                     \\ \cline{2-8} 
                              & 1280 & \multicolumn{1}{l|}{1.02E-10} & 5.33E-08                     & \multicolumn{1}{l|}{5.82E-11} & 6.07E-08                     & \multicolumn{1}{l|}{1.67E-11} & 3.49E-08                     \\ \hline
\multirow{7}{*}{$\lambda=1$}  & 20   & \multicolumn{1}{l|}{1.10E-05} & 1.65E-03                     & \multicolumn{1}{l|}{1.71E-06} & 6.00E-04                     & \multicolumn{1}{l|}{1.84E-06} & 9.91E-04                     \\ \cline{2-8} 
                              & 40   & \multicolumn{1}{l|}{3.37E-08} & 8.38E-06                     & \multicolumn{1}{l|}{3.06E-08} & 1.58E-05                     & \multicolumn{1}{l|}{1.39E-09} & 1.55E-06                     \\ \cline{2-8} 
                              & 80   & \multicolumn{1}{l|}{6.75E-09} & 2.20E-06                     & \multicolumn{1}{l|}{3.84E-10} & 2.61E-07                     & \multicolumn{1}{l|}{1.44E-10} & 1.94E-07                     \\ \cline{2-8} 
                              & 160  & \multicolumn{1}{l|}{7.99E-10} & 2.98E-07                     & \multicolumn{1}{l|}{2.03E-10} & 1.60E-07                     & \multicolumn{1}{l|}{4.44E-11} & 6.99E-08                     \\ \cline{2-8} 
                              & 320  & \multicolumn{1}{l|}{6.20E-10} & 2.59E-07                     & \multicolumn{1}{l|}{1.10E-10} & 8.98E-08                     & \multicolumn{1}{l|}{2.63E-11} & 4.35E-08                     \\ \cline{2-8} 
                              & 640  & \multicolumn{1}{l|}{2.62E-10} & 1.11E-07                     & \multicolumn{1}{l|}{9.73E-11} & 8.11E-08                     & \multicolumn{1}{l|}{2.19E-11} & 3.83E-08                     \\ \cline{2-8} 
                              & 1280 & \multicolumn{1}{l|}{1.80E-10} & 7.86E-08                     & \multicolumn{1}{l|}{8.20E-11} & 7.23E-08                     & \multicolumn{1}{l|}{1.64E-11} & 2.90E-08                     \\ \hline
\end{tabular}
\caption{Errors of the LRNN-$C^1$DG method for 1-d Helmholtz equation in Example \ref{1dpoisson}.}
\label{tablelrnnc1dg}
\end{table}

\begin{figure}[!ht]  
    \centering   
    
    \begin{subfigure}{0.34\textwidth}
            \centering          
            \includegraphics[width=\textwidth]{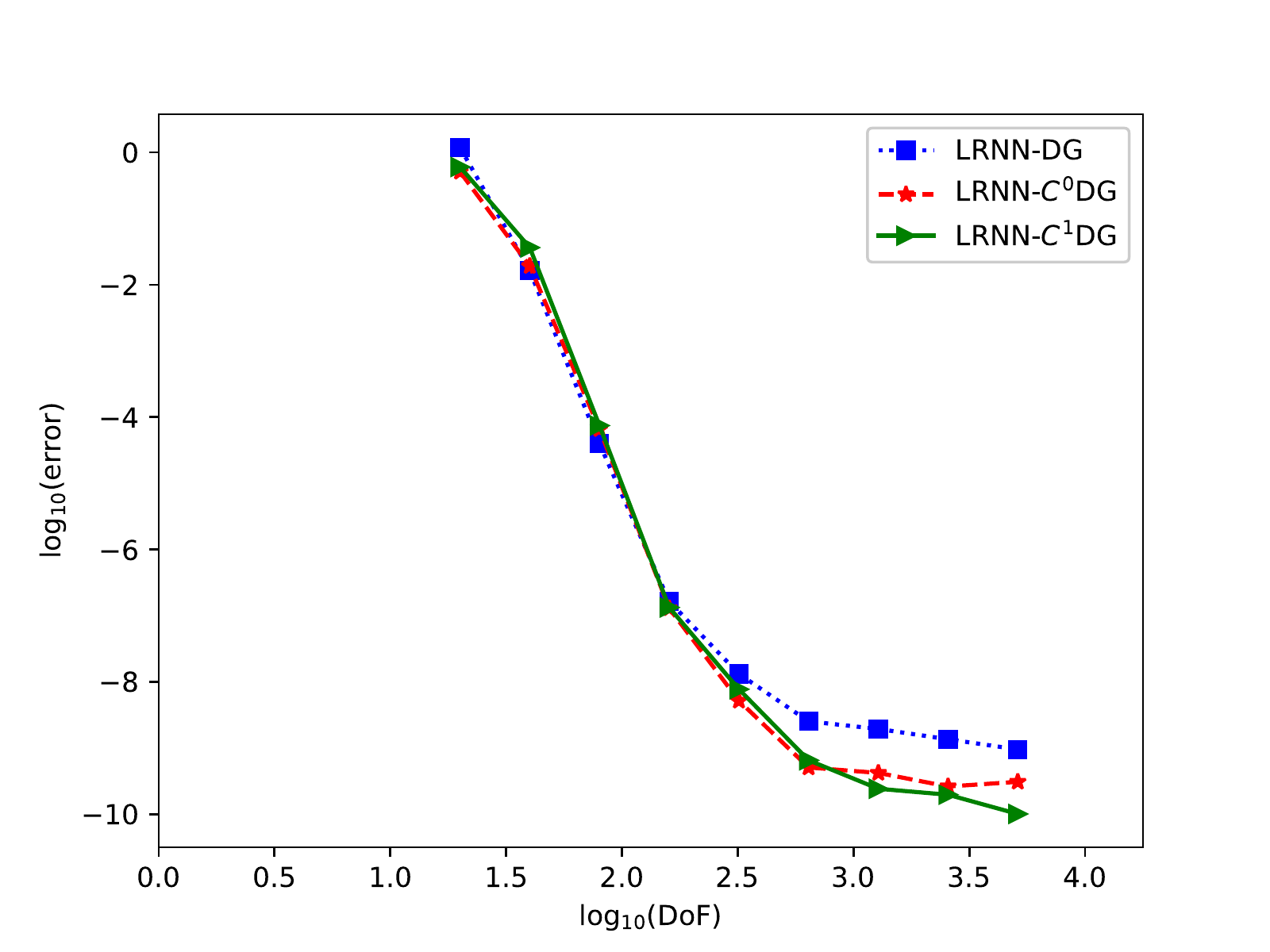}   
            \caption{$L^2$ errors with $h = 2^{-2}$}
            \label{fig:a}
    \end{subfigure}
    \begin{subfigure}{0.34\textwidth}
            \centering          
            \includegraphics[width=\textwidth]{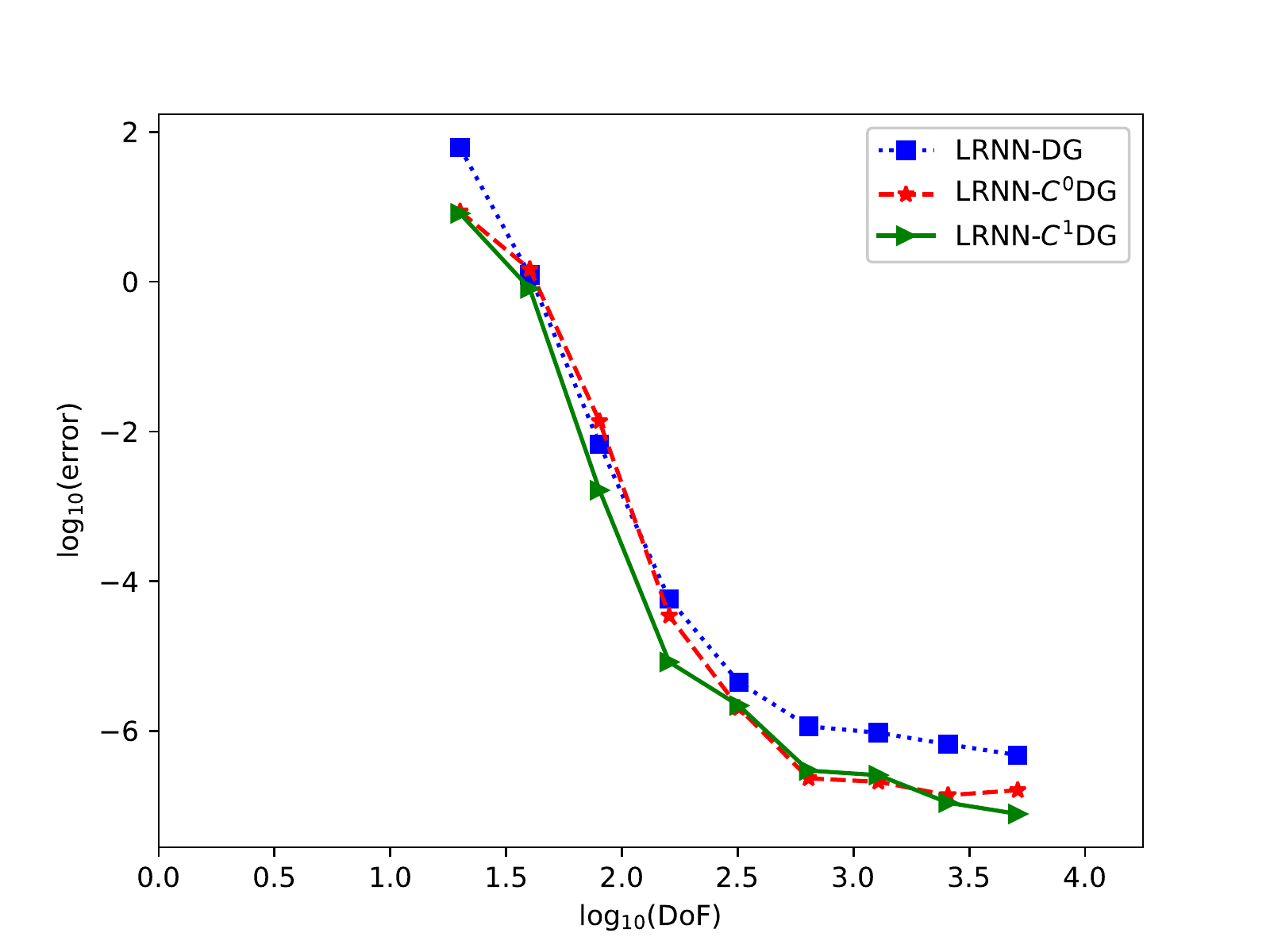}   
            \caption{$H^1$ errors with $h = 2^{-2}$}
            \label{fig:a}
    \end{subfigure}\\
    \begin{subfigure}{0.34\textwidth}
            \centering       
            \includegraphics[width=\textwidth]{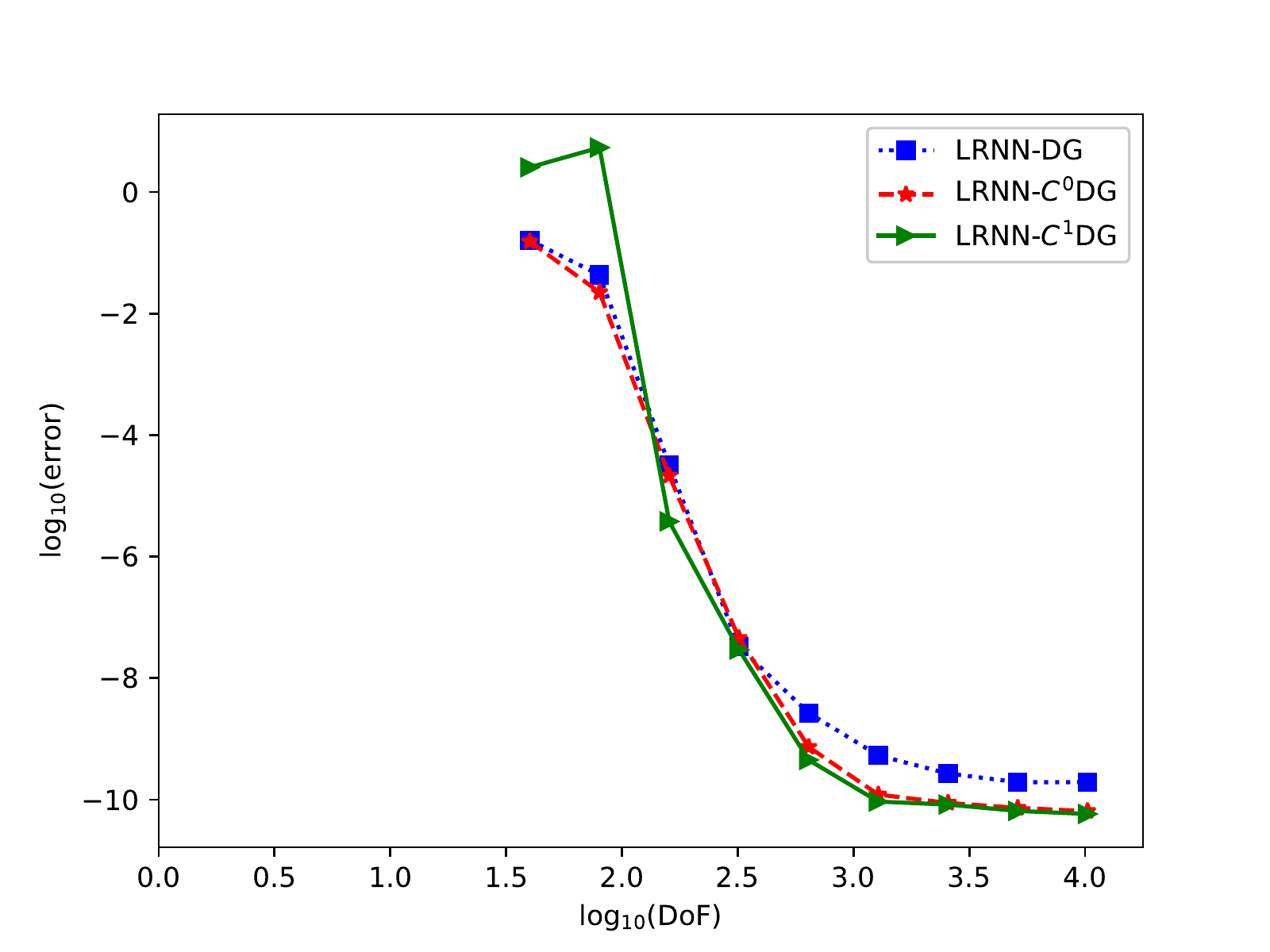}   
		 \caption{$L^2$ errors with $h = 2^{-3}$}
            \label{fig:b}
        \end{subfigure} 
        \begin{subfigure}{0.34\textwidth}
            \centering           
            \includegraphics[width=\textwidth]{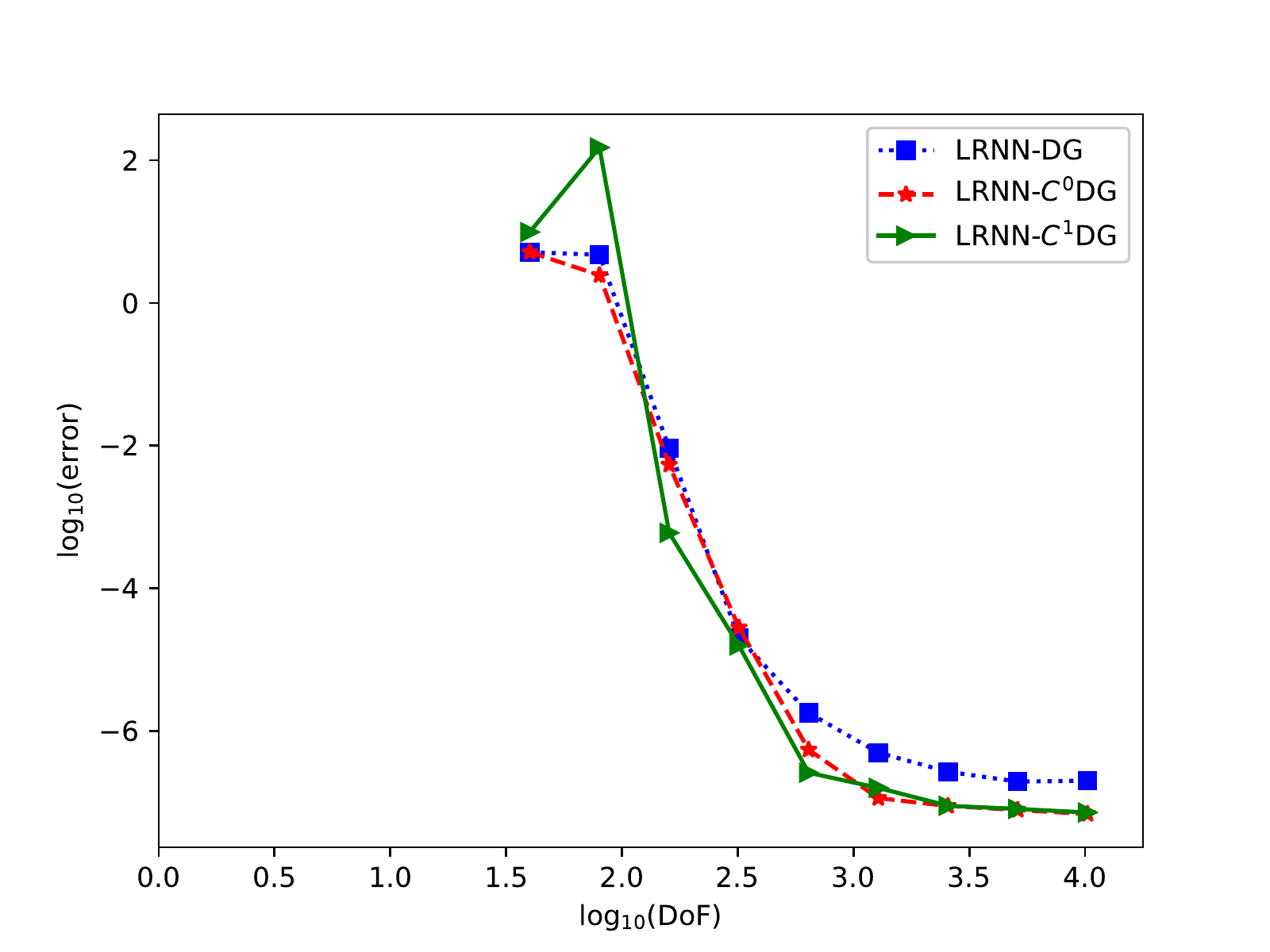}   
            \caption{$H^1$ errors with $h = 2^{-3}$}
            \label{fig:a}
    \end{subfigure}  
    \caption{Comparison of the errors obtained by the three methods in Example \ref{1dpoisson} for $\lambda$=10.} 
    \label{1dpoissonlambda10}
\end{figure}

\begin{figure}[htb]  
    \centering    
    
    \begin{subfigure}{0.34\textwidth}
            \centering           
            \includegraphics[width=\textwidth]{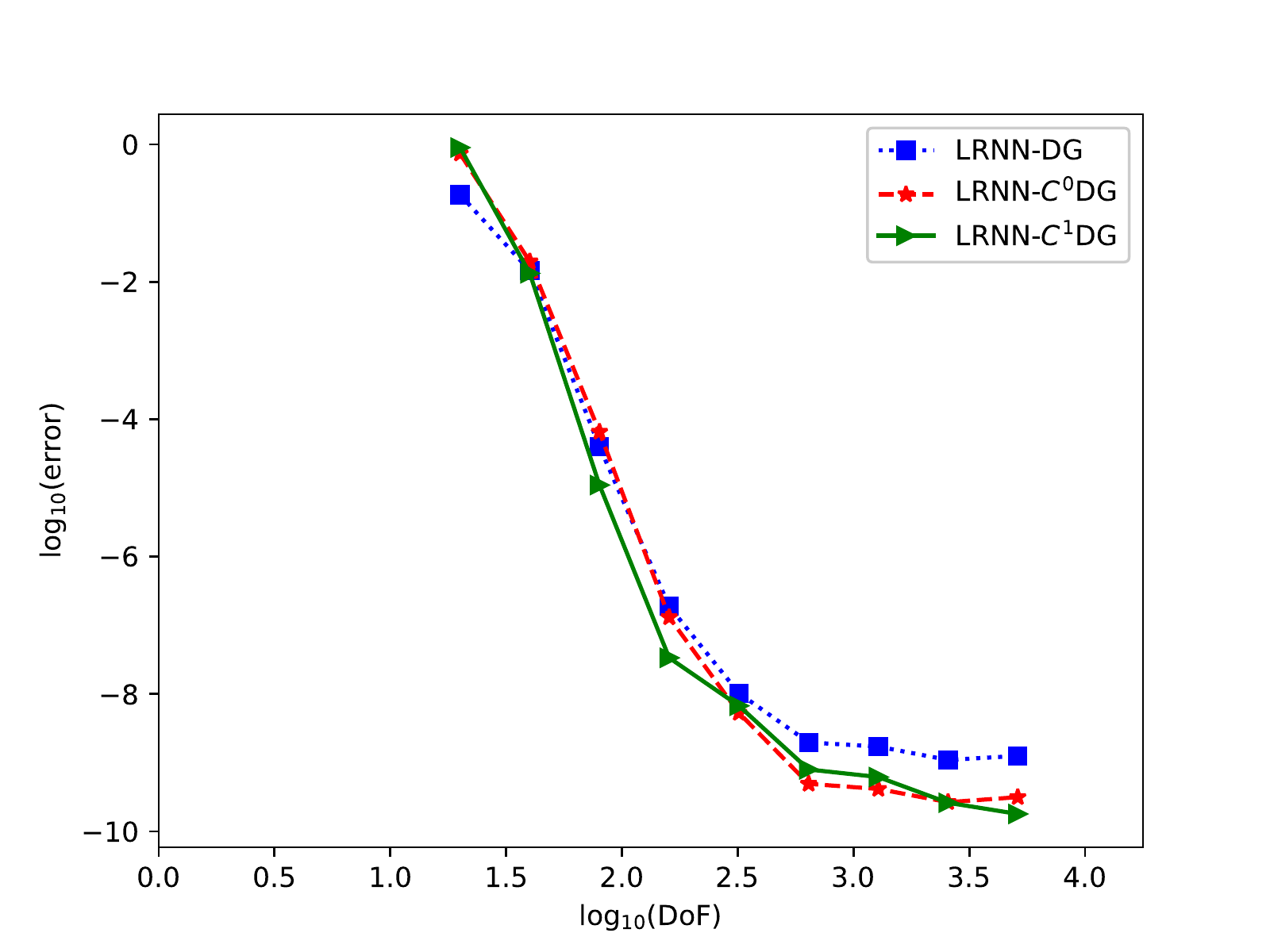}   
            \caption{$L^2$ errors with $h = 2^{-2}$}
            \label{fig:a}
    \end{subfigure}
    \begin{subfigure}{0.34\textwidth}
            \centering           
            \includegraphics[width=\textwidth]{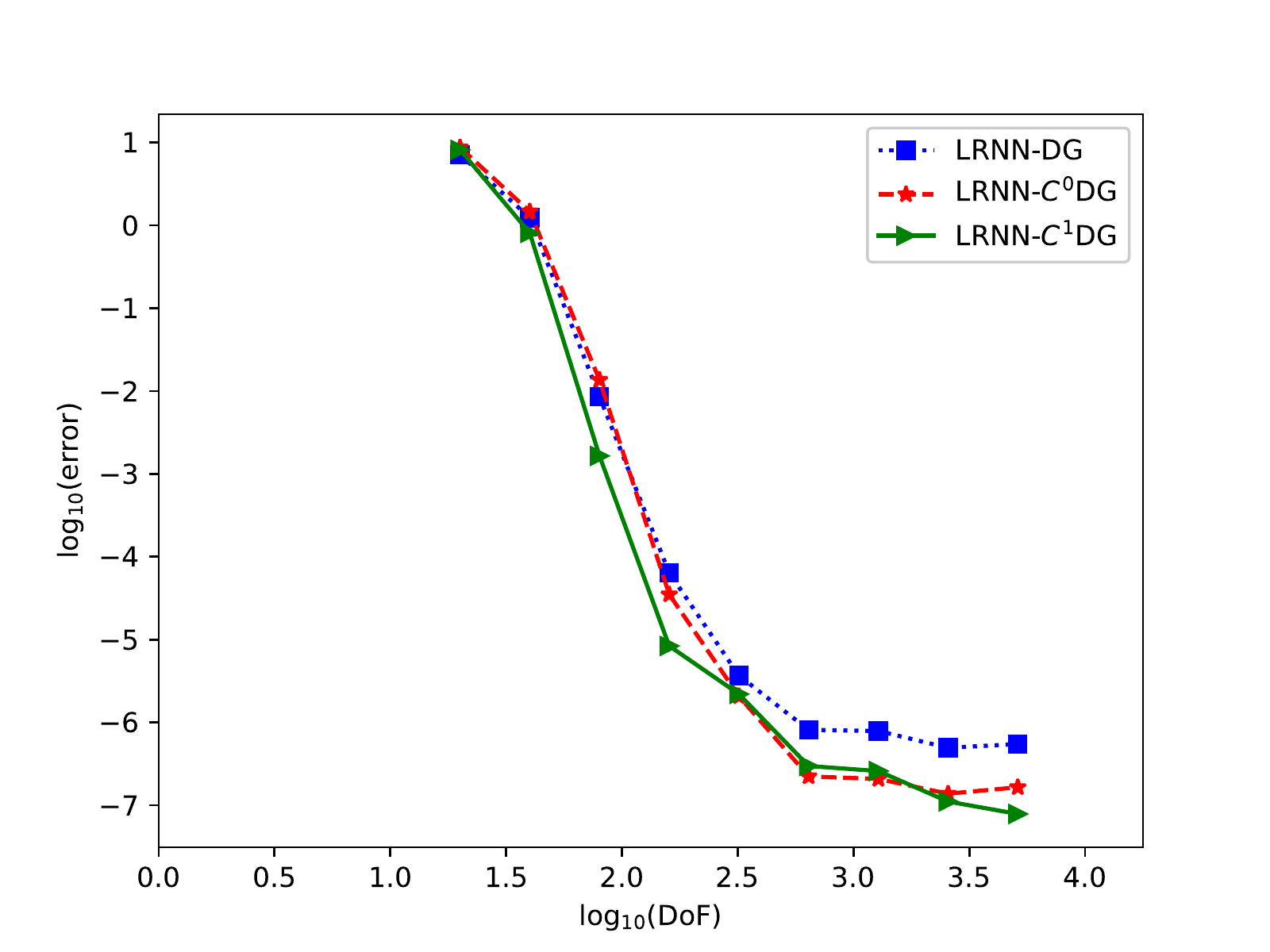}   
            \caption{$H^1$ errors with $h = 2^{-2}$}
            \label{fig:a}
    \end{subfigure}\\
    \begin{subfigure}{0.34\textwidth}
            \centering       
            \includegraphics[width=\textwidth]{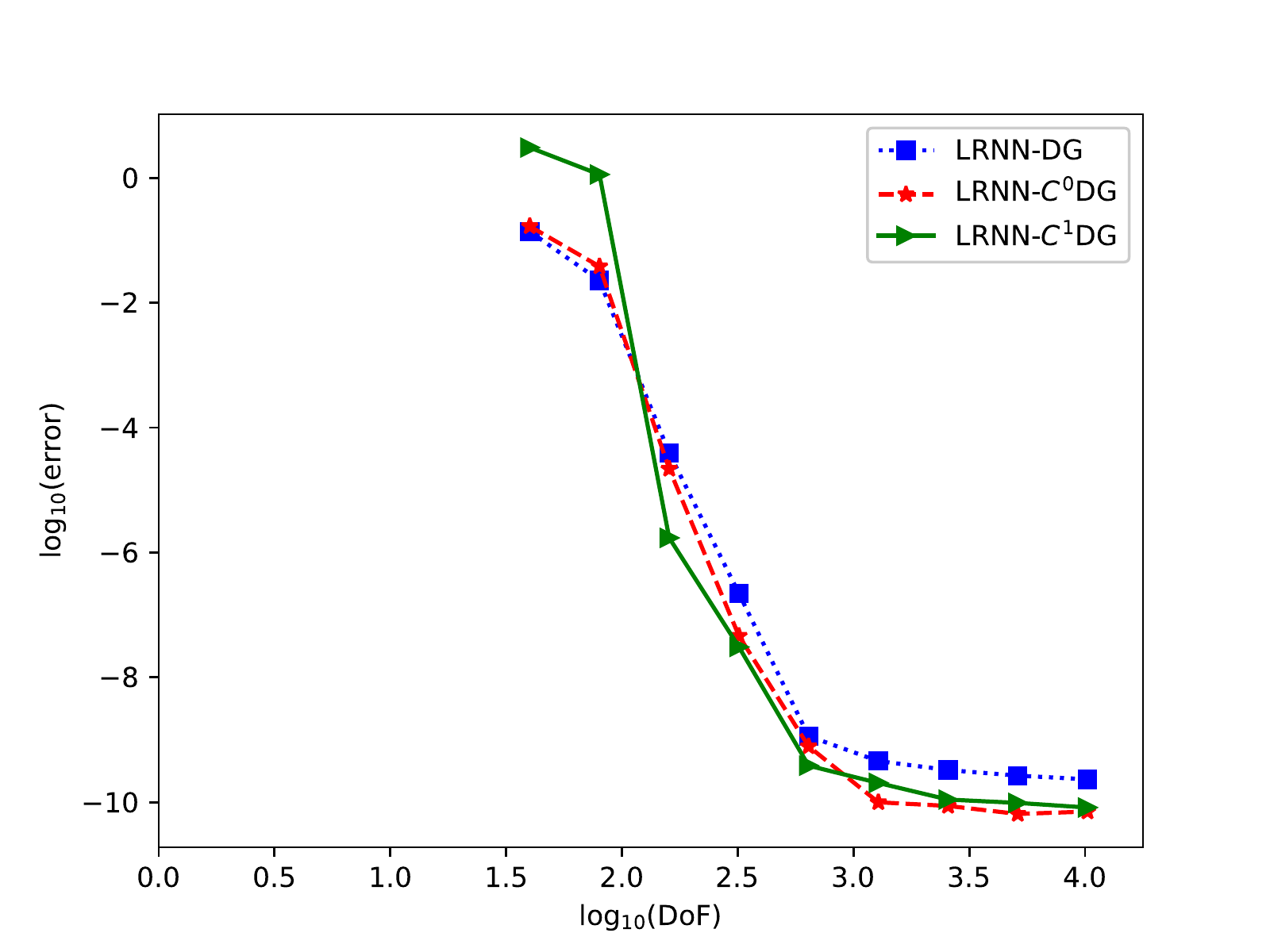}   
		 \caption{$L^2$ errors with $h = 2^{-3}$}
            \label{fig:b}
        \end{subfigure} 
        \begin{subfigure}{0.34\textwidth}
            \centering           
            \includegraphics[width=\textwidth]{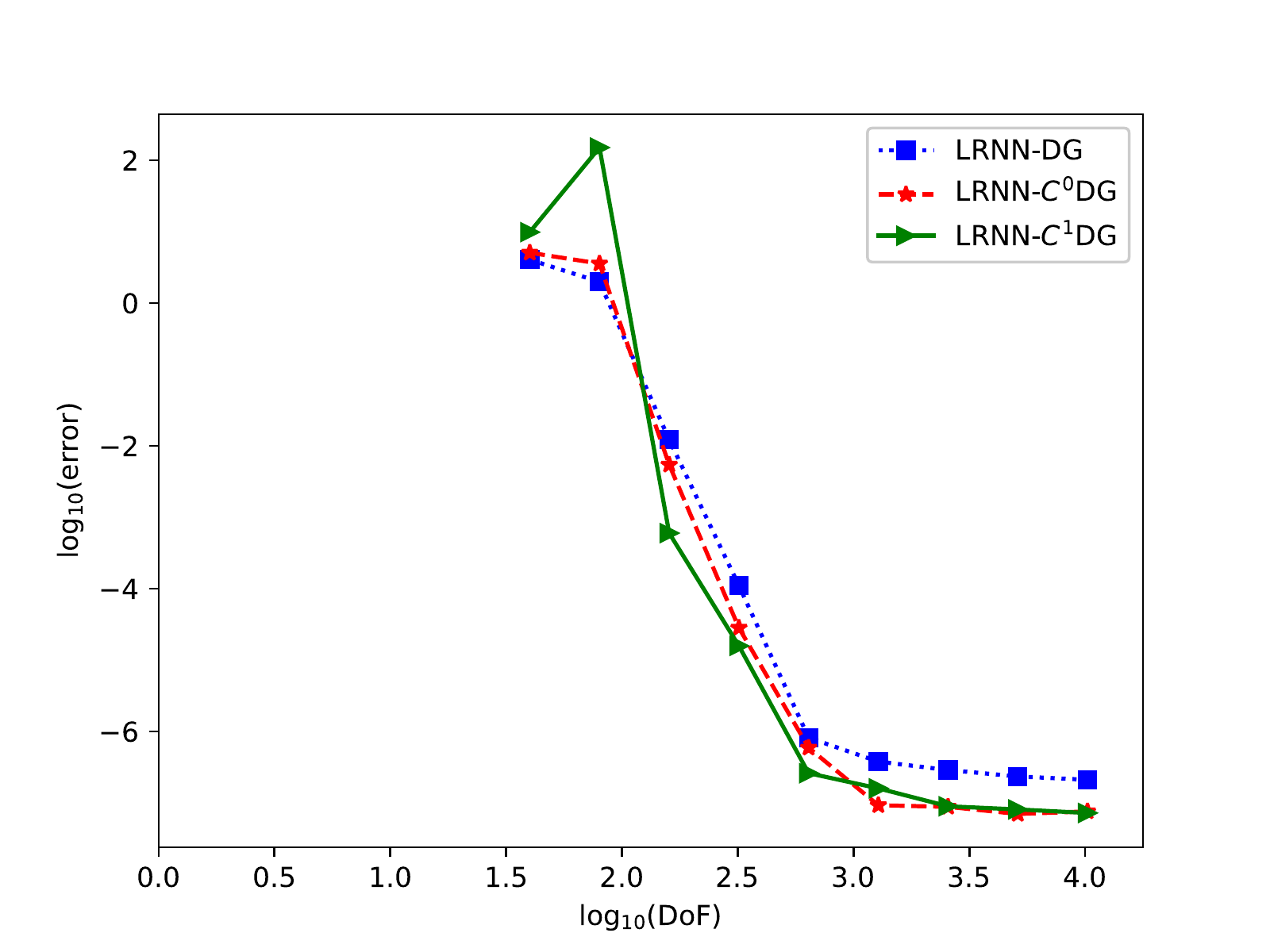}   
            \caption{$H^1$ errors with $h = 2^{-3}$}
            \label{fig:a}
    \end{subfigure}  
    \caption{Comparison of the errors obtained by the three methods in Example \ref{1dpoisson} for $\lambda$=1.} 
    \label{1dpoissonlambda1}
\end{figure}

We partition the interval $\Omega$ into non-overlapping uniform subintervals of size $h$ and choose the source term $f$ such that the solution form as given above satisfies this boundary value problem. The numerical errors in $L^2$ norm and $H^1$ seminorm for different numbers of degrees of freedom are shown in Table \ref{tablelrnndg}, Table \ref{tablelc0rnndg} and Table \ref{tablelrnnc1dg}.

Table \ref{tablelrnndg} lists the LRNN-DG errors with respect to the number of degrees of freedom on each element and the size of the subinterval. In this group of tests, the penalty parameter $\eta$ is set to be 0.0625 and 70 for $\lambda$ = 10, 1, respectively. And the the parameter $w_0$ of the uniform distribution is chosen as 5.5 and 4.8 for $\lambda$ = 10, 1, respectively. One can observe that for fixed $h$ the errors decrease rapidly with increasing numbers of degrees of freedom initially, and then the reduction slows down.
For fixed degrees of freedom per subinterval,
one can observe a general decrease in the errors as the subinterval size $h$ decreases.

Table \ref{tablelc0rnndg} lists the corresponding errors of LRNN-$C^0$DG in terms of the number of DoF per interval and the size of the element. In this group of tests, the parameter $w_0$ of the uniform distribution is chosen to equal 5.5 for $\lambda$ = 10, 1. The trend of errors with the size $h$ and the ${\rm DoF}_K$ is similar to that of LRNN-DG.

Table \ref{tablelrnnc1dg} shows the errors of LRNN-$C^1$DG in terms of the DoF per element and the size of the element. In this group of tests, the parameter $w_0$ of the uniform distribution is chosen to equal 5.5 and 4.5 for $\lambda$ = 10, 1, respectively. The trend of errors with the size $h$ and the ${\rm DoF}_K$ is similar to that of LRNN-DG. 

Figure \ref{1dpoissonlambda10} and Figure \ref{1dpoissonlambda1} are comparisons of the errors of the three methods for different sizes $h$ and different norms for $\lambda=10,1$ respectively. In general, we can see the performance of the three schemes are similar, LRNN-$C^1$DG is better than LRNN-$C^0$DG and LRNN-$C^0$DG is better than LRNN-DG. But we need to note that there are more equations in the last two methods, so they should offer better performance.

\begin{example}[Two-Dimensional Poisson Equation]
\label{2dpoisson}
Consider a two-dimensional Poisson equation with Dirichlet boundary condition,
\begin{equation*}
\begin{aligned}
- \Delta u&= f\quad &{\rm in}\ \Omega,\\
u&=g\quad &{\rm on}\ \partial\Omega. 
\end{aligned}
\end{equation*}
The exact solution $u = e^{x+y}sin(3\pi x+0.5\pi)cos(\pi y+0.2\pi) $ with $\Omega=[0,1]^2$.
\end{example}

The numerical errors in the $L^2$ norm and the $H^1$ seminorm for the different number of degrees of freedom are shown in Table \ref{table2dpoissonlrnndg}, Table \ref{table2dpoissonlrnnc0dg} and Table \ref{table2dpoissonlrnnc1dg}. For this problem, we also compare the proposed methods with the finite element method and the discontinuous Galerkin method. In this example, we partition the domain $\Omega$ into non-overlapping uniform square elements with the edge length $h$. The number of collection points on each edge used in LRNN-$C^0$DG and LRNN-$C^1$DG is 70.

%
\begin{table}[h]
\centering
\begin{tabular}{|l|ll|ll|ll|}
\hline
h   & \multicolumn{2}{c|}{$2^{-1}$}                                & \multicolumn{2}{c|}{$2^{-2}$}                                & \multicolumn{2}{c|}{$2^{-3}$}                                \\ \hline
\diagbox[width=5em,trim=l] {${\rm DoF}_K$}{Norm} & \multicolumn{1}{c|}{$L^{2}$}  & \multicolumn{1}{c|}{$H^{1}$} & \multicolumn{1}{c|}{$L^{2}$}  & \multicolumn{1}{c|}{$H^{1}$} & \multicolumn{1}{c|}{$L^{2}$}  & \multicolumn{1}{c|}{$H^{1}$} \\ \hline
10  & \multicolumn{1}{l|}{6.24E-01}  & 9.22E-00                    & \multicolumn{1}{l|}{1.16E-01} & 3.13E-00                    & \multicolumn{1}{l|}{5.77E-02} & 2.24E-00                     \\ \hline
20  & \multicolumn{1}{l|}{7.25E-02}  & 1.70E-00                    & \multicolumn{1}{l|}{1.72E-02} & 8.16E-01                     & \multicolumn{1}{l|}{2.35E-03} & 2.46E-01                    \\ \hline
40  & \multicolumn{1}{l|}{2.42E-03}  & 1.07E-01                     & \multicolumn{1}{l|}{1.77E-04} & 1.57E-02                     & \multicolumn{1}{l|}{4.41E-05} & 8.77E-03                     \\ \hline
80  & \multicolumn{1}{l|}{6.00E-05}  & 3.72E-03                     & \multicolumn{1}{l|}{2.71E-06} & 4.23E-04                     & \multicolumn{1}{l|}{7.35E-07} & 2.65E-04                     \\ \hline
160 & \multicolumn{1}{l|}{1.41E-05} & 1.05E-03                     & \multicolumn{1}{l|}{5.54E-07} & 1.07E-04                     & \multicolumn{1}{l|}{1.49E-07} & 6.56E-05                     \\ \hline
320 & \multicolumn{1}{l|}{6.83E-06} & 4.31E-04                     & \multicolumn{1}{l|}{3.94E-07} & 9.14E-05                     & \multicolumn{1}{l|}{1.09E-07} & 5.69E-05                     \\ \hline
640 & \multicolumn{1}{l|}{5.26E-06} & 4.18E-04                    & \multicolumn{1}{l|}{5.18E-07} & 1.45E-04                     & \multicolumn{1}{l|}{1.31E-07} & 6.65E-05                     \\ \hline
\end{tabular}
\caption{Errors of the LRNN-DG method for 2-d Poisson equation in Example \ref{2dpoisson}.}
\label{table2dpoissonlrnndg}
\end{table}

\begin{table}[h]
\centering
\begin{tabular}{|l|ll|ll|ll|}
\hline
h   & \multicolumn{2}{c|}{$2^{-1}$}                                & \multicolumn{2}{c|}{$2^{-2}$}                                & \multicolumn{2}{c|}{$2^{-3}$}                                \\ \hline
\diagbox[width=5em,trim=l] {${\rm DoF}_K$}{Norm} & \multicolumn{1}{c|}{$L^{2}$}  & \multicolumn{1}{c|}{$H^{1}$} & \multicolumn{1}{c|}{$L^{2}$}  & \multicolumn{1}{c|}{$H^{1}$} & \multicolumn{1}{c|}{$L^{2}$}  & \multicolumn{1}{c|}{$H^{1}$} \\ \hline
10  & \multicolumn{1}{l|}{7.83E-01} & 8.57E-00                     & \multicolumn{1}{l|}{4.66E-01} & 4.85E-00                     & \multicolumn{1}{l|}{8.31E-01} & 8.02E-00                     \\ \hline
20  & \multicolumn{1}{l|}{7.57E-01} & 7.02E-00                     & \multicolumn{1}{l|}{6.35E-02} & 1.28E-00                     & \multicolumn{1}{l|}{1.27E-02} & 4.48E-01                     \\ \hline
40  & \multicolumn{1}{l|}{3.81E-02} & 7.24E-01                     & \multicolumn{1}{l|}{6.41E-04} & 2.45E-02                     & \multicolumn{1}{l|}{1.90E-04} & 1.40E-02                     \\ \hline
80  & \multicolumn{1}{l|}{1.48E-04} & 4.49E-03                     & \multicolumn{1}{l|}{3.89E-06} & 2.47E-04                     & \multicolumn{1}{l|}{3.70E-07} & 4.60E-05                     \\ \hline
160 & \multicolumn{1}{l|}{5.35E-06} & 2.11E-04                     & \multicolumn{1}{l|}{9.38E-08} & 7.29E-06                     & \multicolumn{1}{l|}{1.12E-08} & 1.71E-06                     \\ \hline
320 & \multicolumn{1}{l|}{1.24E-06} & 5.25E-05                     & \multicolumn{1}{l|}{5.57E-08} & 4.52E-06                     & \multicolumn{1}{l|}{6.59E-09} & 1.07E-06                     \\ \hline
640 & \multicolumn{1}{l|}{1.27E-06} & 5.28E-05                     & \multicolumn{1}{l|}{4.29E-08} & 3.60E-06                     & \multicolumn{1}{l|}{7.52E-09} & 1.26E-06                     \\ \hline
\end{tabular}
\caption{Errors of the LRNN-$C^0$DG method for 2-d Poisson equation in Example \ref{2dpoisson}.}
\label{table2dpoissonlrnnc0dg}
\end{table}

\begin{table}[h]
\centering
\begin{tabular}{|l|ll|ll|ll|}
\hline
h   & \multicolumn{2}{c|}{$2^{-2}$}                                & \multicolumn{2}{c|}{$2^{-3}$}                                & \multicolumn{2}{c|}{$2^{-3}$}                                \\ \hline
\diagbox[width=5em,trim=l] {${\rm DoF}_K$}{Norm}  & \multicolumn{1}{c|}{$L^{2}$}  & \multicolumn{1}{c|}{$H^{1}$} & \multicolumn{1}{c|}{$L^{2}$}  & \multicolumn{1}{c|}{$H^{1}$} & \multicolumn{1}{c|}{$L^{2}$}  & \multicolumn{1}{c|}{$H^{1}$} \\ \hline
10  & \multicolumn{1}{l|}{1.02E-00} & 1.11E+01                     & \multicolumn{1}{l|}{6.09E-01} & 7.28E-00                     & \multicolumn{1}{l|}{1.16E-00} & 1.12E+01                     \\ \hline
20  & \multicolumn{1}{l|}{2.97E-01} & 3.80E-00                     & \multicolumn{1}{l|}{1.17E-01} & 2.83E-00                     & \multicolumn{1}{l|}{4.39E-02} & 1.49E-00                     \\ \hline
40  & \multicolumn{1}{l|}{3.61E-02} & 8.29E-01                     & \multicolumn{1}{l|}{5.63E-03} & 2.30E-01                     & \multicolumn{1}{l|}{1.57E-03} & 1.26E-01                     \\ \hline
80  & \multicolumn{1}{l|}{6.70E-04} & 2.39E-02                     & \multicolumn{1}{l|}{1.03E-04} & 7.09E-03                     & \multicolumn{1}{l|}{2.00E-05} & 2.72E-03                     \\ \hline
160 & \multicolumn{1}{l|}{2.17E-06} & 1.20E-04                     & \multicolumn{1}{l|}{6.86E-08} & 7.38E-06                     & \multicolumn{1}{l|}{1.19E-08} & 2.54E-06                     \\ \hline
320 & \multicolumn{1}{l|}{1.45E-07} & 9.24E-06                     & \multicolumn{1}{l|}{1.05E-08} & 1.35E-06                     & \multicolumn{1}{l|}{3.28E-09} & 8.36E-07                     \\ \hline
640 & \multicolumn{1}{l|}{8.30E-08} & 5.83E-06                     & \multicolumn{1}{l|}{6.95E-09} & 9.54E-07                     & \multicolumn{1}{l|}{3.02E-09} & 8.39E-07                     \\ \hline
\end{tabular}
\caption{Errors of the LRNN-$C^1$DG method for 2-d Poisson equation in Example \ref{2dpoisson}.}
\label{table2dpoissonlrnnc1dg}
\end{table}

\begin{figure}[!ht] 
  \centering  
   
  \begin{subfigure}{0.34\textwidth}
      \centering      
      \includegraphics[width=\textwidth]{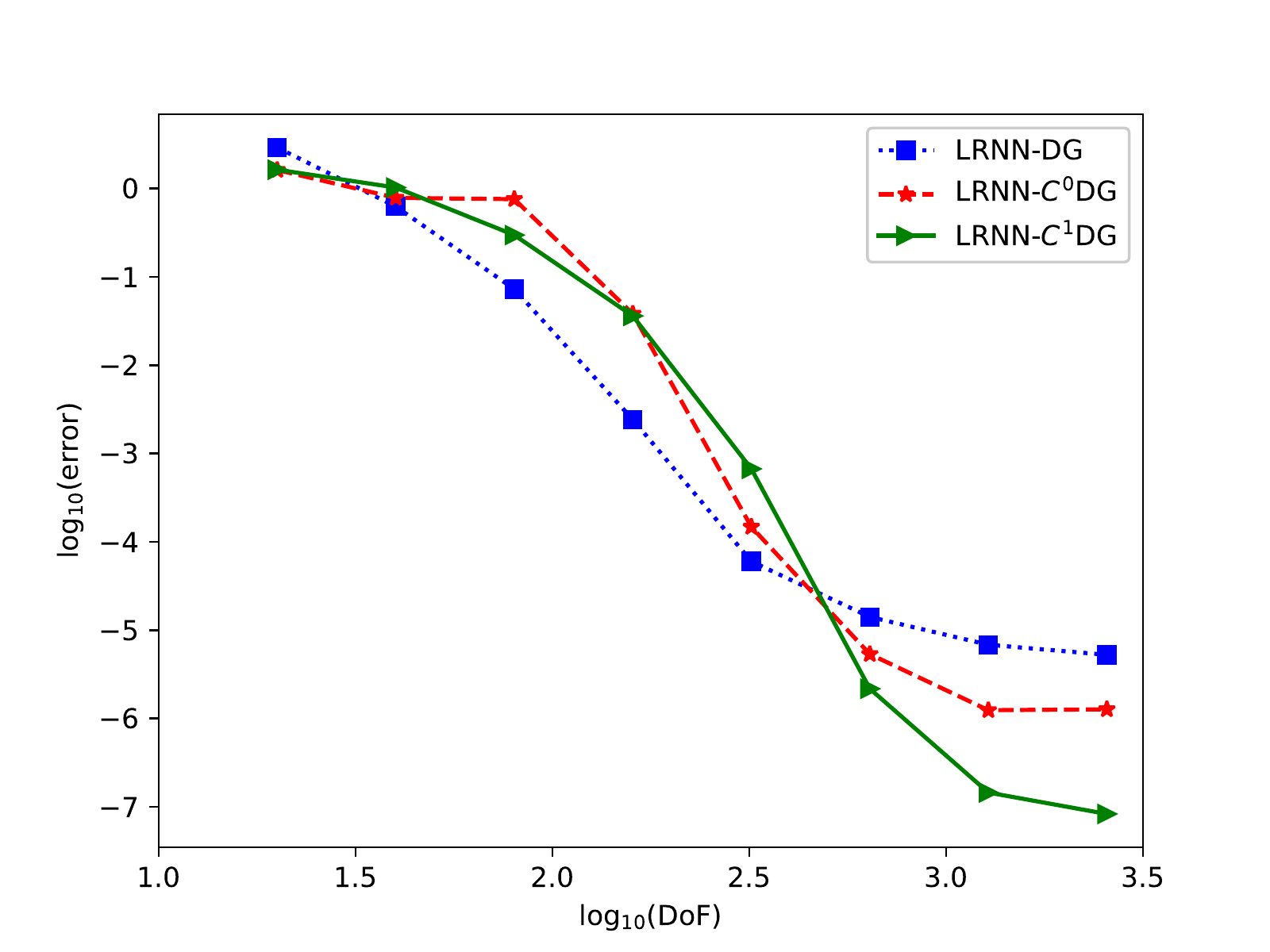}  
      \caption{$L^2$ errors with h=$2^{-1}$}
      \label{fig:a}
  \end{subfigure}
   \begin{subfigure}{0.34\textwidth}
      \centering      
      \includegraphics[width=\textwidth]{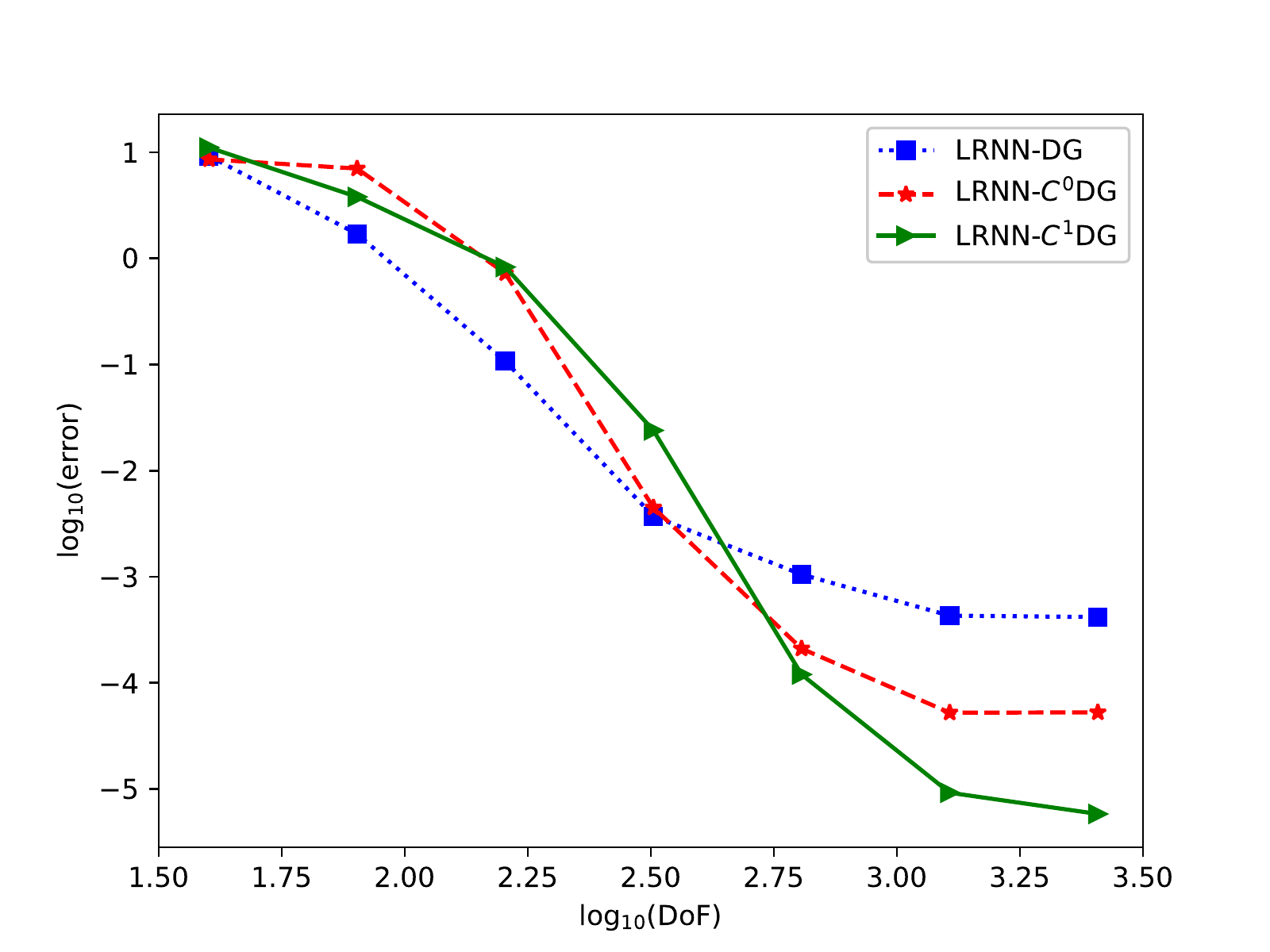}  
      \caption{$H^1$ errors with h=$2^{-1}$}
      \label{fig:b}
  \end{subfigure}\\
  \begin{subfigure}{0.34\textwidth}
      \centering    
      \includegraphics[width=\textwidth]{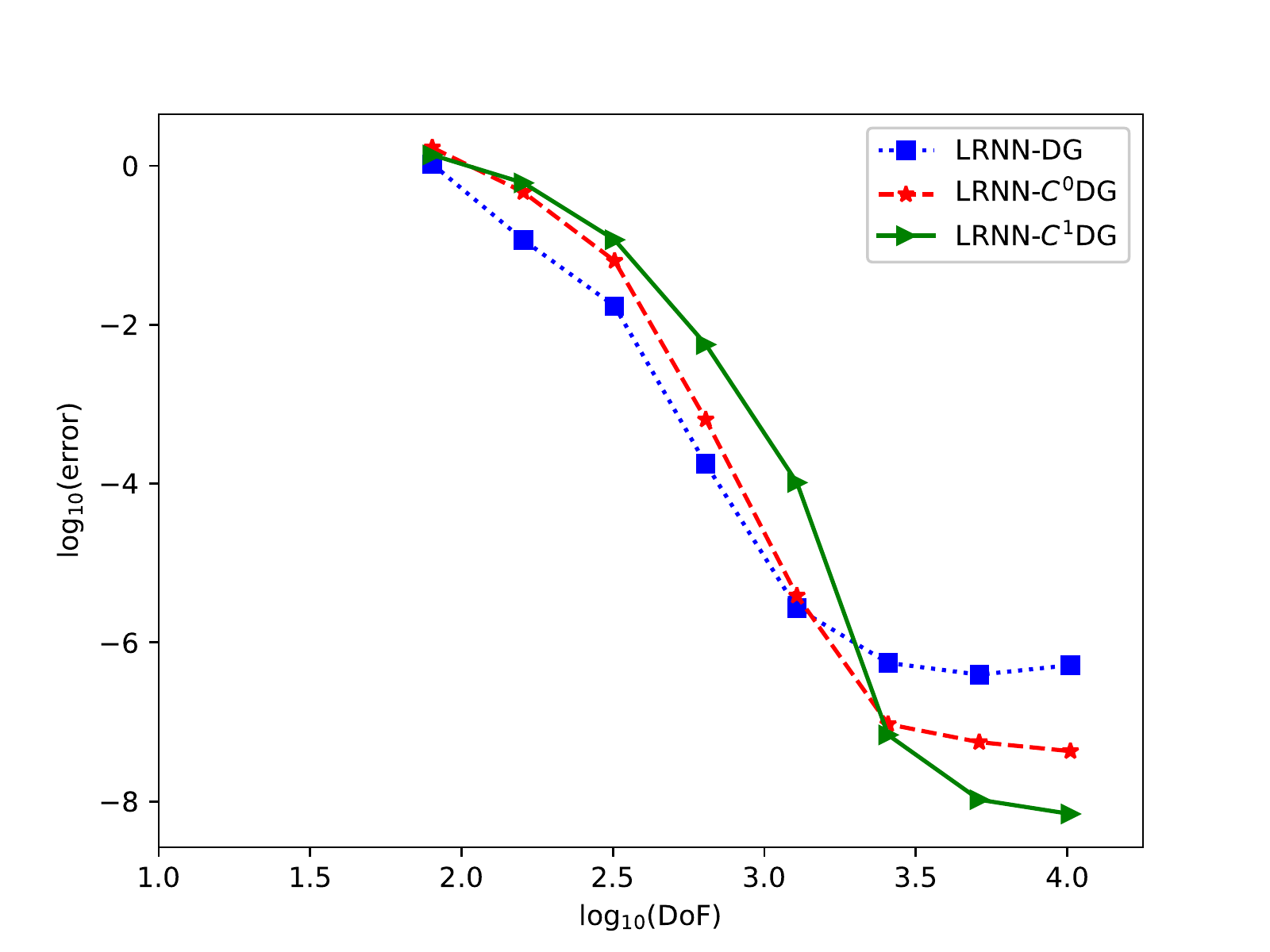}  
		 \caption{$L^2$ errors with h=$2^{-2}$}
      \label{fig:c}
    \end{subfigure}  
    \begin{subfigure}{0.34\textwidth}
      \centering    
      \includegraphics[width=\textwidth]{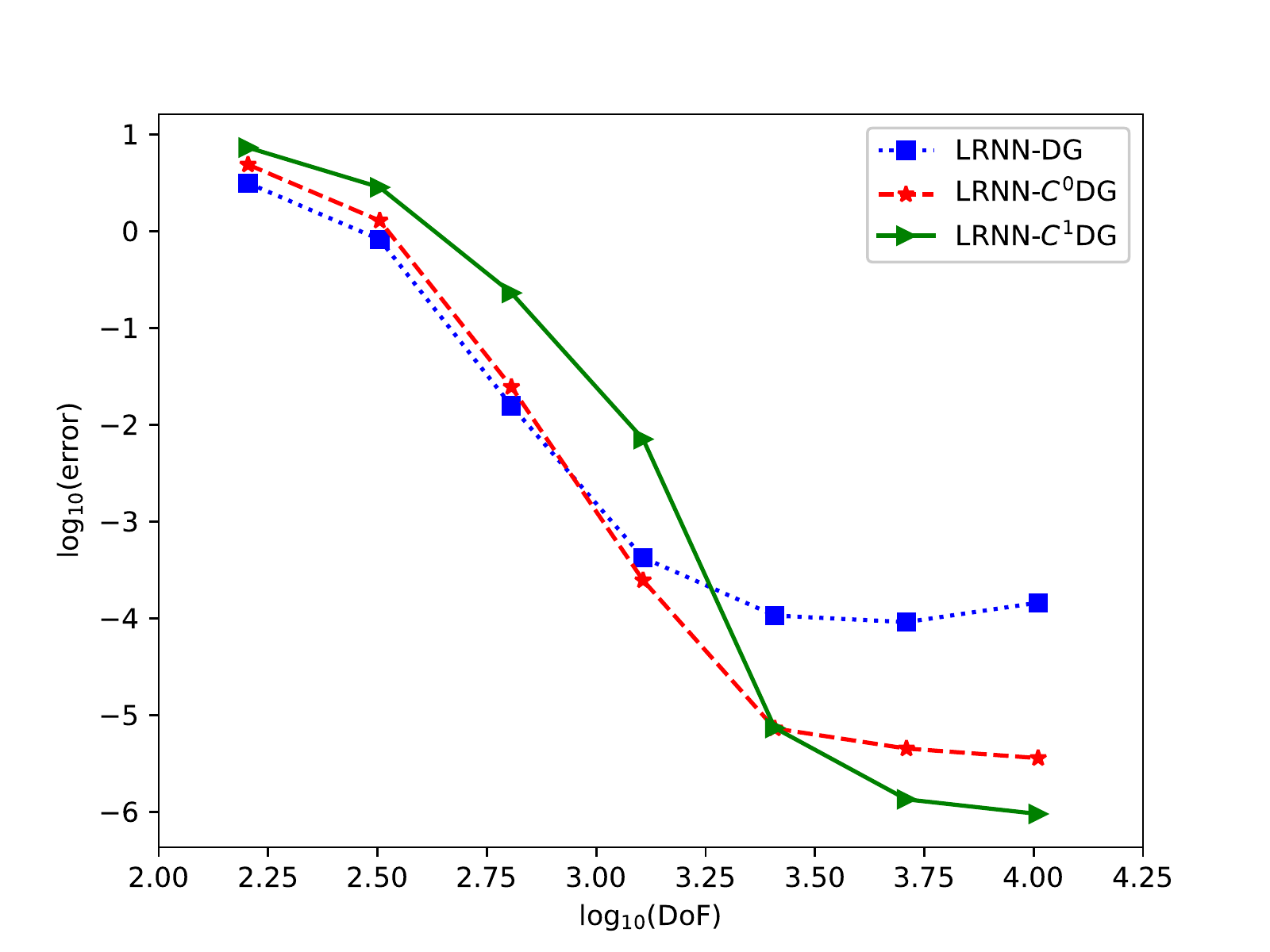}  
		 \caption{$H^1$ errors with h=$2^{-2}$}
      \label{fig:d}
    \end{subfigure} 
  \caption{Errors obtained from three different methods in Example \ref{2dpoisson}.} 
    \label{figure2dpoissonthree}
\end{figure}

\begin{figure}[htb] 
  \centering  
   
  \begin{subfigure}{0.34\textwidth}
      \centering      
      \includegraphics[width=\textwidth]{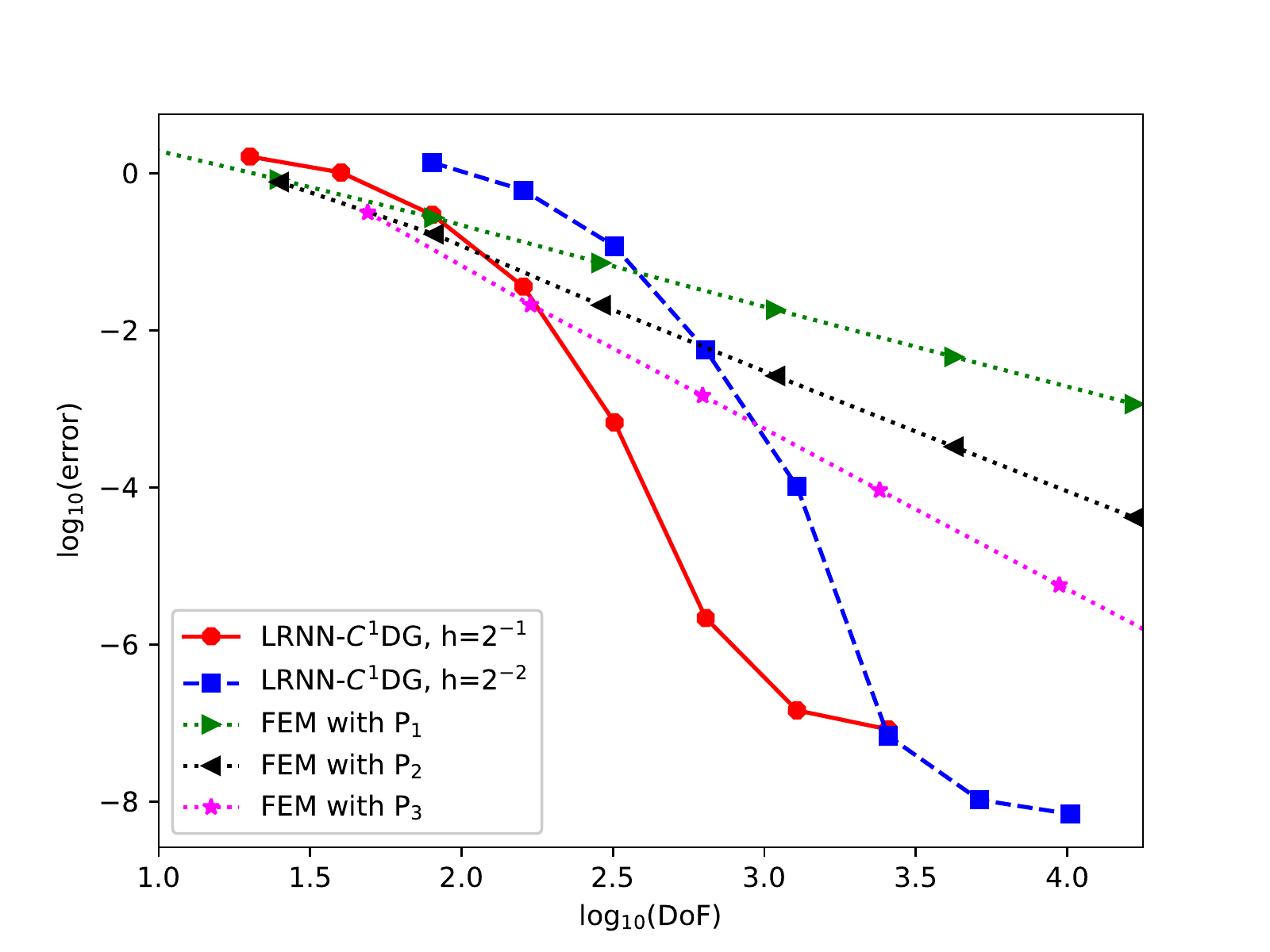}  
      \caption{$L^2$ errors by FEM and LRNN-$C^1$DG}
      \label{fig:a}
  \end{subfigure}
  \begin{subfigure}{0.34\textwidth}
      \centering      
      \includegraphics[width=\textwidth]{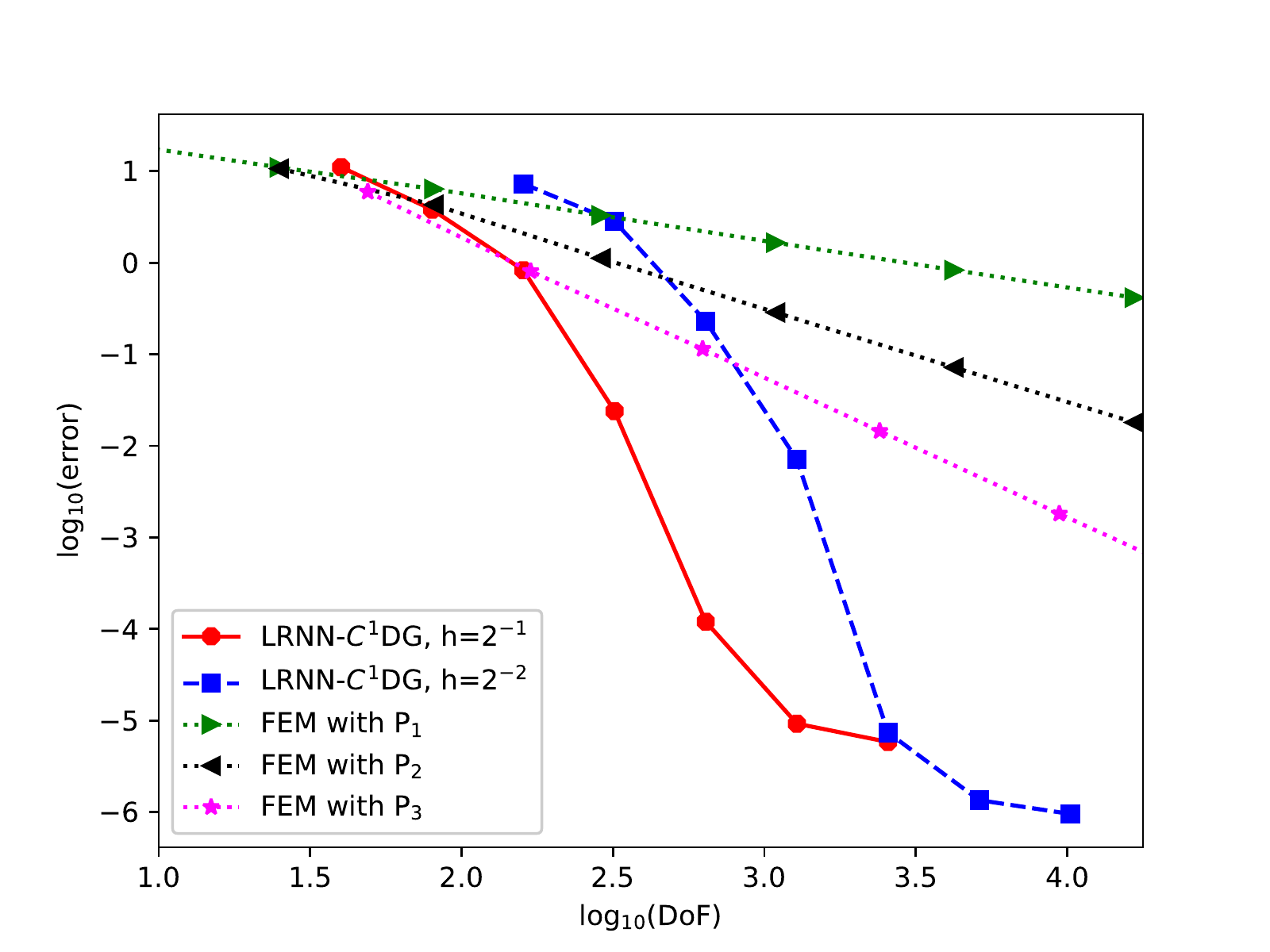}  
      \caption{$H^1$ errors by FEM and LRNN-$C^1$DG}
      \label{fig:b}
  \end{subfigure}\\
  \begin{subfigure}{0.34\textwidth}
      \centering    
      \includegraphics[width=\textwidth]{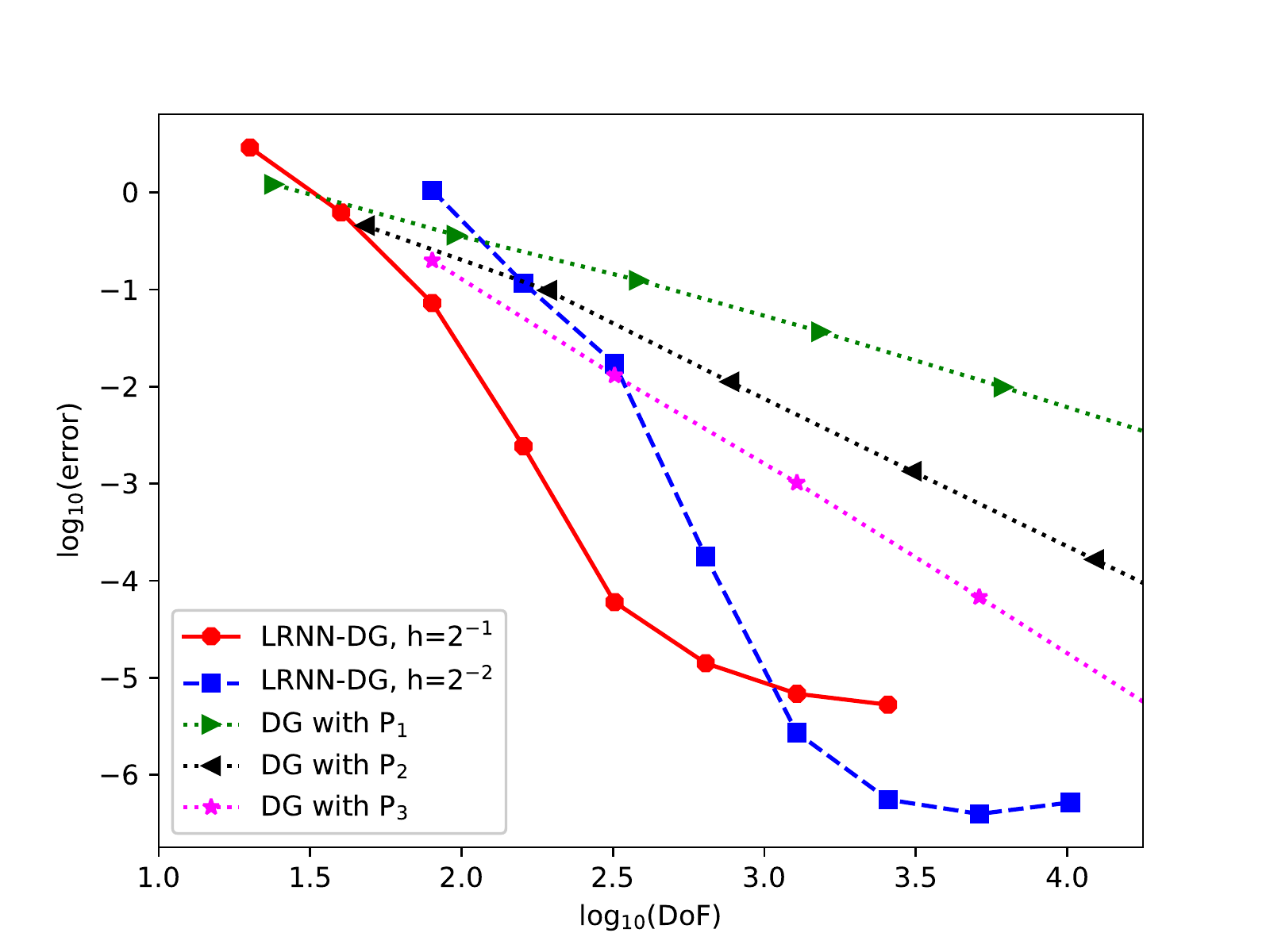}  
		 \caption{$L^2$ errors by DG and LRNN-DG}
      \label{fig:c}
    \end{subfigure}  
   \begin{subfigure}{0.34\textwidth}
      \centering    
      \includegraphics[width=\textwidth]{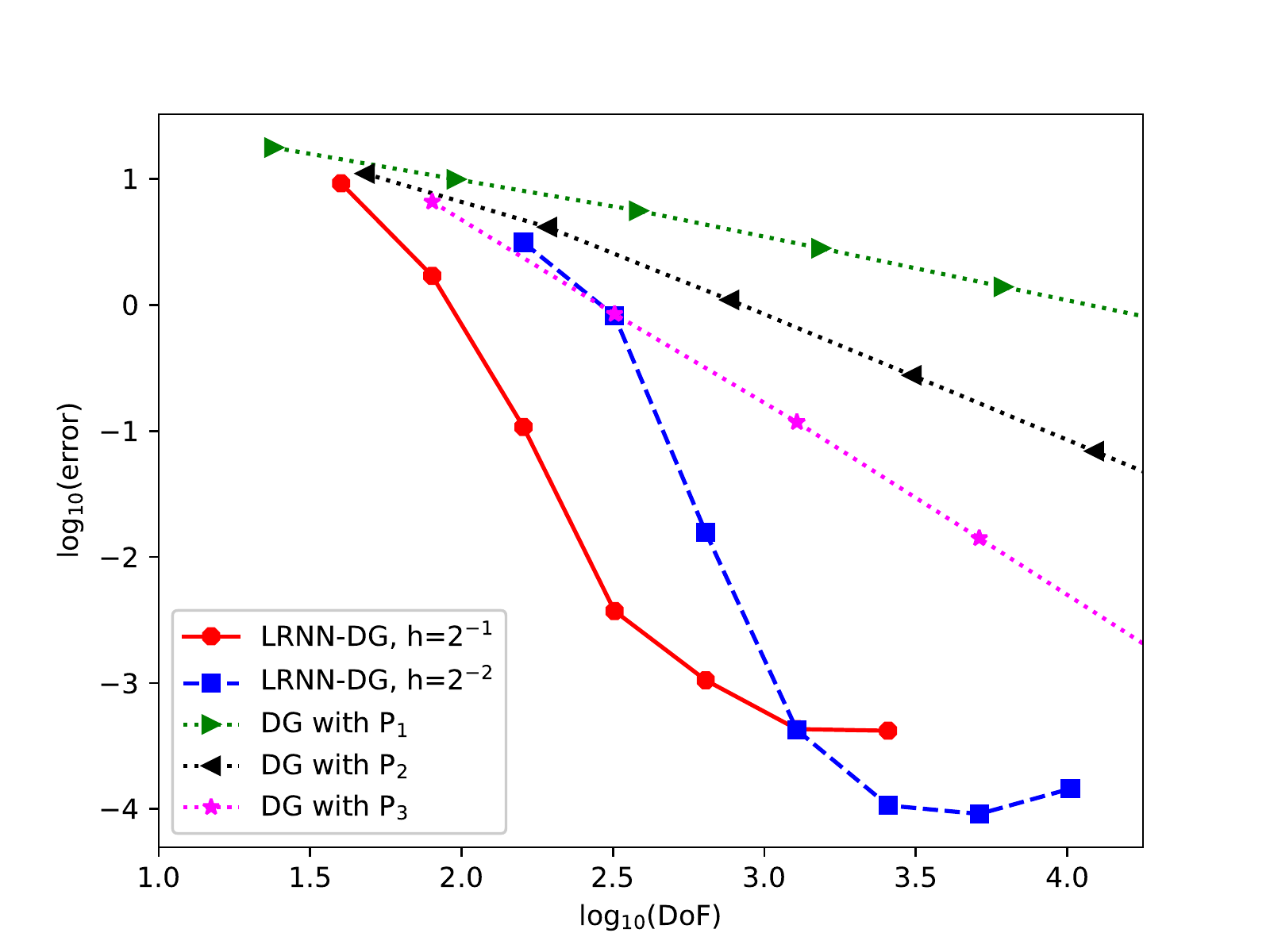}  
		 \caption{$H^1$ errors by DG and LRNN-DG}
      \label{fig:d}
    \end{subfigure}  
  \caption{Comparison of the errors obtained by FEM, DG and LRNN with DG methods in Example \ref{2dpoisson}.}  
  \label{figure2dpoissonoldnew}
\end{figure}

\begin{figure}[htb]  
  \centering   
   
  \begin{subfigure}{0.4\textwidth}
      \centering      
      \includegraphics[width=0.8\textwidth,height=0.24\textheight]{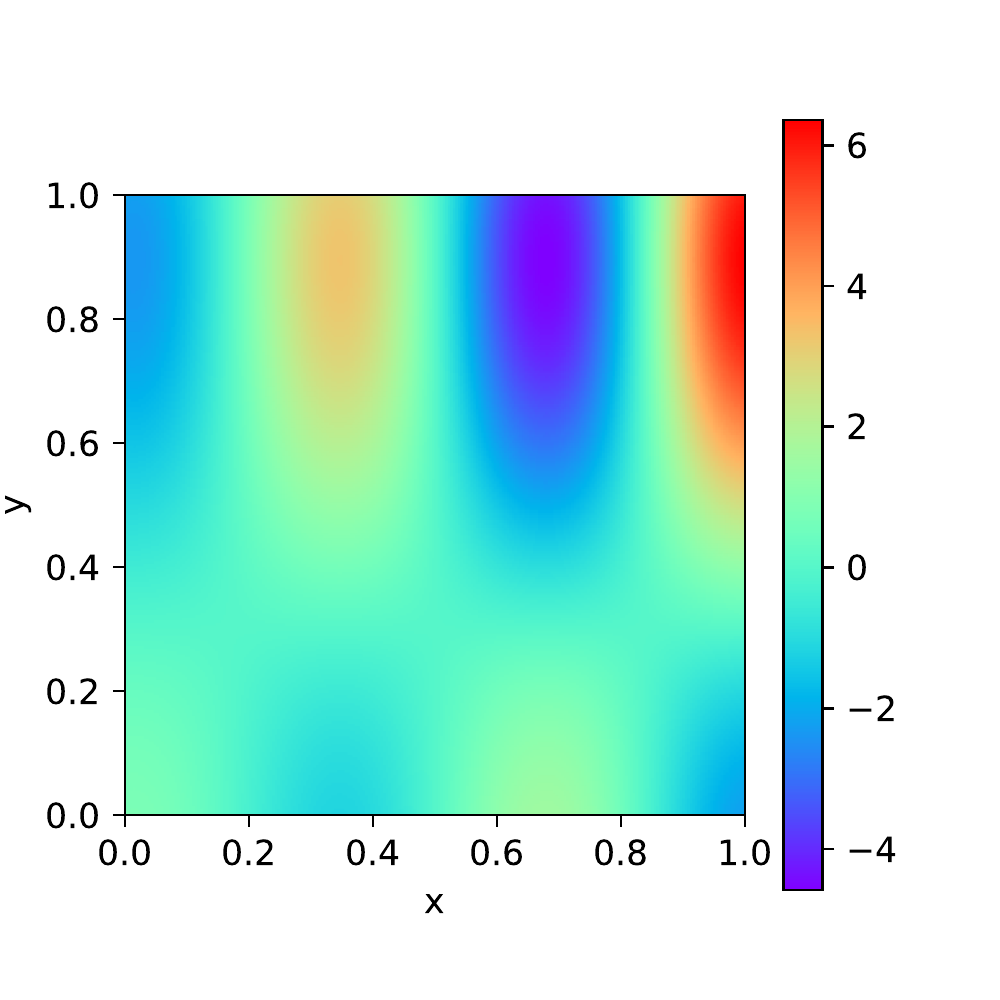}  
      \caption{The exact solution.}
      \label{fig:a}
  \end{subfigure}
\begin{subfigure}{0.4\textwidth}
      \centering      
      \includegraphics[width=0.8\textwidth,height=0.24\textheight]{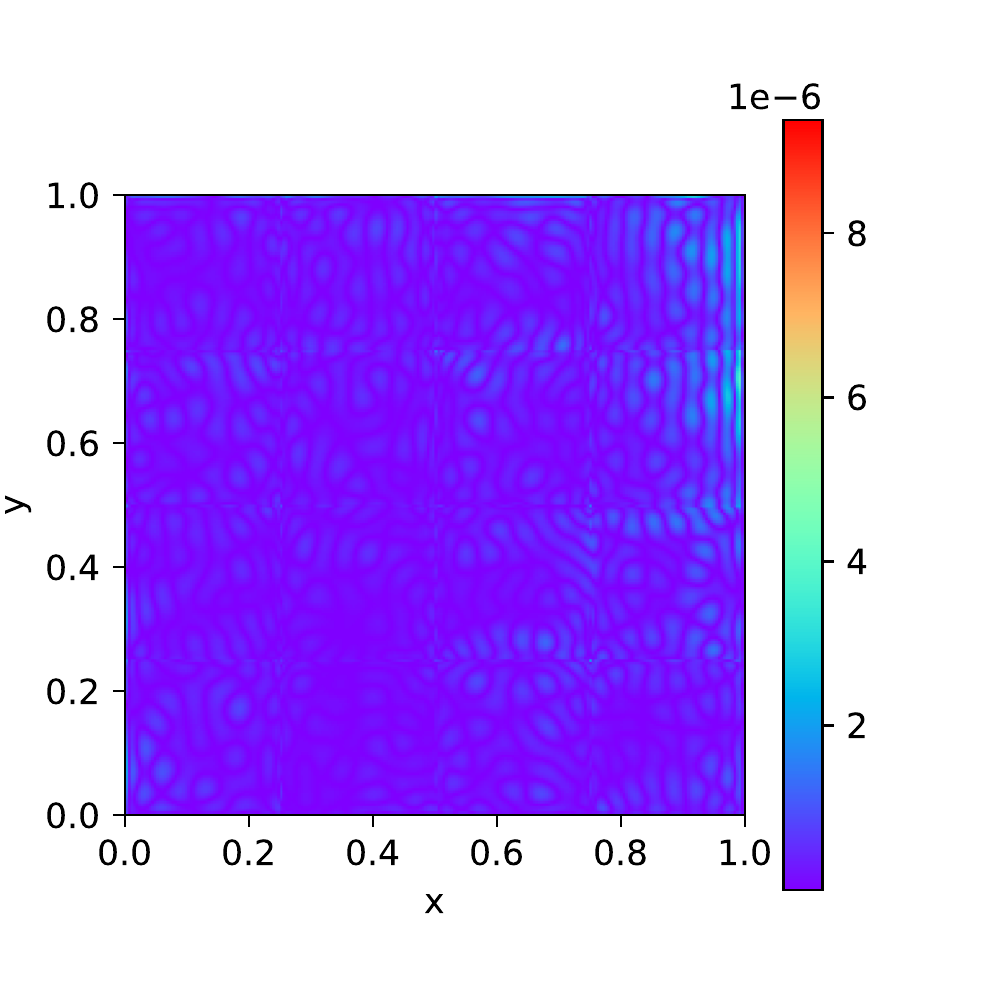}  
      \caption{The absolute error for LRNN-DG method.}
      \label{fig:b}
  \end{subfigure}\\

  \begin{subfigure}{0.4\textwidth}
      \centering    
      \includegraphics[width=0.8\textwidth,height=0.24\textheight]{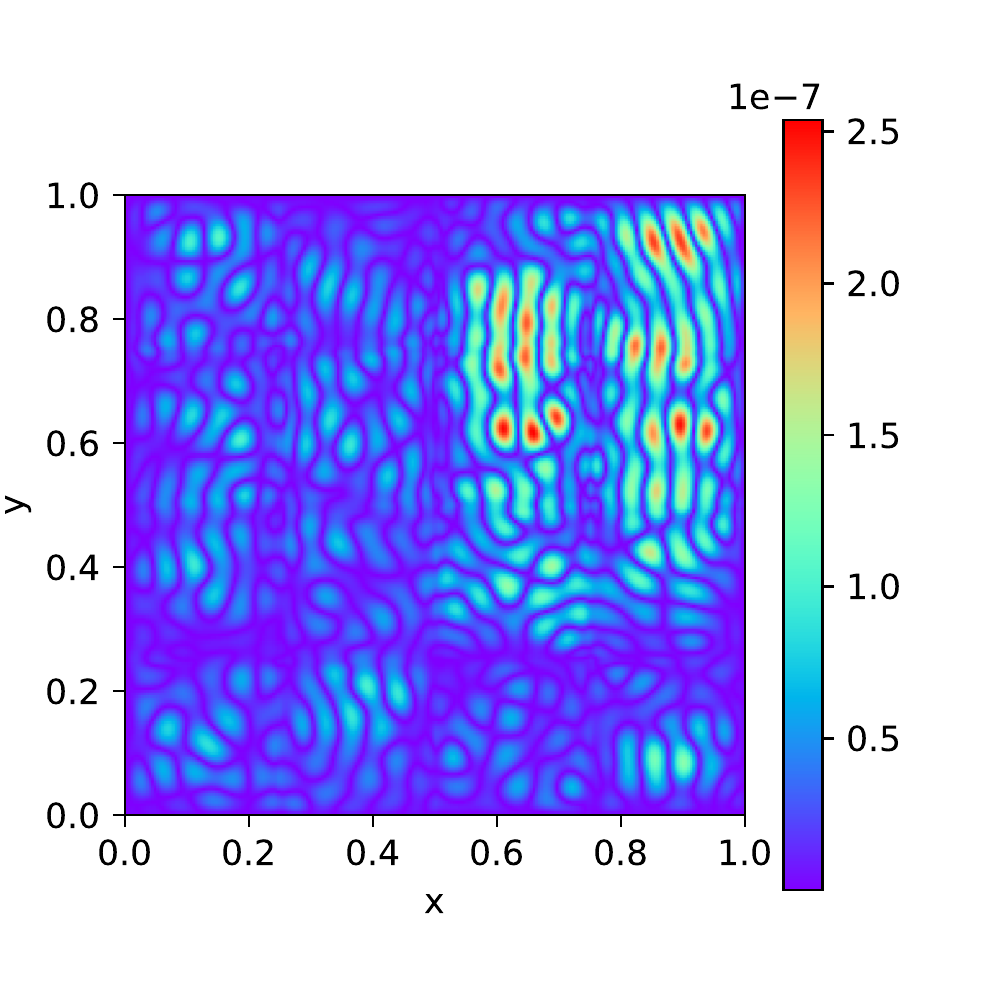}  
		 \caption{The absolute error for LRNN-$C^0$DG method.}
      \label{fig:c}
    \end{subfigure}  
\begin{subfigure}{0.4\textwidth}
      \centering      
      \includegraphics[width=0.8\textwidth,height=0.24\textheight]{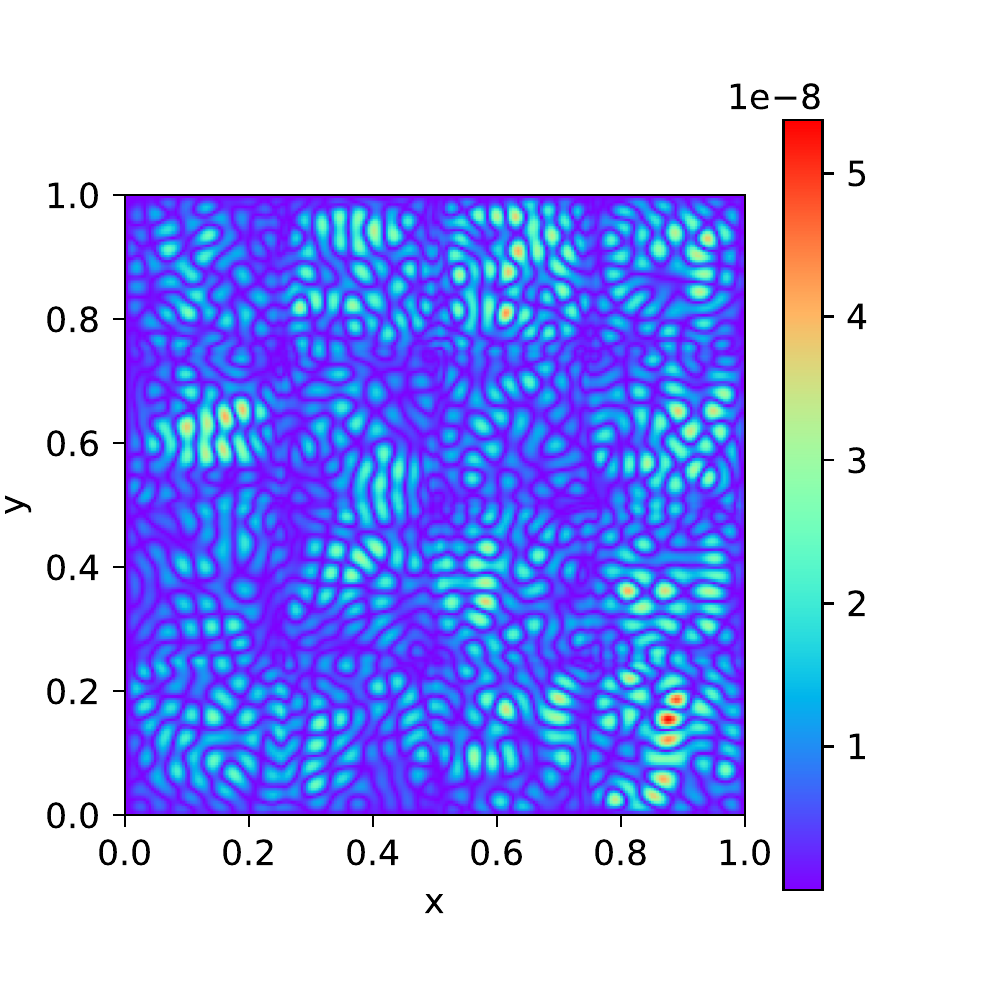}  
      \caption{The absolute error for LRNN-$C^1$DG method.}
      \label{fig:d}
  \end{subfigure} 
  \caption{Exact solution and absolute errors computed by three methods with $h=2^{-2}$ in Example \ref{2dpoisson}.}  
  \label{figure2dpoissonerror}
\end{figure}

Table \ref{table2dpoissonlrnndg} records the errors of LRNN-DG in terms of degrees of freedom on each element and the size of the element. In this group of tests, the weight/bias coefficients in the hidden layer of each local network are set to uniform random values generated in [-1, 1]. We can see that errors in both the $L^2$ norm and the $H^1$ seminorm decrease along with the increase of ${\rm Dof}_K$ fast firstly and then slowly for the fixed $h$. We also observe that if we reduce the size $h$, the errors decrease for the fixed ${\rm Dof}_K$.

Table \ref{table2dpoissonlrnnc0dg} and Table \ref{table2dpoissonlrnnc1dg} show the errors of LRNN-$C^0$DG and LRNN-$C^1$DG in terms of the number of degrees of freedom on each element and the size of the element, respectively. The parameters $w_0$ of the uniform distribution are chosen to equal 0.63 and 1.29 for the two methods, respectively. The trends observed here for LRNN-$C^0$DG and LRNN-$C^0$DG are similar to that of the LRNN-DG method.

Figure \ref{figure2dpoissonthree} compares the errors of the three methods for different sizes $h$ and different norms. We can see that LRNN-$C^1$DG is better than LRNN-$C^0$DG and LRNN-$C^0$DG is better than LRNN-DG.

Figure \ref{figure2dpoissonoldnew} shows a comparison between the finite element method and the current LRNN-$C^1$DG method, and between the usual discontinuous Galerkin method (based on polynomials) and the current LRNN-DG method.
Here $P_k$ denotes the $k$th-order piecewise continuous or discontinuous polynomial functions. We observe that the errors of the proposed methods are much smaller than those of the DG method and FEM with the same DoF. However, when the DoF increases to a certain value, the errors of the current methods appear to stagnate. This phenomenon may be due to the possibility that the basis functions $\{\phi^{K}_j(\bx): j = 1,2, \cdots, M\}$ can tend to be linearly dependent as $M$ becomes larger, which leads to a rank deficient linear system.
 One way to reduce the error is to increase the number of the subdomains, i.e., using elements with a smaller size $h$. Another approach is to design better neural networks that can provide improved basis functions. This is one aspect we will further explore in future work.

Figure \ref{figure2dpoissonerror} shows distributions of the point-wise absolute error computed using the three methods with the element size $h = 2^{-2}$ and the number of ${\rm Dof}_K$ in each element 320. We can see the absolute error of LRNN-DG is larger but smoother. The absolute error of LRNN-$C^0$DG and LRNN-$C^1$DG are smaller but have a larger amplitude.

 Table~\ref{test_assump} provides some numerical evidence about the reasonableness of the Assumption \ref{assump_approC0} and Assumption \ref{assump_approC1}. Here, we have performed a test to examine how the jump of the numerical solution $u_h$ and its gradient $\nabla u_h$ vary as the number of collocation points increases with the LRNN-$C^1$DG method.
 We consider a vertical interior edge and a horizontal interior edge and compute the $L^2(e)$-norm of $\llbracket u_h \rrbracket$ and $[\nabla u_h]$ on these edges. We also consider a vertical boundary edge and a horizontal boundary edge and compute the $L^2(e)$-norm of $u_h -g$ on these boundary edges. In these tests, the width of each local neural network is fixed at 320 and the size of each element is $h=2^{-2}$. Table \ref{test_assump} lists these errors corresponding to a set of collocation points.
 It is evident that $\|\llbracket u_h \rrbracket\|_{0,e}$, $\|[\nabla u_h]\|_{0,e}$ and $\|u_h - g\|_{0,e}$ decrease rapidly with increasing number of collection points $N_e$, reaching a level close to the machine zero as $N_e$ becomes large.

\begin{table}[h]
\centering
\begin{tabular}{|l|lll|lll|}
\hline
\multicolumn{1}{|c|}{\multirow{2}{*}{$N_e$}} & \multicolumn{3}{c|}{Vertical Edge}                                                                                        & \multicolumn{3}{c|}{Horizontal Edge}                                                                                      \\ \cline{2-7} 
\multicolumn{1}{|c|}{}                                   & \multicolumn{1}{c|}{$\|\llbracket u_h \rrbracket\|_{0,e}$} & \multicolumn{1}{c|}{$\|[\nabla u_h]\|_{0,e}$} & \multicolumn{1}{c|}{$\|u_h - g\|_{0,e}$} & \multicolumn{1}{c|}{$\|\llbracket u_h \rrbracket\|_{0,e}$} & \multicolumn{1}{c|}{$\|[\nabla u_h]\|_{0,e}$} & \multicolumn{1}{c|}{$\|u_h - g\|_{0,e}$} \\ \hline
5                                                        & \multicolumn{1}{l|}{1.73E-05}                    & \multicolumn{1}{l|}{4.52E-04}       & 1.53E-05                         & \multicolumn{1}{l|}{6.18E-05}                    & \multicolumn{1}{l|}{1.47E-03}       & 3.93E-05                         \\ \hline
10                                                       & \multicolumn{1}{l|}{3.31E-08}                    & \multicolumn{1}{l|}{3.01E-06}       & 2.08E-08                         & \multicolumn{1}{l|}{5.95E-08}                    & \multicolumn{1}{l|}{2.34E-06}       & 3.49E-08                         \\ \hline
20                                                       & \multicolumn{1}{l|}{3.00E-11}                    & \multicolumn{1}{l|}{1.89E-09}       & 5.61E-11                         & \multicolumn{1}{l|}{1.62E-11}                    & \multicolumn{1}{l|}{1.42E-09}       & 1.25E-10                         \\ \hline
40                                                       & \multicolumn{1}{l|}{2.34E-13}                    & \multicolumn{1}{l|}{7.94E-13}       & 2.72E-13                         & \multicolumn{1}{l|}{2.51E-13}                    & \multicolumn{1}{l|}{8.84E-13}       & 4.90E-13                         \\ \hline
80                                                       & \multicolumn{1}{l|}{1.62E-13}                   & \multicolumn{1}{l|}{5.59E-13}       & 2.51E-13                         & \multicolumn{1}{l|}{1.76E-13}                    & \multicolumn{1}{l|}{1.18E-13}       & 1.89E-13                         \\ \hline
\end{tabular}
\caption{$L^2(e)$-norm of $\llbracket u_h \rrbracket$, $[\nabla u_h]$, and $u_h - g$ computed by LRNN-$C^1$DG in Example \ref{2dpoisson}.}
\label{test_assump}
\end{table}

\begin{example}[Heat Equation]
\label{2dheat}
The last example is a heat equation described as follows
\begin{subequations}
\begin{align}
u_t (t, \bx) - \lambda \Delta u (t, \bx) &= f (t, \bx)\quad {\rm in}\ I\times\Omega,\nonumber\\
u(0, \bx) &= u_0( \bx)\quad\ \ {\rm in}\ \Omega,\nonumber\\
u (t, \bx) & = g (t, \bx) \quad {\rm on}\ I \times\partial \Omega,\nonumber
\end{align}
\end{subequations}
where $\Omega=[0,1]$, $I=[0,1]$, the constant
$\lambda = 1$ or $0.001$, $f$ is a source term, and $u_0$ and $g$ are
the initial and boundary conditions.
We employ the manufactured exact solution
\begin{equation}
\begin{aligned}
u = -e^{cos(\pi x+3\pi)+t^2}.\nonumber
\end{aligned}
\end{equation}
\end{example}

\begin{table}[h]
	\centering
\begin{tabular}{|l|l|ll|ll|ll|}
\hline
\multirow{2}{*}{}                & $h$   & \multicolumn{2}{c|}{$2^{-1}$}                                & \multicolumn{2}{c|}{$2^{-2}$}                                & \multicolumn{2}{c|}{$2^{-3}$}                                \\ \cline{2-8} 
                                 & \diagbox[width=5em,trim=l]{${\rm DoF}_\sigma$}{Norm} & \multicolumn{1}{c|}{$L^{2}$}  & \multicolumn{1}{c|}{$H^{1}$} & \multicolumn{1}{c|}{$L^{2}$}  & \multicolumn{1}{c|}{$H^{1}$} & \multicolumn{1}{c|}{$L^{2}$}  & \multicolumn{1}{c|}{$H^{1}$} \\ \hline
\multirow{6}{*}{$\lambda=0.001$} & 20  & \multicolumn{1}{l|}{1.64E-01} & 4.96E-00                     & \multicolumn{1}{l|}{3.32E-02} & 1.88E-00                     & \multicolumn{1}{l|}{1.77E-01} & 2.06E+01                     \\ \cline{2-8} 
                                 & 40  & \multicolumn{1}{l|}{1.98E-02} & 1.10E-00                     & \multicolumn{1}{l|}{2.18E-03} & 2.00E-01                     & \multicolumn{1}{l|}{2.16E-02} & 3.94E-00                     \\ \cline{2-8} 
                                 & 80  & \multicolumn{1}{l|}{1.36E-05} & 1.17E-03                     & \multicolumn{1}{l|}{1.14E-05} & 1.78E-03                     & \multicolumn{1}{l|}{4.56E-06} & 1.42E-03                     \\ \cline{2-8} 
                                 & 160 & \multicolumn{1}{l|}{1.91E-06} & 1.99E-04                     & \multicolumn{1}{l|}{2.74E-07} & 5.80E-05                     & \multicolumn{1}{l|}{1.38E-07} & 6.15E-05                     \\ \cline{2-8} 
                                 & 320 & \multicolumn{1}{l|}{5.46E-07} & 5.98E-05                     & \multicolumn{1}{l|}{1.40E-07} & 3.24E-05                     & \multicolumn{1}{l|}{1.06E-07} & 5.36E-05                     \\ \cline{2-8} 
                                 & 640 & \multicolumn{1}{l|}{4.75E-07} & 5.51E-05                     & \multicolumn{1}{l|}{1.65E-07} & 4.83E-05                     & \multicolumn{1}{l|}{6.96E-08} & 3.86E-05                     \\ \hline
\multirow{6}{*}{$\lambda=1$}     & 20  & \multicolumn{1}{l|}{1.11E-01} & 2.98E-00                     & \multicolumn{1}{l|}{1.18E-01} & 5.67E+00                     & \multicolumn{1}{l|}{5.92E-01} & 5.44E+01                     \\ \cline{2-8} 
                                 & 40  & \multicolumn{1}{l|}{4.56E-03} & 1.77E-01                     & \multicolumn{1}{l|}{2.39E-03} & 2.43E-01                     & \multicolumn{1}{l|}{5.46E-03} & 9.47E-01                     \\ \cline{2-8} 
                                 & 80  & \multicolumn{1}{l|}{1.50E-05} & 1.43E-03                     & \multicolumn{1}{l|}{2.78E-05} & 5.37E-03                     & \multicolumn{1}{l|}{1.62E-05} & 6.27E-03                     \\ \cline{2-8} 
                                 & 160 & \multicolumn{1}{l|}{6.51E-07} & 8.70E-05                     & \multicolumn{1}{l|}{1.05E-07} & 2.90E-05                     & \multicolumn{1}{l|}{4.54E-08} & 2.59E-05                     \\ \cline{2-8} 
                                 & 320 & \multicolumn{1}{l|}{2.56E-07} & 3.78E-05                     & \multicolumn{1}{l|}{5.33E-08} & 1.49E-05                     & \multicolumn{1}{l|}{3.04E-08} & 2.38E-05                     \\ \cline{2-8} 
                                 & 640 & \multicolumn{1}{l|}{1.82E-07} & 3.22E-05                     & \multicolumn{1}{l|}{5.53E-08} & 2.30E-05                     & \multicolumn{1}{l|}{3.34E-08} & 2.77E-05                     \\ \hline

\end{tabular}
\caption{Errors of the space-time LRNN-DG method for heat equation when $t=1$ in Example \ref{2dheat}.}
\label{table2dheatlrnndg}
\end{table}

\begin{table}[h]
	\centering
\begin{tabular}{|l|l|ll|ll|ll|}
\hline
\multirow{2}{*}{}                & $h$   & \multicolumn{2}{c|}{$2^{-1}$}                                & \multicolumn{2}{c|}{$2^{-2}$}                                & \multicolumn{2}{c|}{$2^{-3}$}                                \\ \cline{2-8} 
                                 & \diagbox[width=5em,trim=l]{${\rm DoF}_\sigma$}{Norm} & \multicolumn{1}{c|}{$L^{2}$}  & \multicolumn{1}{c|}{$H^{1}$} & \multicolumn{1}{c|}{$L^{2}$}  & \multicolumn{1}{c|}{$H^{1}$} & \multicolumn{1}{c|}{$L^{2}$}  & \multicolumn{1}{c|}{$H^{1}$} \\ \hline
\multirow{6}{*}{$\lambda=0.001$} & 20  & \multicolumn{1}{l|}{1.08E-01} & 2.34E-00                     & \multicolumn{1}{l|}{3.60E-02} & 1.36E-00                     & \multicolumn{1}{l|}{4.92E-02} & 3.11E-00                     \\ \cline{2-8} 
                                 & 40  & \multicolumn{1}{l|}{2.06E-03} & 7.53E-02                     & \multicolumn{1}{l|}{6.79E-04} & 4.80E-02                     & \multicolumn{1}{l|}{1.14E-03} & 1.50E-01                     \\ \cline{2-8} 
                                 & 80  & \multicolumn{1}{l|}{9.77E-06} & 6.08E-04                     & \multicolumn{1}{l|}{1.36E-05} & 1.45E-03                     & \multicolumn{1}{l|}{9.57E-06} & 2.01E-03                     \\ \cline{2-8} 
                                 & 160 & \multicolumn{1}{l|}{6.57E-07} & 6.36E-05                     & \multicolumn{1}{l|}{6.33E-08} & 9.21E-06                     & \multicolumn{1}{l|}{3.39E-08} & 9.84E-06                     \\ \cline{2-8} 
                                 & 320 & \multicolumn{1}{l|}{1.81E-07} & 1.59E-05                     & \multicolumn{1}{l|}{5.60E-08} & 9.98E-06                     & \multicolumn{1}{l|}{2.76E-08} & 8.80E-06                     \\ \cline{2-8} 
                                 & 640 & \multicolumn{1}{l|}{1.16E-07} & 1.03E-05                     & \multicolumn{1}{l|}{3.03E-08} & 5.08E-06                     & \multicolumn{1}{l|}{2.24E-08} & 7.46E-06                     \\ \hline
\multirow{6}{*}{$\lambda=1$}     & 20  & \multicolumn{1}{l|}{3.15E-02} & 6.51E-01                     & \multicolumn{1}{l|}{1.03E-02} & 3.87E-01                     & \multicolumn{1}{l|}{1.66E-02} & 1.16E-00                     \\ \cline{2-8} 
                                 & 40  & \multicolumn{1}{l|}{1.46E-03} & 5.05E-02                     & \multicolumn{1}{l|}{2.63E-04} & 1.78E-02                     & \multicolumn{1}{l|}{5.20E-04} & 6.63E-02                     \\ \cline{2-8} 
                                 & 80  & \multicolumn{1}{l|}{1.23E-05} & 6.84E-04                     & \multicolumn{1}{l|}{4.49E-06} & 5.37E-04                     & \multicolumn{1}{l|}{2.86E-06} & 7.14E-04                     \\ \cline{2-8} 
                                 & 160 & \multicolumn{1}{l|}{2.16E-07} & 1.70E-05                     & \multicolumn{1}{l|}{2.12E-08} & 3.88E-06                     & \multicolumn{1}{l|}{1.25E-08} & 4.15E-06                     \\ \cline{2-8} 
                                 & 320 & \multicolumn{1}{l|}{1.30E-07} & 9.77E-06                     & \multicolumn{1}{l|}{1.90E-08} & 3.38E-06                     & \multicolumn{1}{l|}{9.82E-09} & 3.35E-06                     \\ \cline{2-8} 
                                 & 640 & \multicolumn{1}{l|}{1.49E-07} & 1.10E-05                     & \multicolumn{1}{l|}{2.62E-08} & 4.51E-06                     & \multicolumn{1}{l|}{1.23E-08} & 4.31E-06                     \\ \hline
\end{tabular}
\caption{Errors of the space-time LRNN-$C^0$DG method for heat equation when $t=1$ in Example \ref{2dheat}.}
\label{table2dheatlrnnc0dg}
\end{table}

\begin{table}[h]
	\centering
\begin{tabular}{|l|l|ll|ll|ll|}
\hline
\multirow{2}{*}{}                & $h$   & \multicolumn{2}{c|}{$2^{-1}$}                                & \multicolumn{2}{c|}{$2^{-2}$}                                & \multicolumn{2}{c|}{$2^{-3}$}                                \\ \cline{2-8} 
                                 & \diagbox[width=5em,trim=l]{${\rm DoF}_\sigma$}{Norm} & \multicolumn{1}{c|}{$L^{2}$}  & \multicolumn{1}{c|}{$H^{1}$} & \multicolumn{1}{c|}{$L^{2}$}  & \multicolumn{1}{c|}{$H^{1}$} & \multicolumn{1}{c|}{$L^{2}$}  & \multicolumn{1}{c|}{$H^{1}$} \\ \hline
\multirow{6}{*}{$\lambda=0.001$} & 20  & \multicolumn{1}{l|}{6.05E-02} & 9.28E-01                     & \multicolumn{1}{l|}{4.22E-02} & 1.06E-00                     & \multicolumn{1}{l|}{6.54E-02} & 2.73E-00                     \\ \cline{2-8} 
                                 & 40  & \multicolumn{1}{l|}{3.05E-03} & 8.06E-02                     & \multicolumn{1}{l|}{1.21E-03} & 5.39E-02                     & \multicolumn{1}{l|}{1.20E-03} & 1.03E-01                     \\ \cline{2-8} 
                                 & 80  & \multicolumn{1}{l|}{2.39E-05} & 9.89E-04                     & \multicolumn{1}{l|}{8.41E-06} & 6.43E-04                     & \multicolumn{1}{l|}{5.27E-06} & 7.91E-04                     \\ \cline{2-8} 
                                 & 160 & \multicolumn{1}{l|}{5.87E-07} & 3.23E-05                     & \multicolumn{1}{l|}{1.03E-07} & 1.02E-05                     & \multicolumn{1}{l|}{6.13E-08} & 1.24E-05                     \\ \cline{2-8} 
                                 & 320 & \multicolumn{1}{l|}{2.65E-07} & 1.50E-05                     & \multicolumn{1}{l|}{7.48E-08} & 8.36E-06                     & \multicolumn{1}{l|}{3.96E-08} & 8.99E-06                     \\ \cline{2-8} 
                                 & 640 & \multicolumn{1}{l|}{1.54E-07} & 9.48E-06                     & \multicolumn{1}{l|}{7.24E-08} & 8.51E-06                     & \multicolumn{1}{l|}{5.21E-08} & 1.22E-05                     \\ \hline
\multirow{6}{*}{$\lambda=1$}     & 20  & \multicolumn{1}{l|}{4.36E-02} & 8.14E-01                     & \multicolumn{1}{l|}{1.71E-02} & 5.36E-01                     & \multicolumn{1}{l|}{3.65E-02} & 1.87E-00                     \\ \cline{2-8} 
                                 & 40  & \multicolumn{1}{l|}{2.26E-03} & 6.79E-02                     & \multicolumn{1}{l|}{7.93E-04} & 4.46E-02                     & \multicolumn{1}{l|}{1.24E-03} & 1.24E-01                     \\ \cline{2-8} 
                                 & 80  & \multicolumn{1}{l|}{2.05E-05} & 1.12E-03                     & \multicolumn{1}{l|}{1.03E-05} & 1.09E-03                     & \multicolumn{1}{l|}{7.41E-06} & 1.57E-03                     \\ \cline{2-8} 
                                 & 160 & \multicolumn{1}{l|}{2.00E-07} & 1.64E-05                     & \multicolumn{1}{l|}{3.28E-08} & 5.33E-06                     & \multicolumn{1}{l|}{2.70E-08} & 7.33E-06                     \\ \cline{2-8} 
                                 & 320 & \multicolumn{1}{l|}{1.14E-07} & 8.23E-06                     & \multicolumn{1}{l|}{2.23E-08} & 4.12E-06                     & \multicolumn{1}{l|}{1.54E-08} & 5.19E-06                     \\ \cline{2-8} 
                                 & 640 & \multicolumn{1}{l|}{7.08E-08} & 6.86E-06                     & \multicolumn{1}{l|}{4.14E-08} & 7.21E-06                     & \multicolumn{1}{l|}{2.42E-08} & 8.28E-06                     \\ \hline
\end{tabular}
\caption{Errors of the space-time LRNN-$C^1$DG method for heat equation when $t=1$ in Example \ref{2dheat}.}
\label{table2dheatlrnnc1dg}
\end{table}

\begin{figure}[htb]   
    \centering

\begin{subfigure}{0.4\textwidth}
            \centering           
            \includegraphics[width=0.8\textwidth,height=0.24\textheight]{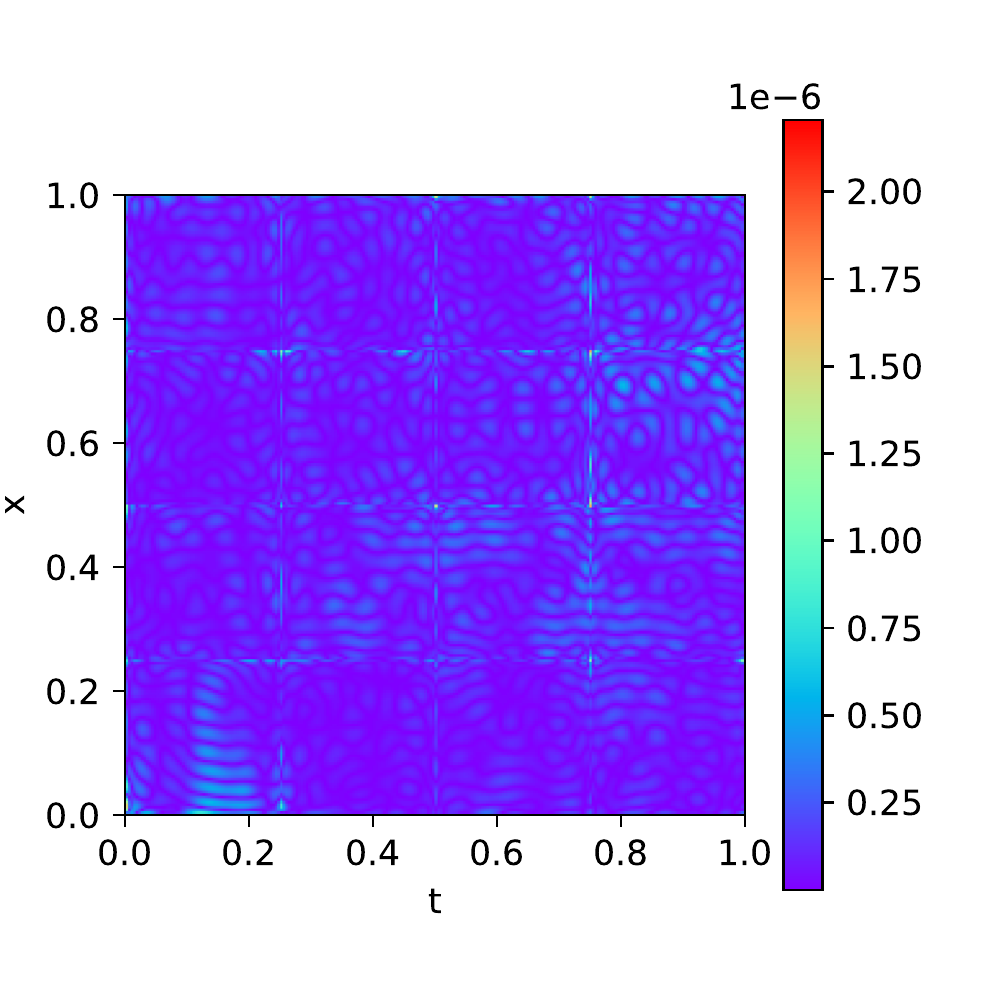}   
            \caption{The absolute error for LRNN-DG method.}
            \label{fig:a}
    \end{subfigure}
    \begin{subfigure}{0.4\textwidth}
            \centering       
            \includegraphics[width=0.8\textwidth,height=0.24\textheight]{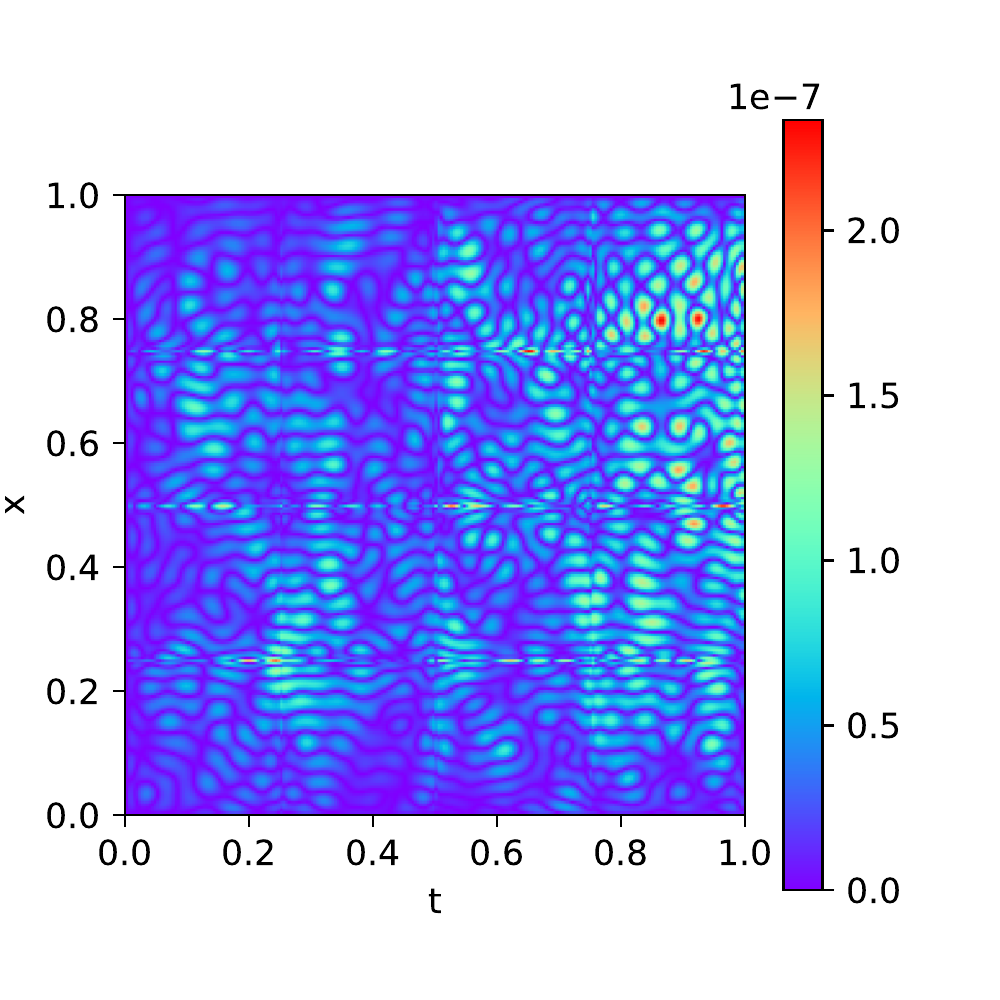}   
		 \caption{The absolute error for LRNN-$C^0$DG method.}
            \label{fig:b}
        \end{subfigure}  \\
        
\begin{subfigure}{0.4\textwidth}
            \centering           
            \includegraphics[width=0.8\textwidth,height=0.24\textheight]{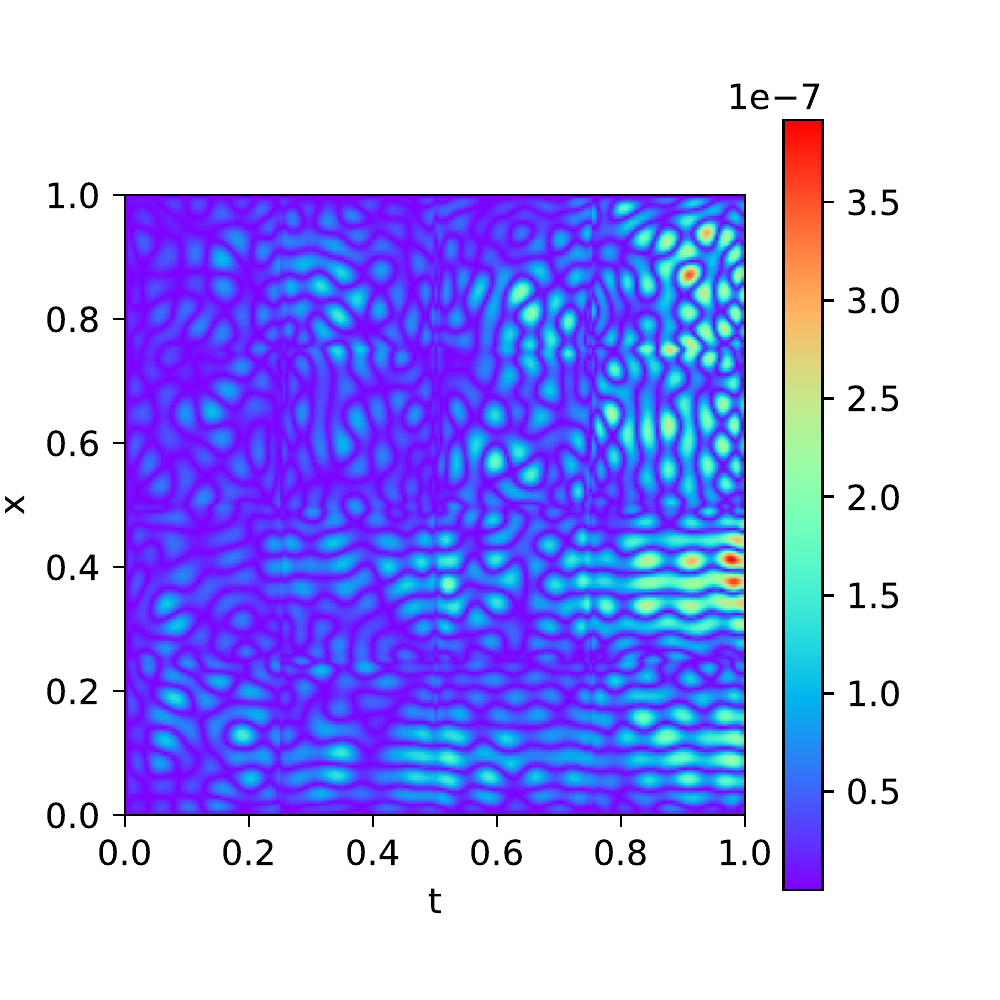}   
            \caption{The absolute error for LRNN-$C^1$DG method.}
            \label{fig:c}
    \end{subfigure} 
\begin{subfigure}{0.4\textwidth}
            \centering           
            \includegraphics[width=0.8\textwidth,height=0.24\textheight]{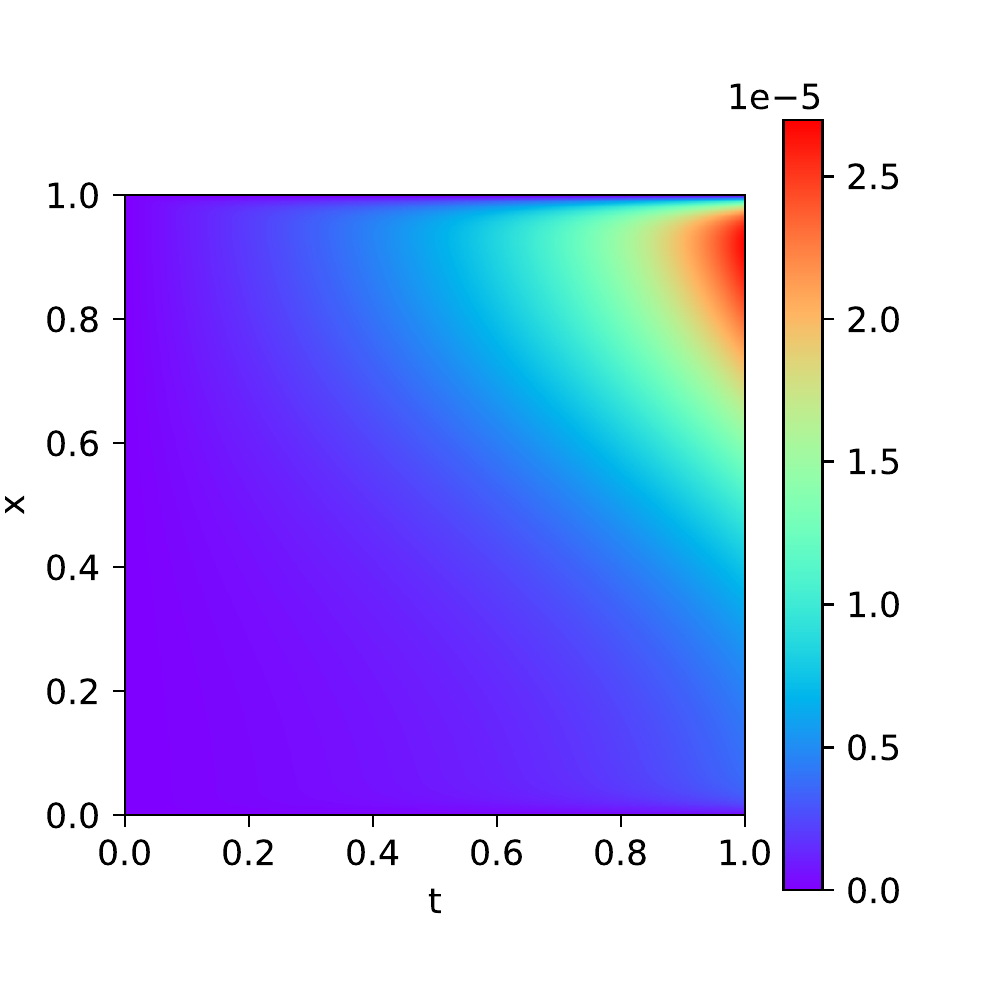}   
            \caption{The absolute error for FEM with back Euler.}
            \label{fig:d}
    \end{subfigure}
\caption{Absolute errors computed by proposed methods and FEM for $\lambda=0.001$ in Example \ref{2dheat}.} 
        \label{figure2dheaterror}
\end{figure}

Let us consider the space-time LRNN with DG methods for solving the heat equation.
The numerical errors measured in the $L^2$ norm and the $H^1$ seminorm at $t=1$ for the different numbers of degrees of freedom are shown in Tables \ref{table2dheatlrnndg}, Tables \ref{table2dheatlrnnc0dg} and Tables \ref{table2dheatlrnnc1dg}. Here, we partition the domain $I\times\Omega$ into some non-overlapping subdomains $\sigma=I_i\times K$ where $I_i$ is a time interval, $K$ is a space interval, and the time interval and space interval have the same size $h$. The number of collection points on each edge used in LRNN-$C^0$DG and LRNN-$C^1$DG is 70.

Table \ref{table2dheatlrnndg} records the $L^2$ errors and the $H^1$ errors of the space-time LRNN-DG at $t=1$ in terms of the number of degrees of freedom in each element and the size of the element. In this group of tests, the penalty parameter is chosen as $\eta$ = 10, 8 for $\lambda$ = 0.001, 1, respectively. The weight/bias coefficients in the hidden layer of each local network are set to be uniform random values generated from $[-1.5, 1.5]$ for both $\lambda= 0.001$ and $\lambda = 1$.
Table \ref{table2dheatlrnnc0dg} records errors of the space-time LRNN-$C^0$DG at $t=1$. The parameter $w_0$ is set to be 1 and 1.25 for $\lambda=0.001$ and $\lambda=1$ respectively. 
Table \ref{table2dheatlrnnc1dg} records errors of the space-time LRNN-$C^1$DG for $t=1$, and the parameter $w_0$ is chosen as 1.1 for both $\lambda$ = 0.001 and $\lambda$ = 1.

For a comparison with the current methods, we have also solved this problem
  using the traditional finite element method. With FEM, we employ the $P_2$ finite element for spatial discretization and the backward Euler scheme for time stepping.  The numerical errors are shown in Table \ref{table2dheatfem} for $t=1$ and $\lambda=0.001$. Here, $h$ is the size of the mesh and $\Delta t$ is the size of the time step. By comparing the FEM data in this table and those of the current methods, we can see that space-time LRNN-DG methods can achieve more accurate numerical solutions. 

\begin{table}[h]
\centering
\begin{tabular}{|c|c|c|c|c|c|}
\hline
\diagbox[width=5em,trim=l]{Norm}{$h$, $\Delta t$}  & {$2^{-5}$, $2^{-10}$} & {$2^{-6}$, $2^{-12}$} & {$2^{-7}$, $2^{-14}$} & {$2^{-8}$, $2^{-16}$} & {$2^{-9}$, $2^{-18}$} \\ \hline
$L^2$      & 3.82E-03            & 9.56E-04           & 2.39E-04           & 5.97E-05           & 1.49E-05           \\ \hline
$H^1$      & 3.84E-02            & 9.65E-03           & 2.41E-03           & 6.04E-04           & 1.51E-04           \\ \hline
\end{tabular}
\caption{Errors of the $P_2$ FEM with back Euler for heat equation when $\lambda=0.001$ and $t=1$ in Example \ref{2dheat}.}
\label{table2dheatfem}
\end{table}

Figure \ref{figure2dheaterror} shows distributions of the point-wise absolute errors in the spatial-temporal domain obtained using the current methods and the traditional FEM for the case $\lambda=0.001$. 
  In the space-time LRNN with DG methods, the size of each element $h = 2^{-2}$ and the number of ${\rm DoF}_\sigma$ in each element is 320. In the $P_2$ FEM with the backward Euler scheme, the size of mesh $h=2^{-9}$ and the size of the time step $\Delta t = 2^{-18}$. We can see the absolute error of the LRNN-DG is larger but smoother. The absolute error of LRNN-$C^0$DG and LRNN-$C^1$DG are smaller but have strong amplitudes. Due to the time marching, the error accumulation over time is evident from the FEM distribution. Unlike the traditional methods, there is little or essentially no error accumulation  in the proposed space-time LRNN with DG methods.

\section{Summary}\label{sec:summary}
Local randomized neural networks with discontinuous Galerkin formulations provide a new framework for solving partial differential equations. With the decomposition of the domain, we use LRNNs to approximate the solution of partial differential equations on each subdomain, and apply the IPDG scheme to couple these LRNNs together. Then we obtain the weights of output layers by the least-squares method. Under certain proper assumptions, we prove the convergence of the proposed methods. Furthermore, we propose space-time LRNN-DG methods for solving the heat equation. The proposed methods have the following advantages: (i) the accuracy of the proposed methods is better than the FEM or the usual DG method with respect to the number of degrees of freedom; (ii) compared to DG methods, the LRNN-$C^0$DG method and LRNN-$C^1$DG method are free of penalty parameters; (iii) the proposed methods can solve time-dependent problems by the space-time approach naturally and efficiently.

We believe that the proposed methods have a great potential for solving partial differential equations. However, many aspects of these methods need to be further investigated. We have considered only linear partial differential equations in this paper. The extension of these methods to nonlinear problems is one issue we would like to address in a future study. Numerical examples indicate that the errors of these methods seem stagnant when the number of degrees of freedom reaches a certain threshold. Can one design neural networks to avoid such a situation? Can we exploit parallel processing to improve their efficiencies? How can we use mesh adaptation to improve their performance for complex problems? How do we derive the error estimates for the proposed methods? These are some of the outstanding questions that remain to be explored.

\vspace{5mm}

\noindent{\bf Acknowledgement.} The authors are grateful to Professor Zongben Xu for the valuable discussions and suggestions.

\end{document}